\documentclass[10pt]{article}

\pdfoutput=1

\usepackage[latin1]{inputenc}
\usepackage{mystyle}
\usepackage{dsfont}
\usepackage{url}
\usepackage[normalem]{ulem}
\usepackage{fullpage}
\usepackage{multirow}

\graphicspath{{./figures/}}

\title{Super-resolved Lasso}

\author{%
Clarice Poon\footnote{Mathematical Institute,  Zeeman Building, University of Warwick, Coventry CV4 7AL, UK, \texttt{clarice.poon@warwick.ac.uk}}, \quad%
Gabriel Peyr\'e\footnote{CNRS and DMA, PSL University, Ecole Normale Sup\'erieure, 45 rue d'Ulm, F-75230 PARIS cedex 05, FRANCE, \texttt{gabriel.peyre@ens.fr} }%
}

\newcommand{\x}{x}

\date{\today}

\begin{document}
\maketitle


\begin{abstract}
Super-resolution of pointwise sources is of utmost importance in various areas of imaging sciences. Specific instances of this problem arise in single molecule fluorescence, spike sorting in neuroscience, astrophysical imaging, radar imaging, and nuclear resonance imaging.
In all these applications, the Lasso method (also known as Basis Pursuit or \( \ell^1 \)-regularization) is the de facto baseline method for recovering sparse vectors from low-resolution measurements. This approach requires discretization of the domain, which leads to quantization artifacts and consequently, an overestimation of the number of sources.
While grid-less methods, such as Prony-type methods or non-convex optimization over the source position, can mitigate this, the Lasso remains a strong baseline due to its versatility and simplicity.
In this work, we introduce a simple extension of the Lasso, termed ``super-resolved Lasso" (SR-Lasso). Inspired by the Continuous Basis Pursuit (C-BP) method, our approach introduces an extra parameter to account for the shift of the sources between grid locations. Our method is more comprehensive than C-BP, accommodating both arbitrary real-valued or complex-valued sources. Furthermore, it can be solved similarly to the Lasso as it boils down to solving a group-Lasso problem.
A notable advantage of SR-Lasso is its theoretical properties, akin to grid-less methods. Given a separation condition on the sources and a restriction on the shift magnitude outside the grid, SR-Lasso precisely estimates the correct number of sources.
\end{abstract}


\section{Introduction}

Sparse-spike super-resolution refers to the problem of recovering point-wise sources from low-resolution linear measurements, which are essentially linear superpositions of patterns translated to the source locations. These measurements can for instance take the form of Fourier coefficients, convolutions, or Laplace transforms.
Typical examples of such an ill-posed inverse problems arose in single molecule fluorescence imaging~\cite{rust2006sub,hell2003toward}. For this type of imaging processes, the final high-resolution image (e.g., of a cell) is produced by combining multiple snapshots. Each snapshot consists of a sparse set of fluorescence molecules, which are treated as point-wise sources.
Other applications include astrophysics~\cite{puschmann2005super} (for estimating star locations), radar array imaging~\cite{krim1996two}, seismic imaging~\cite{khaidukov2004diffraction} (used to recover sharp localized transitions in the ground), Nuclear Magnetic Resonance spectroscopy~\cite{mobli2014nonuniform} (for determining molecular structures), and intracellular neuron recording~\cite{lewicki1998review} (to sort spikes in time series).
To address these inverse problems, a common approach involves gridding the space and applying a sparse regularization principle, which promotes the appearance of spikes in the solution. The regularizer of choice is often the \( \ell^1 \) norm, resulting in a convex optimization problem. Recently, there has been an inclination to eliminate the grid to achieve greater accuracy. However, this introduces computational challenges, given that the optimization problem becomes infinite-dimensional. This paper proposes a middle ground, where spikes between grid points are parameterized in a convex manner, aiming to capture the advantages of both approaches.

\subsection{Previous Works}
\label{eq:previous}

\paragraph{Lasso methods.}

\( \ell^1 \) regularization, often referred to as the Lasso problem in statistics~\cite{tibshirani1996regression} or basis pursuit in signal processing~\cite{chen2001atomic}, has its roots in the seismic inversion community~\cite{claerbout1973robust}. It serves as a simple but powerful convex regularizer inducing sparsity and can be readily applied to the super-resolution problem once the spatial domain is discretized.
Among the numerous extensions of the basic Lasso, the group Lasso~\cite{yuan2006model} is of primary importance for this work. It enables the recovery of multiple channels in signals or images (e.g., in multi-spectral imaging) while enforcing the sparsity patterns of the different channels to coincide. Our SR-Lasso method necessitates the resolution of a group Lasso problem.
Although the main focus of this paper is not the development of new sparse optimization algorithms (as our method can employ any off-the-shelf state-of-the-art group Lasso solver), it is worth noting that there have been numerous improvements in state-of-the-art solvers. One straightforward yet effective method is the iterative soft thresholding algorithm~\cite{daubechies2004iterative}, rooted in non-smooth first-order optimization techniques. A comprehensive review can be found in~\cite{bach2012optimization}.
Greedy-type algorithms, which progressively increase the number of sources, are beneficial for highly sparse solutions. Simpler algorithms include homotopy-type methods or LARS, which track the regularization path~\cite{friedman2010regularization}. More efficient solvers encompass safe screening rules and pruning techniques~\cite{ndiaye2017gap,massias2018celer}.
Another category of solvers includes iterative least squares~\cite{daubechies2010iteratively}. These can be enhanced using an Hadamard re-parameterization~\cite{hoff2017lasso}. Such re-parameterization can further benefit from the VarPro marginalization technique~\cite{poon2023smooth}, which is the solver we employ in this paper.

\paragraph{Off the grid and BLASSO.}

Recently, there has been a surge of interest in avoiding the explicit discretization of space, allowing for the optimization of source positions. This approach can be mathematically characterized as an infinite-dimensional convex problem over the space of measures \cite{de2012exact,bredies2013inverse}, commonly termed the ``Beurling Lasso'' (BLASSO). The thought after sparse solution should then be a sum of Dirac masses.
While this ``off-the-grid'' method is attractive in terms of precision, it presents computational challenges. Optimizing the source locations, which can be likened to the computation of an adaptive grid, is inherently non-convex, making it difficult.
A prominent category of optimization techniques for this problem draws from generalizations of the Frank-Wolfe optimization method applied over the space of measures~\cite{bredies2013inverse}. These techniques often alternate between a convex optimization step on a (desirably small) grid and a non-convex optimization for new grid points~\cite{boyd2017alternating,denoyelle2019sliding,flinth2021linear}. In practical applications, they appear to be computationally superior to an a priori discretization of space. However, they are more complex to implement and lack robust theoretical backing due to the integration of non-convex optimization steps, despite enjoying favorable properties close to the minimizers~\cite{traonmilin2023basins}.

\paragraph{SOS, Prony and deep-learning methods.}

Another class of techniques leverages sum-of-squares (SOS) relaxation to transform this infinite-dimensional optimization problem into a semi-definite program~\cite{de2012exact,candes2014towards}. A significant limitation of these methods is that they are restricted to Fourier-type (translation invariant) operators and, more generally, to semi-algebraic (e.g., polynomial) operators~\cite{de2016exact}. Moreover, the resulting SDP problem's dimensionality grows at least quadratically with the number of measured frequencies.
These SOS spectral methods bear a close resemblance to Prony-type methods. While they don't minimize the \( \ell^1 \) norm of the solution, they can also identify localized sources from Fourier frequencies. Notable methods in this category include MUSIC~\cite{schmidt1986multiple,odendaal1994two,liao2016music}, ESPRIT~\cite{roy1989esprit}, and matrix pencil~\cite{hua1990matrix}. Originally developed for signal processing, these methods can be extended to higher dimensions~\cite{kunis2016multivariate}. Let us also notice that their super-resolution properties as sources nearly collide can be precisely analyzed~\cite{batenkov2021super,li2021stable}. 
A final category of methods, which also bypasses explicit spatial gridding, relies on deep networks to directly estimate source positions~\cite{boyd2018deeploco}. One challenge is the need for extensive pre-training using pairs of simulated ground truth and observations. Additionally, it is imperative to possess detailed knowledge of the signal generation process and the structure of the input sources. Such methods are gaining traction, especially in single-molecule microscopy, where in-depth understanding of the image formation process is accessible~\cite{speiser2021deep,chen2023single,hyun2022development,nehme2018deep}.

\paragraph{Super-resolution theory.}

A fundamental theoretical question is to determine the capability of these sparse estimation techniques to either exactly (in noise-free conditions) or approximately (in noisy conditions) recover source locations~\cite{donoho1992superresolution}. This necessitates constraints on the relative positioning of the sources. Indeed, closely positioned sources might be irretrievable, or the signal-to-noise ratio could explodes as the sources collapse.
Generic recovery conditions (not necessarily related to super-resolution problems) have been thoroughly examined for the Lasso technique~\cite{fuchs2004sparse,zhao2006model,grasmair2011necessary}. These conditions have been expanded for the group Lasso~\cite{bach2008consistency,liu2009estimation} and other general sparsity-promoting regularizers~\cite{vaiter2017model}. 
We should highlight that in Section~\ref{sec:sparsistency-group-lasso} we propose a more detailed analysis of the group Lasso setting, which necessitates a milder injectivity condition.

Focusing on the specific instance of super-resolution with Fourier measurements (analogous to translation-invariant linear operators), \cite{candes2014towards} demonstrated that exact recovery is possible if sources are separated by a distance greater than a ``Rayleigh limit'', which is proportional to the kernel's width. This improves over previous works of~\cite{dossal2005sparse}. Notably, this outcome is applicable ``off-the-grid'', meaning it is relevant to the BLASSO problem.
This finding was further extended to accommodate noisy data, allowing for approximate recovery results~\cite{candes2013super}, as well as the exact support recovery (exact estimation of the number of spikes)~\cite{duval2015exact}. it is also possible to incorporate more general operators (not strictly translation-invariant)~\cite{poon2023geometry}, and to under-sample measurements to implement compressed sensing~\cite{tang2013compressed,poon2023geometry}.

\paragraph{Discretization and continuous Basis Pursuit.}

Although off-the-grid theoretical results are compelling, the prevalent approaches still involve discretizing the computational domain. This serves as an approximation to the original infinite-dimensional problem~\cite{tang2013sparse}. This approximation introduces discretization artifacts, which can be reduced using adaptive refinement strategies~\cite{flinth2023grid}. Often, these artifacts result in a single source being spread across multiple grid points~\cite{duval2017sparse}, hindering exact support recovery.
To somewhat mitigate these quantization errors, the Continuous Basis-Pursuit (C-BP), as detailed in~\cite{ekanadham2011recovery,ekanadham2014unified}, introduces additional parameters to accommodate for source shifts outside the grid point locations. 
A theoretical examination of this technique~\cite{duval2017sparse} reveals a more accurate source estimation. However, C-BP is not immune to quantization errors and does not guarantee recovery of the exact number of sources, even if they are distinctly separated. Our SR-Lasso approach, which draws significant inspiration from the C-BP procedure, exhibits superior support recovery properties.


\subsection{Contributions}

Our first major contribution is methodological: Section~\ref{sec:sr-lasso} introduces the Super-resolved Lasso (SR-Lasso), a new method capable of locating sources outside the grid using a Taylor expansion. Central to our approach is the strategy of coupling the source amplitude with the  off-the-grid's shift by imposing a shared support, which is achieved using a group Lasso regularization. This methods differs from the C-BP approach, which relies on a hard constraint. This enable the use of off-the-shelf group Lasso solvers, offers enhanced theoretical guarantees, and can cope with signed measure, unlike the C-BP which limits itself to positive measures.
Section~\ref{sec:sparsistency-group-lasso} present a thorough sparsistency (i.e., exact support recovery) analysis of the group Lasso, which is of independent interest. The main result is Theorem~\ref{thm:group_IFT}, which refines existing analyses like those in~\cite{bach2008consistency,liu2009estimation} by demanding a less stringent restricted injectivity condition~\eqref{eq:cond-ic}.
Section~\ref{sec:discretized-theory} instantiates in Theorem~\ref{thm:nondegen} these results in the case where this group Lasso problem is a discretization of some continuous (``off-the grid'') operator. The section exposes the relationship between kernel properties and the separation distance between spikes, which leads to sufficient conditions for exact support recovery.
Section~\ref{sec:ti-kernel} proposes a detailed analysis of the SR-Lasso method, focussing on translation-invariant kernels for clarity. By drawing from prior sections, Theorem~\ref{thm_G} shows that -- provided the sources are sufficiently apart and the off-the-grid shifts remain small enough -- SR-Lasso can reliably determine the sources' support. This is the first result of this kind, i.e. the first finite dimensional Lasso-type algorithm which can achieve a consistent support estimation. 
Section~\ref{sec:applis} exemplify the usefulness of the SR-Lasso by contrasting it with Lasso and C-BP on synthetic examples inspired by applications in imaging sciences. Specifically, it underscores using several evaluation metrics the SR-Lasso's capacity to consistently estimate the sought-after sources' support.
The code to reproduce the results of this paper is available online\footnote{\url{https://github.com/gpeyre/SuperResolvedLasso}}.

For simplicity, throughout, we assume that all the sought-after vectors and measures are real-valued (although extensions to the complex setting is possible).


\section{Super-resolved Lasso}
\label{sec:sr-lasso}

\subsection{BLASSO and Lasso}
\label{subsec:blasso-lasso}

We consider linear inverse problems over the space of measures. The linear operator $\Phi$ to invert maps a measure $\mu \in \Mm(\RR^d)$ over a space of dimension $d$ to measurements $\Phi \mu \in \Hh$ where $\Hh$ is some Hilbert space with norm $\norm{\cdot}_\Hh$. Under mild continuity condition, such an operator can be written in integral form as 
\begin{equation}\label{eq:fwd-model}
	\Phi \mu \eqdef \int \phi(x) \mathrm{d}\mu(x)
\end{equation}
where $\phi:\RR \to \Hh$ is a continuous function taking values in $\Hh$.
We assume that $\phi$ is normalized such that $\norm{\phi(x)}_\Hh=1$ for all $x$.
Typical examples includes Fourier measurement where $\phi(x)=( e^{\imath \dotp{x}{\omega_k}} )_{k}$ for some frequency set $(\om_k)_k$,  (sampled) convolution, where $\phi(x)=(k(x-x_k))_{k}$ for some sample location $x_k$, and Laplace transform when $\phi(x) =( e^{\dotp{x}{\omega_k}} )_{k}$. 

Given some measurements $y \in \Hh$, the BLASSO problem~\cite{de2012exact} aims at minimizing a weighted sum of a fidelity criterion using the Euclidean norm $\norm{\cdot}_\Hh$ and a sparsity promoting penalty
\begin{equation}\label{eq:blasso}
	\umin{\mu} \frac{1}{2\lambda} \norm{\Phi \mu -y}_\Hh^2 + |\mu|(\RR^d).
\end{equation}
The parameter $\la>0$ should be tuned to the level of noise contaminating the measurements $y$.
Here $|\mu|(\RR^d)$ is the total mass of the absolute value of the measure, which is often called ``total variation norm'' in probability. If $\mu=\sum_i a_j \de_{z_i}$ is a discrete measure (weighted sum of Dirac masses), then 
$$
	|\mu|(\RR^d) = \sum_j |a_j| = \norm{a}_1
$$
boils down to the $\ell^1$ norm of the vector $a$.

Problem~\eqref{eq:blasso} is infinite dimensional, and we refer to Section~\ref{eq:previous} for a review of possible computational approaches. The simplest way to approximate it is to discretize the problem on a grid $X = \ens{x_j}_{j=1}^n\subset \RR$ and impose that the measure is supported on $X$, so that $\mu = \sum_i a_i \de_{x_i}$. In this case, optimizing only on $a$ replaces the infinite dimensional convex problem~\eqref{eq:blasso} by a standard finite dimensional Lasso problem
\begin{equation}\label{eq:lasso}
	\umin{a \in \RR_n} \frac{1}{2\lambda} \norm{\Phi_X a -y}_\Hh^2 + \norm{a}_1
\end{equation}
where the finite dimensional linear operator reads
\begin{equation}
	\Phi_X a \eqdef \sum_{j=1}^n a_j \phi(x_j) \in \Hh. 
\end{equation}

\subsection{SR-Lasso}

To improve the quality of the discretization of \eqref{eq:blasso}, we follow the \emph{continuous basis pursuit} (C-BP) rationale~\cite{ekanadham2011recovery,ekanadham2014unified} and perform a Taylor expansion of the operator $\Phi$ around the grid positions $X$. For the sake of simplicity, we derive the expression here in 1-D and assume that the grid spacing is $h$, i.e. $\abs{x_j-x_{j+1}} = h$. The extension to higher dimension is discussed in Section~\ref{sec:multidim}. 

We consider a discrete measure $\mu = \sum_{j=1}^m a_j \delta_{z_j}$. 
We denote $\ens{z_j}_{j=1}^m \eqdef \ens{x_i + t_i}_{i\in\Ii}$ where $\Ii \subset [n]$ is an index set of size $m$ and $t_i\in [-h/2,h/2]$. 
This means that $(t_i)_i$ are shift variables with respect to the nearest grid point.
Then, by Taylor expansion (after possibly relabelling the indexing on $a_j$'s),
\begin{equation}\label{eq:taylorexp}
	\sum_{j=1}^m a_j \phi(z_j) = \sum_{j\in\Ii} a_j \phi(x_j) + \sum_{j\in\Ii}  a_j t_j \phi'(x_j) + w,
\end{equation}
where the remainder term is defined with points $\xi_j$ between $z_j$ and $x_j$ as
\begin{equation}\label{eq:remainder}
	w = \frac12\sum_{j\in\Ii}  a_j t_j^2 \phi''(\xi_j) = \Oo(h^2).
\end{equation}

The SR-Lasso problem with parameter $\tau>0$ is obtained by ignoring this remainder term, performing the change of variable 
$b_j \eqdef a_j t_j \norm{\phi'(x_j)}/\tau$ and penalizing together both $a$ and $b$ by a group sparsity-promoting criterion.
It reads
\begin{equation}\label{eq:c-BP-unconstr}
	\min_{a,b} \frac12 \norm{\Phi_X a + \tau \Psi_X b -y}^2 + \lambda \norm{(a,b)}_{1,2} \tag{$\Pp_\lambda(y)$}
\end{equation}
where
$$
	\norm{(a,b)}_{1,2} \eqdef \sum_j \sqrt{\abs{a_j}^2 + \abs{b_j}^2}
$$
and 
\begin{equation}\label{eq:normalizing-psi}
	\Psi_X b \eqdef \sum_{j=1}^n b_j \psi(x_j)\qwhereq \psi(x) \eqdef \frac{\phi'(x)}{\norm{\phi'(x)}}.
\end{equation}
Note that $\Psi$ is the \emph{normalised} derivative operator corresponding to $\Phi$. 
With respect to the initial Lasso problem~\eqref{eq:lasso}, the SR-Lasso problem~\eqref{eq:c-BP-unconstr} introduce an extra shift parameter $b$ to estimate.
The crux of this SR-Lasso approach is that the support of this extra shift is however not free: the use of the $\norm{(a,b)}_{1,2}$ group norm forces the support of $a$ and $b$ to be the same. This is a convex way to impose the constraint that the two variable $(a,b)$ comes from a discretization of the same input measure $\mu$.

\paragraph{Link with group lasso.}

Observe that by defining the matrix 
$$ \Gamma \eqdef \begin{pmatrix}
	\Phi_X & \tau \Psi_X
\end{pmatrix} : \RR^{2 n} \to \Hh
$$
then~\eqref{eq:c-BP-unconstr} is a special case of a group Lasso problem over $z = \binom{a}{b} \in \RR^{2n}$, with groups of size $2$
  \begin{equation}\label{eq:group_cbp}
\min_{c} \frac12 \norm{\Gamma z - y}^2 + \lambda \norm{z}_{1,2}.
\end{equation}
In the following, given an index set $\Ii$, denote 
$$
\Gamma_\Ii \eqdef \begin{pmatrix}
(\phi(x_j))_{j\in\Ii} &(\tau \psi(x_j))_{j\in\Ii} 
\end{pmatrix},
$$
the restriction of $\Gamma$ to columns corresponding to $x_j$ for $j\in \Ii$.

\paragraph{Recovery of a BLASSO approximate solution. }
If $(a,b)$ solve \eqref{eq:c-BP-unconstr}, then, we expect that the ``shift'' $t_j$ is such that $b_j = a_j t_j \norm{\phi'(x_j)}/\tau$. So, provided that   $\abs{b_j}\leq \abs{a_j} h \norm{\phi'(x_j)}/\tau$,  the recovered measure is taken to be
\begin{equation}\label{eq:recovering-measure}
	\mu = \sum_j a_j \delta_{x_j + t_j}
	\qwhereq
	t_j \eqdef \tau \frac{b_j}{a_j\norm{\phi'(x_j)}}.
\end{equation}
Our main theoretical contributions assess the quality of this SR-Lasso method (combining together the resolution of~\eqref{eq:c-BP-unconstr} and the computation~\eqref{eq:recovering-measure} of the solution measure) to recover a given sparse measure from noisy measurements $y$.
%

\newcommand{\tauconst}{c}
\paragraph{Choice of the parameter $\tau$. } 
The parameter $\tau$ controls how far Dirac masses in the solution can ``move'' inside the grid. 
The standard Lasso problem is obtained by setting $\tau=0$ (the solution is constrained on the grid).   
Our theoretical results will cover the case of $\tau \in (0,1)$ and in practice, we choose $\tau$ close to 1.  

\paragraph{The normalization of $\Psi$.}
In the special case where $\phi$ corresponds to a translation invariant problem, i.e. when $\dotp{\phi(x)}{\phi(x')} = \kappa(x-x')$ for some kernel function $\kappa$, the normalisation appearing with the definition of $\Psi$ is constant for all $x$ since
$$
\forall x, \quad \norm{\phi'(x)}^2 = -\kappa''(0)
$$
and hence, $\psi(x) = \abs{\kappa''(0)}^{-\frac12}\phi'(x)$. 
In the case of non-translation invariant operators, the normalisation is not a constant and is important both in theory and practice to ensure a correct weighting between position and derivatives. 
Note that in the multivariate setting, the normalisation in $\Psi$ should be done with a matrix scaling (see next paragraph).

\subsection{Multivariate setting}
\label{sec:multidim}

\newcommand{\metric}{\mathfrak{g}}
For simplicity, we analyse the univariate case in this paper. However, the ideas also extend to the multivariate setting where $\phi:\RR^d \to \Hh$.
Define the metric $\metric_x \eqdef \nabla\phi(x) \nabla\phi(x)^*$, the multi-dimensional SR-Lasso is then 
\begin{equation}
\min_{a,b} \frac12 \norm{\Phi_X a +  \Psi_X b -y}^2 + \lambda \sum_j \sqrt{a_j^2 + \norm{b_j}^2} 
\end{equation}
where $a = (a_j)_{j=1}^n\in \CC^n$ and $b = (b_j)_{j=1}^n \in (\CC^d)^n$, 
\begin{equation}\label{eq:normmultivar}
\Phi_X a = \sum_j a_j \phi(x_j)\qandq \Psi_X b = \sum_j \dotp{\tau \odot b_j}{ \metric_x^{-\frac12} \nabla \phi(x_j)}.
\end{equation}
 for some $\tau \in \RR^{d}_+$.
 The use of the metric $\metric_x$ is for normalisation purposes and the $\tau$ vector is a weighting between the different directions.
 Note that for $K(x,x') =  \dotp{\phi(x)}{\phi(x')} $,
 $$
K(x,x)= 1 \implies  -\partial_1^2 K(x,x) = \partial_1 \partial_2 K(x,x)= \nabla \phi(x) \nabla \phi(x)^* = \metric_x.
 $$ 
 Note that in the univariate setting, one recovers the normalization~\eqref{eq:normalizing-psi}.

\section{Sparsistency for the group-Lasso}
\label{sec:sparsistency-group-lasso}

This section is of independent interest, and contains a sparsistency analysis (i.e. analysis of the recovery of the correct support of a sparse vector) of the group Lasso which improves over existing analysis in the literature such as~\cite{bach2008consistency,liu2009estimation}. More precisely, Theorem~\ref{thm:group_IFT} details a sharper (almost necessary and sufficient) criteria under a weaker injectivity condition on the operator to invert.

\subsection{Group Lasso}

We consider a general group Lasso problem such as~\eqref{eq:group_cbp}, with possibly arbitrary group structures. 
Let  $\Gamma:\RR^N \to \Hh$ be a linear operator defined in the following way: Let $z = \ens{z_i}_{i=1}^K$ where $z_i\in\RR^{n_i}$ and $\sum_i n_i = N$ and let $\Gamma_i :\RR^{n_i}\to \Hh$ be a linear operator. We define 
$$
	\Gamma z \eqdef \sum_{i=1}^n \Gamma_{i} z_i.
$$
Similarly to~\eqref{eq:group_cbp}, letting $\norm{z}_{1,2}\eqdef \sum_i \norm{z_i}_2 $, the group Lasso optimisation problem is
\begin{equation}\label{eq:group_lasso}
	\min_{z\in \RR^N} \lambda \norm{z}_{1,2}+ \frac12 \norm{\Gamma z - y}^2.
\end{equation}

\paragraph{Notations.}

Define $\sign(z)  \in \RR^N$ as the vector such that for all $i\in [G]$, the  restriction to group $i$ is
$$
	\sign(z)_i \eqdef z_{i}/\norm{z_{i}}_2.
$$
Given $\Ii\subset [n]$, let $\Ii^c \eqdef [n]\setminus \Ii$ be its complement. Let $\Gamma_\Ii$ be the operator $\Gamma$ restricted to the index set $\Ii$. That is, for $z = \ens{z_i}_{i\in\Ii}\in\RR^m$ where $m=\sum_{i\in\Ii}n_i$,
$$
\Gamma_\Ii z= \sum_{i\in\Ii} \Gamma_{i} z_{i}.
$$
We denote $\Upsilon\eqdef \Gamma^*\Gamma$ and $\Upsilon_\Ii = \Gamma_\Ii^*\Gamma_\Ii$ where $\Gamma^*$ and $\Gamma_\Ii^*$ are the adjoint operator to $\Gamma$ and $\Gamma_\Ii$ respectively.
Given $\Ii$ and $m = \sum_{i\in\Ii} n_i$,
we define the projection operator $Q_{z}^\perp:\RR^m\to \RR^m$ by
$$
	Q_z^\perp v \eqdef \pa{  \pa{\Id - \frac{1}{\norm{z_{i}}^2_2} z_{i} z_{i}^*} v_{i}  }_{i\in\Ii}.
$$

\subsection{Refined Irrepresentability Condition (IC)}

Given some (unknown) $z^* \in \RR^N$, the goal is to assess the capability of~\eqref{eq:group_lasso} to recover $z^* $ approximately given noisy measurements 
$$
	y = y^* + w \in \Hh \qwhereq y^* \eqdef \Gamma z^* .
$$
In particular, when $w$ is small enough, one wishes to recover the correct group support $\Ii \eqdef \enscond{i}{\norm{z^*_i}\neq 0}$ of $z^* $.
It is well known in the literature on Lasso~\cite{fuchs2004sparse} and group Lasso problems~\cite{bach2008consistency,liu2009estimation} that this ``sparsistency'' property is governed by a so-called ``Irrepresentability Condition'' (IC) which somehow imposes some incoherence of $\Gamma$ when operating on vectors supported on $\Ii$ and $\Ii^c$
\begin{equation}\label{eq:cond-ic}\tag{IC}
	\norm{\Gamma_{\Ii^c}^*p_0}_{\infty,2} <1
	\qandq \forall i\in \Ii,\;
	\Gamma_i^* p_0 = \sign(z^*)_i
	\qwhereq 
	p_0 \eqdef \Gamma_\Ii  (\Gamma_\Ii^*\Gamma_\Ii )^\dagger \sign(z^*), 
\end{equation}
where we denoted $\norm{(z_i)_{i}}_{\infty,2} \eqdef \max_{i} \norm{z_{i}}_2$.
In addition to this IC condition, support stability analysis of the Lasso necessitate that the operator $\Gamma_\Ii$ is injective. Existing analysis of the group Lasso assumes the same hypothesis, but as we show below, as soon as groups have size larger than one (i.e. if one is not simply in the Lasso setup), then this is not sharp. 
We make instead the following weaker nullspace (N) assumptions:
\begin{equation}\label{eq:null}
	\ker(\Gamma_\Ii) \cap \ker(Q_{z^*}^\perp) = \ens{0}. \tag{N}
\end{equation}
This weaker condition was also studied in \cite{fadili2023sharp} as a necessary condition for establishing uniqueness of minimizers in the noiseless setting, we complement their results in Theorem \ref{thm:group_IFT} by establishing sparsistency under low noise.

\subsection{Sparsistency result}

We now show that, in the small noise regime, under (N) and (IC),  the solution of the group lasso share the same support as $z^*$ (Proposition~\ref{prop:group_spp_stab1}) and that the dependency with respect to the observation $y$ and $\lambda$ is smooth (Theorem~\ref{thm:group_IFT}). 

\begin{prop}\label{prop:group_spp_stab1}
Assume that  conditions (N) and (IC) holds.
For all $\lambda$ and $\norm{w}/\lambda$ sufficiently small, \eqref{eq:group_lasso} has a unique solution $z_{\lambda,w}$, $\ker(\Gamma_\Ii) \cap \ker(Q_{z_{\lambda,w}}^\perp) = \ens{0}$  and $\Supp(z_{\la,w})  =  \Ii$.
\end{prop}

\begin{rem}[Relationship to convex duality] 
Before proving this proposition, we first remark on the relationship between the incoherence condition and convex duality.
The dual problem  to \eqref{eq:group_lasso} is
\begin{equation}\label{eq:c-BP-unconstr-dual}
	\sup \enscond{ \dotp{y}{p}-\frac{\lambda}{2} \norm{p}^2 }{  \max_{i=1}^n \norm{\Gamma_{i}^* p}_2  \leq 1 }. \tag{$\Dd_\lambda(y)$}
\end{equation}
Problem~\eqref{eq:c-BP-unconstr-dual}  has a unique solution, since it is strongly concave. 
The dual solution $p_\lambda$ and any primal solution solution $z$ are related by $p_\lambda =  (y - \Gamma z)/\lambda$ and
$$
\Gamma^* p_\lambda \in \partial\norm{z}_{1,2} = \enscond{\eta}{\forall i\in \Supp(z),\quad \eta_{i} = \sign(z)_{i} , \qandq \norm{\eta}_{\infty,2}\leq 1}.
$$

Moreover, it can be shown \cite{vaiter2017model} that the dual solution $p_{\lambda}$ with data $y^*$ converges as $\lambda \to 0$ to the minimal norm dual solution:
\begin{equation}\label{eq:mnc}
	p_0 = \mathrm{argmin}\enscond{\norm{p}}{ p \text{ solves } \Dd_0(y^*)}.
\end{equation}
When (IC) holds, the minimal norm solution is precisely the vector defined in (IC). 
This is a consequence of  the primal-dual properties: any primal solution $z$ and dual solution $p$ satisfy $\Gamma^* p \in \partial \norm{z}_{1,2}$.  By (IC), $z^*$ is a primal solution and we can equivalently write \eqref{eq:mnc} as
\begin{equation}\label{eq:mnc2}
\begin{split}
p_0 &= \argmin\enscond{\norm{p}}{  \Gamma^* p \in \partial \norm{z^*}_{1,2}}, \\
&=\argmin\enscond{\norm{p}}{ \forall i\in\Ii, \;  \Gamma_i^* p = \sign(z^*)_i \qandq \norm{\Gamma^* p}_{\infty,2} \leq 1}.
\end{split}
\end{equation}
Note that
$$
	\Gamma_\Ii(\Gamma_\Ii^*\Gamma_\Ii)^\dagger \sign(z^*) =\argmin\enscond{\norm{p}}{ \forall i\in\Ii ,\; \Gamma_i^* p = \sign(z^*)_i}
$$
and by (IC), it satisfies the constraint of \eqref{eq:mnc2}. So, $\Gamma_\Ii(\Gamma_\Ii^*\Gamma_\Ii)^\dagger \sign(z^*)  = p_0$.
\end{rem}
\begin{proof}[Proof of Proposition \ref{prop:group_spp_stab1}]
Let $p_\lambda$ solve \eqref{eq:c-BP-unconstr-dual} with $y = y^*$. Let $p_0$ be the minimal norm dual solution. Then,
$$
\dotp{y^*}{p} - \frac{\lambda}{2}\norm{p_\lambda}^2 \geq \dotp{y^*}{p_0} - \frac{\lambda}{2}\norm{p_0}^2 \geq  \dotp{y^*}{p_0} - \frac{\lambda}{2}\norm{p_0}^2,
$$
where we used the optimality of $p_\lambda$ for the first inequality and the optimality of $p_0$ for the second inequality.
This implies that for all $\lambda$
$$
	\norm{p_\lambda}\leq \norm{p_0}.
$$
Since bounded sequences have weakly convergent subsequences, there exists $\bar p$ such that  $p_{\lambda_n} \rightharpoonup \bar p$ and $\norm{\bar p}\leq \norm{p_0}$. By continuity of $\norm{\Gamma_{i}^* \cdot}$, $\bar p$ also satisfies the constraints in the dual problem $\norm{\Gamma_{i}^* p}_2\leq 1$ for all $i$. So, $\bar p = p_0$.  Note also since $\norm{\bar p} = \norm{p_0} \leq \liminf_{n\to \infty } \norm{p_{\lambda_n}} \leq \norm{p_0}$, we have $\lim_{n\to\infty}\norm{p_{\lambda_n}} = \norm{p_0}$  and in fact, the full sequence $(p_{\la})_\lambda$ converges to $p_0$.
Let $p_{\lambda,w}$ be the solution to \eqref{eq:c-BP-unconstr-dual} with $y = y^* + w$, since this is the projection of $y/\lambda$ onto a convex set, we have
$$
	\norm{p_{\lambda,w} - p_\lambda} \leq \norm{w}/\lambda.
$$
So, for all $\lambda$ and $\norm{w}/\lambda$ sufficiently small, $\norm{\Gamma_{i}^* p_0}<1$ for  $i\not \in \Ii$ implies that $\norm{\Gamma_{i}^* p_{\lambda,w}}<1$. It follows that $z_{\lambda,w}$ has support in $\Ii$ for all $\lambda$ and $\norm{w}/\lambda$ sufficiently small.

Define $U_{\lambda,w}:\RR^{\abs{\Ii}}\to \Hh$ where its  $j$th column is  $(U_{\lambda,w})_{j} =  \Gamma_{j} (\Gamma^* p_{\lambda,w})_{j}$. Note that $(\Gamma^* p_{\lambda,w})_{i} = \sign(z_{\lambda,w})_{i}$ for $i\in \Supp(z_{\lambda,w})$. By the convergence of $p_{\lambda,w}$ to $p_0$ and invertibility of $U_{0,0}^*U_{0,0}$, we have that $U_{\lambda,w}^* U_{\lambda,w}$ is invertible for all $\lambda$ and $\norm{w}/\la$ sufficiently small. So, letting $(t_{\lambda,w})_i \eqdef \norm{(z_{\lambda,w})_{i}}$ for $i\in\Ii$, from $\Gamma_\Ii^*(\Gamma_\Ii z_{\lambda,w} - y)/\lambda =- \Gamma_\Ii^* p_{\lambda,w}$,
\begin{equation}\label{eq:abs_val_glasso}
t_{\lambda,w} = (U_{\lambda,w}^* U_{\lambda,w})^{-1} U_{\lambda,w} U_{0,0} t_{0,0} + U_{\la,w}^* w - \lambda  v.
\end{equation}
where $v_j =\norm{(\Gamma^* p_{\lambda,w})_{j}}^2 \leq 1$.
Since $(U_{\lambda,w}^* U_{\lambda,w})^{-1} U_{\lambda,w} U_{0,0} t_{0,0}\to t_{0,0}$ has all non-zero entries as $\lambda, \norm{w}\to 0$, it follows that $t_{\lambda,w} $ has all non-zero entries for $\lambda$ and $\norm{w}$ sufficiently small. Therefore, $\Supp(z_{\lambda,w}) = \Ii$.
\end{proof}

\begin{thm}\label{thm:group_IFT}
Assume that  conditions (N) and (IC) holds.
There exists a neighbourhood $V$ of $(0,0) \in \RR_+\times \Hh$ and a Frechet differentiable function $g: V\to \RR^n$ such that for all $(\lambda,w)\in V$, \eqref{eq:group_lasso} with $y = \Gamma z^* + w$ has a unique solution $z_{\lambda, w}$ with  support $\Ii$ and $(z_{\la,w})_\Ii = g(\lambda, w)$.
\end{thm}
\begin{proof}

By Proposition \ref{prop:group_spp_stab1}, there exists $c>0$ such that  for all $\lambda,\norm{w}/\lambda \leq c$, $z_{\lambda,w}$ is supported on $\Ii$ and 
$$
\eta_{\lambda,w} \eqdef \Gamma^* p_{\la,w} =  \Gamma^*\pa{ \Gamma z_{\la,w} - \Gamma z^* + w}/\lambda
$$
satisfies $\norm{(\eta_{\lambda,w})_{i}}_2<1$ for all $i\not\in \Ii$,
Let $$V= \enscond{(\la,w)}{0\leq \lambda \leq c \qandq\norm{w}/\lambda \leq c}.$$ 
For each $\lambda_0,w_0\in V\setminus \ens{(0,0)}$, we will apply the implicit function theorem to define a continuously differentiable function mapping $(\la,w)$ to solutions on a neighbourhood of $(\lambda_0,w_0)$. Define
$$
F(z,\lambda,w) =  \Gamma_\Ii^* \Gamma_\Ii z -  \Gamma_\Ii^* \Gamma_\Ii z^* - \Gamma_\Ii^* w + \lambda\sign(z).
$$
By the assumption that $z_{\la_0,w_0}$ is a group Lasso solution with regularisation parameter   $\lambda_0$ and noise $w$, $F(z_{\lambda_0,w_0},\lambda_0,w_0) = 0$. Moreover, 
$$\partial_z F(z,\lambda,w) =  \Gamma_\Ii^* \Gamma_\Ii +\lambda\diag\pa{ \pa{\frac{1}{\norm{z_i}}}_{i\in \Ii}}Q_{z^\perp}.$$
By Proposition \ref{prop:group_spp_stab1}, $\partial_z F$ is invertible (note however that $\partial_z F(0,0) = \Gamma_\Ii^* \Gamma_\Ii$ which is not necessarily invertible).

By the implicit function theorem, there exists a unique continuous function $g$ defined on a neighbourhood of $(\lambda_0, w_0)$ such that $g(\lambda_0, w_0) = z_{\la_0, w_0}$ and $F(g(\lambda,w),\lambda,w) = 0$ on this neighbourhood. Note that for $(\lambda, w)$ sufficiently close to $\lambda_0, w_0$,
$$
\eta \eqdef \frac{1}{\lambda} \Gamma^* \pa{\Gamma_\Ii g(\lambda,w) - \Gamma_\Ii z^* - w}
$$
satisfies $\norm{\eta_{i}} <1$ for $i\not\in\Ii$ by continuity of $g$ and since $\lambda>0$. So, $g(\lambda,w)$ defines a solution to the group lasso problem on a neighbourhood of $(\lambda_0,w_0)$.

The above construction of the function $g$ can be carried out  for all $(\lambda,w)\in \Vv\setminus \ens{(0,0)}$. Note that if we have $g$ constructed at $(\lambda,w)$ and $\hat g$ constructed at $(\hat \lambda,\hat w)$ with $(\tilde \lambda,\tilde w) \in \mathrm{dom}(g) \cap \mathrm{dom}(\hat g)$, then $g(\tilde \lambda,\tilde w) = \hat g(\tilde \lambda,\tilde w)$ by uniqueness of the function constructed via the implicit function theorem. So, the required function has been constructed on $\Vv \setminus \ens{(0,0)}$. It remains to extend this function to include $\ens{(0,0)}$. We already know from \eqref{eq:abs_val_glasso} that $z_{\lambda,w}$ converges to $z^*$, so $g(\lambda, w) \to z^*$ as $\lambda,w \to 0$. 
From the implicit function theorem, 
 \begin{align*}
  \partial_\lambda g(\lambda,w) &=  \partial_z F(g(\lambda,w),\lambda,w)^{-1} \sign(g(\lambda,w)) , \\
   \partial_w g(\lambda,w) &=  \partial_z F(g(\lambda,w),\lambda,w)^{-1} \Gamma_\Ii^*.
  \end{align*}
By the following Lemma, the limit  $\lim_{\lambda,w \to 0}\nabla g(\lambda,w)$ exists. Note also that since this limit exists, we can assume that there exists a constant $C$ that depends only on $\Gamma_\Ii$ and $z^*$ such that $\norm{\nabla g}\leq C$ on $\Vv$.

\end{proof}

  \begin{lem}
  $$\lim_{\lambda',w'\to 0}\partial_w g(\lambda',w')  \qandq  \lim_{\lambda',w'\to 0}\partial_\lambda g(\lambda',w') 
$$
exist. 
  \end{lem}
\begin{proof}
Let $z = g(\lambda,w)$ and recall that $z_{\lambda,w} \to z^*$ as $\lambda,\norm{w}\to 0$. Note that
\begin{align*}
\partial_z F(z,\lambda,w)^{-1} &= \pa{\Gamma_\Ii^* \Gamma_\Ii + \lambda\diag\pa{\frac{1}{\norm{z_{i}}}\Id -\frac{1}{\norm{z_{i}}} z_{i}z_{i}^* }_i}^{-1}\\
&= D_z \pa{D_z \Gamma_\Ii^* \Gamma_\Ii D_z + \lambda\Id - \lambda P_z}^{-1} D_z
\end{align*}
where $D_z$ diagonal matrix defined by $D_z v = (\norm{z_{i}} v_{i})_i$ and $P_z v = (\frac{1}{\norm{z_{i}}^2} z_{i} \dotp{ z_{i}}{v_{i}})_i$ is a projection operator satisfying $P_z^2 = P_z$. Letting $A_\lambda \eqdef D_z \Gamma_\Ii^* \Gamma_\Ii D_z + \lambda\Id$, by the Woodbury matrix identity, we have
\begin{align*}
\partial_z F(z,\lambda,w)^{-1} = D_z \pa{ A_\lambda^{-1} + \lambda A_\lambda^{-1} \pa{\Id - \lambda P_z A_\lambda^{-1} P_z}^{-1}  A_\lambda^{-1}   } D_z.
\end{align*}

Let us first check that $\pa{\Id - \lambda P_z A_\lambda^{-1} P_z}$ is indeed non-singular:
Define the singular value decomposition of $\Gamma_\Ii D_z = V S U^*$, where $U$, $V$ are orthonormal matrices and $S$ is a diagonal matrix with entries $s_k$. Note that $D_z \Gamma_\Ii^* \Gamma_\Ii D_z = US^2 U^*$, so $\lambda A_\lambda^{-1} = \lambda U (S^2 + \lambda\Id)^{-1} U^*$ and $\norm{\lambda A_\lambda^{-1}} \leq 1$ for all $\lambda\geq 0$. Let the columns of $U$ be $\ens{u_j}$ and let the index set $J$ be such that $\ens{u_j}_{j\not \in J}$ is a basis of $\mathrm{Ker}(\Gamma_\Ii D_z)$ and so $s_k>0$ for all $k\in J$. Let $v\neq 0$ and note that
\begin{align*}
v^*(\Id - P_z\lambda A_\lambda^{-1} P_z) v &= v^* P_z^\perp v +  v^* P_z  \sum_{k\in J} u_k \frac{s_k^2}{s_k^2 + \lambda} \dotp{u_k}{P_z v}.
\end{align*}
This is trivially non-zero if $P_z^\perp v  \neq 0$. Suppose $P_z^\perp v  = 0$ and $P_z v  \neq 0$.  Note that $ \ens{u_j }_{j\in J}$ is a basis for $\mathrm{Ker}(\Gamma_\Ii D_z)^\perp$. Moreover,
$$
\norm{P_z v}^2 - \norm{P_{\mathrm{Ker}(\Gamma_\Ii D_z)} {P_z v} }^2 = \norm{P_{\mathrm{Ker}(\Gamma_\Ii D_z)}^\perp {P_z v}}^2.
$$
This is zero if  $P_{\mathrm{Ker}(\Gamma_\Ii D_z)} {P_z v} = P_z v$, but this would imply that $D_z P_z v \in \mathrm{Ker}(Q_z^\perp) \cap \mathrm{Ker}(\Gamma_\Ii) = \ens{0}$ which is impossible. Hence, $ \sum_{k\in J} \frac{s_k^2}{s_k^2 + \lambda} \abs{\dotp{u_k}{P_z v}}^2 \geq \min_{k} \frac{s_k^2}{s_k^2 + \lambda}  \norm{P_{\mathrm{Ker}(\Gamma_\Ii D_z)}^\perp {P_z v}}^2>0$. Thus, $(\Id -\lambda P_z A_\lambda^{-1} P_z) $ is non-singular. Moreover, $\sup_{\lambda\in B}\norm{(\Id - \lambda P_z A_\lambda^{-1} P_z)^{-1}}$  exists on every bounded set $B\subset \RR_+$.

Now consider
\begin{align*}
\partial_w g& = \partial_z F(z,\lambda,w)^{-1} \Gamma_\Ii^*\\
&= D_z \pa{ \Id + \lambda A_\lambda^{-1} \pa{\Id - \lambda P_z A_\lambda^{-1} P_z}^{-1}    }  A_\lambda^{-1} D_z \Gamma_\Ii^*.
\end{align*}
This has limit 
$$
 D_{z^*} \pa{ \Id +  P_{\mathrm{Ker}(\Gamma_\Ii D_{z^*})}  \pa{\Id -  P_{z^*}  P_{\mathrm{Ker}(\Gamma_\Ii D_z)}  P_{z^*}}^{-1}    }  D_{z^*}^{-1} \Gamma_\Ii^*.
$$
 as $\lambda\to 0$ since  $\lim_{\lambda\to 0} A_\lambda^{-1} D_z \Gamma_\Ii^* = D_{z^*}^{-1} \Gamma_\Ii^\dagger$ and $\lim_{\lambda\to 0}\lambda A_\lambda^{-1} = P_{\mathrm{Ker}(\Gamma_\Ii D_{z^*})}$. Note also that for all  $\lambda\geq 0$,
 $$
 \norm{\partial_w g} \leq \norm{D_z} \pa{1+\norm{(\Id - \lambda P_z A_\lambda^{-1} P_z)^{-1}}} \norm{D_z^{-1} \Gamma_\Ii^\dagger}
 $$
 is uniformly bounded over $\lambda$ on every bounded set.

 Similarly,
 $$
 \partial_\lambda g(\lambda,w) = \partial_z F(z,\lambda,w)^{-1} \sign(z).
 $$
 By definition, $\sign(z) \in \mathrm{Im}(\Gamma_\Ii^*)$. So,
 \begin{align*}
  \partial_\lambda g(\lambda,w) &= \partial_z F(z,\lambda,w)^{-1} \Gamma_\Ii^{*} \Gamma_\Ii^{*,\dagger}\sign(z)\\
  &= D_z \pa{ \Id + \lambda A_\lambda^{-1} \pa{\Id - \lambda P_z A_\lambda^{-1} P_z}^{-1}    }  A_\lambda^{-1} D_z \Gamma_\Ii^{*} \Gamma_\Ii^{*,\dagger}\sign(z).
 \end{align*}
 This has limit
 $$
 D_{z^*} \pa{\Id + P_{\mathrm{Ker}(\Gamma_\Ii D_z)} (\Id - P_{z^*} P_{\mathrm{Ker}(\Gamma_\Ii D_{z^*})}P_{z^*} )^{-1} } D_{z^*}^{-1} \Gamma_\Ii^\dagger\Gamma_\Ii^{*,\dagger}\sign(z^*).
 $$

\end{proof}

\paragraph{Numerical illustration}
One interesting aspect of our refined irrepresentability condition is that it allows for sparsistency, even when there is no unique solution restricted to the support. 
%
%
One can construct a numerical example to illustrate this phenomenon in the following way (see also Figure \ref{fig:group-lasso-illustrate}). 
Consider a random Gaussian matrix $\Gamma \in \RR^{m\times qN}$ where $m=4$ is the number of columns, $N$ is the number of groups and $q=2$ is the size of each group. Let $y\in \RR^m$. Solving \eqref{eq:group_lasso} with $\lambda = 10^{-5}$, we obtain a solution $z_\lambda$ supported on 4 groups, denote its support by $\Ii$. In this case, $\Gamma_\Ii\in \RR^{4\times 8}$ has non-trivial kernel. Note that $\xi \eqdef \sign(z_\lambda)_\Ii = \Gamma_\Ii^*(y-X z_\lambda)/\lambda \in\mathrm{Ran}(\Gamma_\Ii^*)$. To check that (IC) holds with this sign patter, we compute  
\begin{equation}\label{eq:z}
z \eqdef \Gamma^* \Gamma_\Ii(\Gamma_\Ii^* \Gamma_\Ii)^\dagger \xi,
\end{equation}
 and observe that $\norm{z_{\Ii^c}}_\infty <1$ and condition (N) is satisfied. From Theorem \ref{thm:group_IFT}, it then follows that any vector $v$ such that $v$ has support $\Ii$ and $\sign(v_\Ii) = \xi$ can be stably recovered by solving \eqref{eq:group_lasso} with data $y = \Gamma v+w$ when $\lambda$ and $w$ are sufficiently small, even though $\Gamma_\Ii$ has non-trivial kernel. This is different from Lasso where uniqueness of the Lasso minimiser with support $\Ii$ implies that $\Gamma_\Ii$ must be injective \cite{foucart2013invitation}.

\begin{figure}
\begin{center}
\begin{tabular}{ccc}
\includegraphics[width=0.3\linewidth]{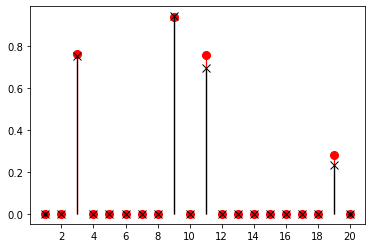}
&
\includegraphics[width=0.3\linewidth]{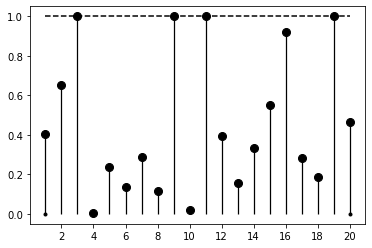}
\end{tabular}
\caption{Left: the ground truth is shown in red and the group-Lasso solution is shown in black. The signal here is of length $40$, made up of 20 groups each of size 2. The number of observations is 4. Right: Plot of  $\ens{\|z_{i}\|}_{i=1}^N$ where $z$ is defined in \eqref{eq:z}. Since $\|z_{i}\|<1$ for all $i$ not in the recovered support on the RHS, we know that the vector here can be stably recovered using the group-Lasso. 
\label{fig:group-lasso-illustrate}
}
\end{center}
\end{figure}


\section{Sparsistency for the discretized group-Lasso}
\label{sec:discretized-theory}
In this section, we discuss the stability of \eqref{eq:c-BP-unconstr} when \(\Gamma\) is a discretized version (on some grid) of continuous operators. Although this analysis is general, a typical example is the SR-Lasso problem \eqref{eq:c-BP-unconstr}. Therefore, this section serves as an abstract analysis of the sparsistency of the SR-Lasso for arbitrary smooth functions \(\gamma\). The subsequent section~\ref{sec:ti-kernel} applies this general analysis to the specific case of translation-invariant operators (typically convolutions), which enables us to derive more insightful stability criteria relating to the parameters of interest, particularly the kernel width.

%
%
%


\subsection{The dual certificate} 

We assume that $\Gamma$  is a discretized version of continuous operators taking values in Hilbert space $\Hh$. Precisely, we let $\ens{x_i}_{i=1}^n$ be a set of uniformly-spaced points with $\abs{x_i - x_{i+1}} = h$, and let $\Gamma_i = \gamma(x_i)\eqdef (\gamma_k(x_i))_{k=1}^q$, where $\gamma_k:\RR\to \Hh$ is a twice continuously differentiable function. In the notation of the previous section, we have $N = qn$.
The case of the SR-Lasso problem~\eqref{eq:c-BP-unconstr} corresponds to using $\ga_1=\phi$ and $\ga_2=\tau \psi$.  

In this setting, condition \eqref{eq:cond-ic} can equivalently and conveniently be written in terms of the ``dual certificate'' function.  It is defined as 
\begin{equation}\label{eq:f_0}
	f_0(x) \eqdef \norm{ \gamma(x)p_0}^2 = \sum_k \abs{\dotp{\gamma_k(x)}{p_0}}^2,
\end{equation}
where with a slight abuse of notation, we interpret $\gamma(x)$ as a linear operator $\gamma(x): \Hh \to \RR^q$ so for $p\in\Hh$, $ \gamma(x) p = \pa{\dotp{\gamma_k(x)}{p}}_{k=1}^q$.
The condition (IC) is equivalent to $$\max_{j\not\in \Ii} f_0(x_j)<1\qandq \forall i\in\Ii, \; f_0(x_i) = 1$$ in which case, we say that $p_0$ is \textit{non-degenerate}.
%
%
%

\begin{figure}
\begin{center}
\begin{tabular}{cccc}
\includegraphics[width=0.4\linewidth]{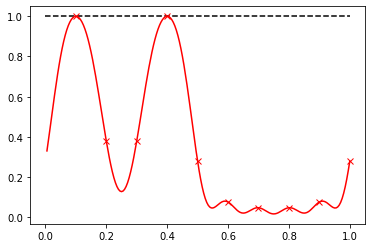}&
\includegraphics[width=0.4\linewidth]{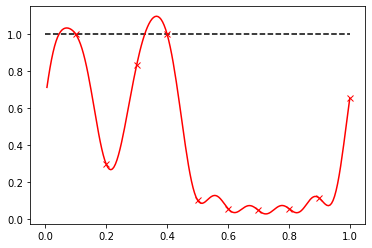}
\end{tabular}
\caption{Display of $f_0$ when $\phi$ is a Fourier operator. Left:  Plot of $f_0$ when the true Diracs are on the grid (left) and  when the true Diracs are off the grid (right). Even though $f_0$ escapes 1 when then true Diracs are  off-the-grid shift, the incoherence condition still holds if the grid points (shown with the crosses) are sufficiently coarse and skip the part where $f_0$ escapes 1. \label{fig:certificate}}
\end{center}
\end{figure}

\subsection{Sparsistency result}

This dual certificate function \( f_0 \) enables us to introduce a quantitative version of the (IC) sparsistency condition, controlling ``how much'' the certificate is non-degenerate. This is used in the subsequent Theorem~\ref{thm:nondegen} to present a quantitative sparsistency result, where the kernel conditioning is encapsulated by the following constants
\begin{align*}
	B_0 \eqdef \sup_x \norm{\bar \gamma(x)}^2 
	\qandq 
	B_1 \eqdef \sup_x \norm{\bar \gamma(x)^* \bar \gamma^{(1)}(x)}
	\qandq 
	B_2 \eqdef \sup_x\norm{\bar \gamma^{(2)}(x)}^2 + \norm{\bar\gamma^{(1)} (x)}^2.
\end{align*}
where for $j=1,2$, we define $\bar \gamma^{(j)}(x):\Hh \to \RR^q$ by $\bar \gamma^{(j)}(x) p \eqdef \pa{\dotp{\gamma_k^{(j)}(x)}{p}}_{k=1}^q$, where the norms above are operator norms.

\begin{thm}\label{thm:nondegen}
Assume that  conditions (N) hold. Assume  (IC) hold with the following bounds on $f_0$: for all $i\in\Ii$,
\begin{equation}\label{eq:cond_f0_near}
\abs{f_0' (x_i)}\leq \epsilon_1\qandq {f''_0(x)}\leq -M ,\quad \forall x\in [x_i-r,x_i+r]
\end{equation}
 and 
\begin{equation}\label{eq:cond_f0_far}
 f_0(x)<1-\mu, \quad \forall x\not\in \bigcup_{i\in\Ii} [x_i-r,x_i+r],
 \end{equation} 
for some $\epsilon_1, M, r, \mu>0$.
Let $\delta_2\in (0,1)$ and assume  that $h/2 >\frac{ \epsilon_1}{M(1-\delta_2) }$.
Let 
\begin{equation}\label{eq:cond_w}
V\eqdef \enscond{(\lambda,w)}{\frac{\norm{w}}{\lambda} +  \frac{2 C\norm{ \Gamma_\Ii^\dagger}}{\min_i\norm{(z^*)_i}} \lambda <\frac{1}{(1+2\norm{p_0})}\max\pa{ \frac{\mu}{B_0 }, \; \frac{\delta_2 M}{B_2},\; \frac{M(1-\delta_2) h/2 - \epsilon_1}{B_1} }}
\end{equation}
for some constant $C>0$.
For all $(\lambda, w)\in V$, there exists a unique solution to \eqref{eq:group_cbp} with $y = \Gamma z^* + w$. Moreover, this solution $z_{\la,w}$ has support $\Ii$ and there exists a Frechet differentiable function $g: V \to (\RR^q)^{\abs{\Ii}}$ such that  $(z_{\lambda,w})_\Ii = g(\lambda,w)$.
\end{thm}

\begin{rem}
In the case where $\gamma_1(x) = \phi(x)$ and $\gamma_2(x) = \psi(x)$ as mentioned in \eqref{eq:normalizing-psi}, if $z_{0,i} = (a_i, 0)$ for all $i\in \Ii$, i.e. the off-the-grid shift is zero, $f_0$ corresponds precisely to the square of the certificate studied for the Beurling-Lasso. By definition $f_0'(x_i) = 0$ for all $i\in \Ii$ and for many operators, it is well known that $f_0(x)<1$ for all $x\not\in \ens{x_i}_{i\in\Ii}$ if $\ens{x_i}_{i\in\Ii}$ are sufficiently separated \cite{candes2014towards,poon2023geometry}. When $z_i = (a_i, b_i)$ with $b_i\neq 0$, i.e. the off-the-grid shift is non-zero, then $f_0'(x_i) \neq 0$ and the function $f_0$ will exceed 1 near $x_i$ (see Figure \ref{fig:certificate}, however, if the grid points $x_i$ are sufficiently separated, then it is still the case that the incoherence condition is satisfied. Theorem \ref{thm:nondegen} formalises this idea: if the grid-size $h$ is sufficiently large (relative to the gradient and curvature of $f_0$ at $x_i$ for $i\in \Ii$), then the incoherence condition  holds for all $\lambda, w$ sufficiently small and there is stable support recovery.
\end{rem}

\begin{rem}[Extension to higher dimension]
For simplicity, the discretization is in dimension 1. The result extends to the higher dimensional setting where $\Gamma_k:\RR^d \to \Hh$. One can prove an analogous result  in higher dimensions under an assumption  on the gradient and Hessian of $f_0$:
$$ \abs{\nabla f_0 (x_i)}\leq \epsilon_1\qandq {\nabla^2 f_0(x)}\preceq -M \Id,\quad \forall \norm{x-x_i}\leq r.$$
\end{rem}

\subsection{Proof of Theorem \ref{thm:nondegen}}

By assumption (N), there exists $c>0$  (dependent only on $\Upsilon_I$ and $z^*$) such that
 $ \Upsilon_\Ii + \lambda \diag(1/\norm{z_i}) Q_z^\perp$ is invertible provided that $\norm{z-z^*} \leq c$. We also choose $c$ sufficiently small such that $\norm{z_i - z_{0,i}}\leq \frac12 \norm{z_{0,i}}$, to ensure that $z_i\neq 0$ for all $i\in I$.
Define $\bar V = \cup_{v\in\Vv} V$ where $\Vv$ is the collection of sets $V\subset \RR_+ \times \Hh$ satisfying the following conditions:
\begin{itemize}
\item[i)] $V$ is star-shaped
\item[ii)] $(0,0)\in V$ and $\norm{w}\leq \lambda$ for all $(\la,w)\in V$.
\item[iii)] There exists a $\Cc^1$ function $g:V\to \Bb_c(z^*) \subset \RR^{d\abs{I}}$ such that $g(0,0) = z^*$ and $F(g(\lambda,w), \lambda,w) = 0$ for all $(\lambda,w)\in V$, where
$$
F(z,\lambda,w) =  \Upsilon_\Ii  z -  \Upsilon_\Ii  z^* - \Gamma_\Ii^* w + \lambda\sign(z).
$$

\end{itemize}

By Theorem \ref{thm:group_IFT}, we know that $\Vv$ is non-empty and we can also assume that all such functions $g$ defined in iii) have bounded gradients $\norm{\nabla g}\leq B$ where $C$ depends only on $\Upsilon_\Ii$ and $z^*$. Also, $\Vv$ is closed to unions, since given any $V,V'\in\Vv$, $V\cup V'$ is star-shaped and $V\cap V'$ is nonempty. Moreover, if $g$ and $g'$ are the corresponding functions, then $g = g'$ on $V\cap V'$: to see this, we first note that $g = g'$ in a neighbourhood around $(0,0)$ by uniqueness from Theorem \ref{thm:group_IFT}. We argue by contradiction: Suppose that  there exists $\lambda>0$ and $\lambda,w\in V$  such that $g(\lambda,w) = g'(\lambda,w)$ but for all $\epsilon>0$, $g\neq g'$ on $V\cap V'\cap B_\epsilon((\la,w))$.  Let $z = g(\lambda,w)$, since $F(z,\lambda,w) = 0$ and $z \in B_c(z^*)$ by assumption, $\partial_z F(z,\lambda,w)$ is injective.  So, we can apply the implicit function theorem to construct a unique $\Cc^1$ function on a neighbourhood around $(\la,w)$  satisfying the constraints in iii). By uniqueness of this function,  $g = g'$ on this neighbourhood. 

Let us show that the set $\bar V$ contains $\enscond{(\lambda,w)}{\norm{w}\leq c/C\qandq \lambda\leq c/C}$.
Let $\lambda$  be the largest number such that $(\lambda,0)\in \bar V$.  Note that $\norm{g(\lambda,0) - z^*} = c$, if not, then we can again apply the IFT at $(\lambda,0)$ to contradict the maximality of $\lambda$.   So,
$$
c = \norm{g(\lambda,0) - g(0,0)} \leq \norm{\partial_\lambda g}  \lambda\leq C\lambda.
$$
So, $\lambda \geq \frac{c}{ C}$.
Let $\lambda\leq c/C$ and let  $w$ be such that $\norm{w}= 1$. Let $r\geq 0$ be the largest number such that $(rw,\lambda) \in \bar V$. Again, we must have $\norm{g(\lambda,w)-z^*} =c$. So,
$$
c = \norm{g(\lambda,0) - g(\lambda,rw)} \leq \norm{\partial_w g}  \norm{rw} \leq \norm{\partial_w g}  r \leq Cr,
$$
So, $\bar V$ contains $(\lambda,w)$ such that $\lambda,\norm{w} \leq c/C$.

In the following result, we   show that given $V\in \Vv$ and the corresponding function $g$, if $\lambda,\norm{w}/\lambda\leq R$ and $(\lambda,w)\in V$, then $g(\lambda,v)$ defines a group-Lasso solution.

%
%

\begin{prop} \label{thm:stability1}
Assume that $f_0$ and $h$ satisfy the conditions of Theorem \ref{thm:nondegen}.
Let $(\lambda,w)\in \bar V$ be such that
$$
\frac{\norm{w}}{\lambda} + \bar C \lambda <\frac{1}{(1+2\norm{p_0})}\max\pa{ \frac{\mu}{B_0 }, \; \frac{\delta_2 M}{B_2},\; \frac{M(1-\delta_2) h/2 - \epsilon_1}{B_1} } \eqqcolon R.
$$
where $B_0,B1,B2$ are constants that depend only on $\Gamma$  and $\bar C$ is a constant that depends only on $\Gamma_\Ii$ and $z^*$.
Then $g(\lambda,v)$ defines a group-Lasso solution.
\end{prop}

\begin{proof}
By assumption, there exists a $\Cc^1$ function $g: V\to (\RR^q)^n$ such that $$g((0,0)) = z^*\qandq F(g(\lambda,w),\lambda,w) = 0$$ and for some $C>0$,
$$
\norm{g((\lambda,w)) - g((0,0))} \leq  C\pa{ \norm{w}+\lambda} \leq C\lambda.
$$
For the remainder of this proof, we derive conditions on  $(\lambda,w)$ for which $z = g((\lambda,w))$ corresponds to a solution of \eqref{eq:c-BP-unconstr} with data $y=y^*+w$. For such $z$,
define
$$ p_{\lambda,w}\eqdef -( \Gamma_\Ii z - y^*-w)/\lambda$$
and
$$
\tilde f_{\la,w}(x)\eqdef \norm{\bar \gamma(x)p_{\lambda,w}}^2.
$$
By the optimality condition of \eqref{eq:c-BP-unconstr},  $z$ is a solution if and only if  
\begin{equation}\label{eq:toshow1}
\max_{i\in [n]\setminus\Ii}\tilde f_{\la,w}(x_i)<1.
\end{equation}
It therefore suffices to find conditions on $\lambda$ and $w$ under which \eqref{eq:toshow1} holds.  In the following lemma, we bound $\norm{ p_{\lambda,w} - p_0}$ where $ p_0$ is as defined in \eqref{eq:cond-ic}, then  use properties of $ f_0$ to  deduce that \eqref{eq:toshow1} holds.

\begin{lem}\label{lem:lipschitz_p}
The following holds:
$$
\norm{p_{\lambda,w}-p_0}  \leq \frac{\norm{w}}{\lambda} + \bar C\lambda
$$ 
where $$ \bar C\eqdef  \frac{2 C\norm{ \Gamma_\Ii^\dagger}}{\min_i\norm{(z^*)_i}}.$$
Consequently, if $\norm{p_{\lambda,w}-p_0}\leq 1$, then
$$
\abs{f_{\lambda,w}^{(j)}(x)-f_0^{(j)}(x) }\leq B_j   (1+2\norm{p_0})\frac{\norm{w}}{\lambda} + \bar C \lambda.
$$

\end{lem}

\begin{proof}[Proof of Lemma \ref{lem:lipschitz_p}]
 Since $F(z,\lambda,w) = 0$, we  deduce that
$$
 \Gamma_\Ii  \Upsilon_\Ii^\dagger  \Upsilon_\Ii z -
 \Gamma_\Ii  \Upsilon_\Ii^\dagger  \Gamma_\Ii^* (y^*+w) + \lambda  \Gamma_\Ii   \Upsilon_\Ii^\dagger \sign(z)  = 0
$$
and hence
\begin{equation}\label{eq:implicit_2}
-p_{\lambda,w} +\frac{1}{\lambda}\pa{\Id-  \Gamma_\Ii   \Gamma_\Ii^\dagger} w+    \Gamma_\Ii  \Upsilon_\Ii^{\dagger}\pa{\sign(z) -\sign(z^*) } +   p_0= 0.
\end{equation}
Note that  by the mean value theorem and since $\norm{z_i - z_{0,i}}\leq \frac12 \norm{z_{0,i}}$, we have $z_i\neq 0$ for all $i\in\Ii$ and
$$
\norm{\sign(z)_i-\sign(z^*)_i} \leq \frac{2}{\min_i\norm{(z^*)_i}} \norm{(z-z^*)_i}.
$$
It follows that
$$
\norm{p_{\lambda,w}-p_0} \leq \frac{\norm{w}}{\lambda} + \frac{2\norm{\bar \Gamma_\Ii^\dagger}   }{\min_i\norm{(z^*)_i}}\norm{z-z^*}   \leq \frac{\norm{w}}{\lambda} +  \frac{2C\norm{\bar \Gamma_\Ii^\dagger}}{\min_i\norm{(z^*)_i}} \lambda.
$$
The result now follows because
\begin{align*}
\tilde f_{\lambda,w} (x) -f_0(x)  &= \norm{ \bar \gamma(x)   p_{\lambda,w} }^2 - \norm{ \bar \gamma(x) p_0}^2\\
&=  \norm{\bar \gamma(x)  (p_{\lambda,w}-p_0) + \bar \gamma(x) p_0}^2 - \norm{ \bar \gamma(x)  p_0}^2 \\
&\leq  \norm{\bar \gamma(x)}^2 \norm{p_{\lambda,w}-p_0}^2 + 2\norm{\bar \gamma(x)^*\bar \gamma(x)}\norm{p_{\lambda,w}-p_0}\norm{p_0}.
\end{align*}

Similarly, 
\begin{align*}
\tilde f'_{\la,w}(x)   &=( \bar \gamma(x) p_{\lambda,w} )^* \bar \gamma^{(1)}(x) p_{\lambda,w}\\
& = ( \bar \gamma(x) (p_{\lambda,w}-p_0) )^* \bar \gamma^{(1)}(x) (p_{\lambda,w} -p_0)  +  ( \bar \gamma(x) p_0 )^* \bar \gamma^{(1)}(x) (p_{\lambda,w} -p_0)+ ( \bar \gamma(x) (p_{\lambda,w} -p_0))^* \bar \gamma^{(1)}(x) p_0 + f'_0(x)
\end{align*}
So,
$$
\abs{f_0'(x)-\tilde f_{\lambda,w}'(x)} \leq \norm{\bar \gamma(x)^* \bar \gamma^{(1)}(x)}\pa{ 2 \norm{p_{\lambda,w}-p_0}\norm{p_0} + \norm{p_{\lambda,w}-p_0}^2}
$$

We also have
\begin{align*}
\tilde f''_{\lambda,w}(x) =   (\bar \gamma^{(2)}(x) p_{\lambda,w})^* \bar \gamma(x) p_{\lambda,w} +\norm{\bar \gamma^{(1)}(x) p_{\lambda,w}}^2
\end{align*}
and
$$
\abs{\tilde f''_{\lambda,w}(x) -\tilde f''_0(x) } \leq  \left(\norm{\bar \gamma(x)^* \bar \gamma^{(2)}(x)}+\norm{\bar \gamma^{(1)}(x)}^2\right)\pa{ 2 \norm{p_{\lambda,w}-p_0}\norm{p_0} + \norm{p_{\lambda,w}-p_0}^2}
$$
\end{proof}

To show that \eqref{eq:toshow1} holds, note that for $x\not\in \bigcup_{i\in\Ii} [x_i-r,x_i+r]$,
$$
	\tilde f_{\lambda,w}(x) \leq 1-\mu + B_0 (1+2\norm{p_0}) \pa{\frac{\norm{w}}{\lambda} + \bar C \lambda } <1
$$
if
$$
	\pa{\frac{\norm{w}}{\lambda} + \bar C \lambda }<\frac{\mu}{B_0 (1+2\norm{p_0}) }.
$$
Moreover, for all $i\in\Ii$,
\begin{align*}
	\abs{f'_{\lambda,w}(x_i) }  &\leq \abs{ f_0'(x_i)}+B_1(1+2\norm{p_0}) \pa{\frac{\norm{w}}{\lambda} + \bar C \lambda }\\
	&  \leq \epsilon_1+B_1(1+2\norm{p_0}) \pa{\frac{\norm{w}}{\lambda} + \bar C \lambda }
\end{align*}
and for all $x\in [x_i-r,x_i+r]$ and $i\in\Ii$,
\begin{align*}
f''_{\lambda,w}(x) \leq -M +B_2(1+2\norm{p_0}) \pa{\frac{\norm{w}}{\lambda} + \bar C \lambda } \leq -(1-\delta_2) M 
\end{align*}
if
$$
 \pa{\frac{\norm{w}}{\lambda} + \bar C \lambda }<\frac{\delta_2 M}{B_2(1+2\norm{p_0})}.
$$
So, for all $x\in [x_i-r,x_i+r]$ and $i\in \Ii$,
\begin{align*}
f_{\lambda,w}(x) &=\abs{ f_{\lambda,w}(x_i) + (x-x_i) f_{\lambda,w}'(\x_i) +\frac12(x-x_i)^2 f_{\lambda,w}''(\xi) } \\
&\leq 1 + \pa{\epsilon_1+B_1(1+2\norm{p_0}) \pa{\frac{\norm{w}}{\lambda} + \bar C \lambda }}\abs{x-x_i} - \frac12 M(1-\delta_2) (x-x_i)^2 <1
\end{align*}
provided that
\begin{align*}
 \pa{\epsilon_1+B_1(1+2\norm{p_0}) \pa{\frac{\norm{w}}{\lambda} + \bar C \lambda }}<\frac12 M(1-\delta_2) \abs{x-x_i} 
\end{align*}
To summarise, \eqref{eq:toshow1} holds provided that
$$
\frac{\norm{w}}{\lambda} + \bar C \lambda <\frac{1}{(1+2\norm{p_0})}\max\pa{ \frac{\mu}{B_0 }, \; \frac{\delta_2 M}{B_2},\; \frac{M(1-\delta_2) h/2 - \epsilon_1}{B_1} } \eqdef R.
$$
\end{proof}


\section{Translation invariant kernel}
\label{sec:ti-kernel}

The results of the previous section are broad and quite abstract. To gain insight into the interplay between \(\phi\) and the positions of spikes that enable sparsistency, we now consider the specific case where \(\phi\) is associated with a translation-invariant kernel. We thus assume 
$$ 
	\dotp{\phi(x)}{\phi(x')} = \kappa(x-x'), 
$$
for some kernel function $\kappa$. We assume throughout that $\phi$ is normalized so that $\kappa(0) = \norm{\phi(x)}^2 =1$.
Recall that in this case, $\norm{\phi'(x)}^2 = -\kappa''(0)$.

\subsection{Dual certificate and minimum separation}

We consider the recovery problem from Section~\ref{subsec:blasso-lasso} where we aim to recover the sparse measure $\mu = \sum_{j\in\Ii} a_j \delta_{x_j + t_j}$ from low resolution measurements, that is, the parameters  $a_j$ and $x_j + t_j$ for $j\in\Ii$ from  the observation
$$
	y = \sum_{j\in\Ii}^m a_j \phi(x_j + t_j).
$$
Here, the parameters to be covered are
$$
	z_0 = ((a_j,  b_j))_{j\in\Ii},\qwhereq b_j \eqdef \sqrt{\abs{\kappa''(0)}} a_j t_j/\tau
$$
and the ``noise" is considered to be the remainder term $R$ in \eqref{eq:taylorexp}.  
In the following, we write
$$
\sign(z_0) = (s_a,s_b)
	\qwhereq (s_a)_i \eqdef \frac{a_i}{\norm{(a_i,b_i)}_2} 
	\qandq (s_b)_i \eqdef \frac{b_i}{\norm{(a_i,b_i)}_2} .
$$
To ease the following exposition of our result, we now formally define sparsistency.  

\begin{defn}[stable support recovery]
We say that there is \textit{stable support recovery} of $(a_j, b_j)$   if for all $\lambda$ sufficiently small, there is a unique solution to \eqref{eq:c-BP-unconstr}  supported on $\Ii$, and if moreover, denoting the solution by $(\hat a,\hat b)$, we have the error bound
$$
\| \hat a - a \| + \| \hat b - b \| = \Oo(\lambda).
$$
\end{defn}

We define the ``normalised'' kernels $\tilde \kappa_j(x) \eqdef \kappa^{(j)}(x) \abs{\kappa''(0)}^{-j/2}$ (note that it is guaranteed that $\tilde \kappa_2(0) = 1$). The \textit{dual certificate function} from \eqref{eq:f_0} can be written as
$$
	f_0(x) = \eta(x)^2 +\tau^2 \abs{\kappa''(0)}^{-1} \eta'(x)^2
$$
where
\begin{equation}\label{eq:eta}
\eta(x) = \sum_{j\in\Ii} u_j  \kappa(x_j-x) +\tau v_j  \tilde \kappa_1(x_j-x),\qwhereq \Upsilon_\Ii \binom{u}{v} = \binom{s_a}{s_b}.
\end{equation}
In this section, we establish stable support recovery under  the  assumption that the sought-after positions are sufficiently separated:
 Define the minimum separation distance as  
\begin{equation}\label{eq:min_sep}
 \Delta_{\min} \eqdef \min_{\substack{i\neq j \\i,j\in\Ii}}\abs{x_i - x_j}.
\end{equation} 
 Let $\delta_{\min} >0$ be such that  for all $\ell=0,\ldots, 4$,
$$
\delta_{\min} \geq   \sup\enscond{\sum_{i=1}^s   \abs{\tilde \kappa_\ell(z_0-z_i)}  }{  \min_{i\neq j} \abs{z_i-z_j}\geq \frac12 \Delta_{\min}}.
$$
Note that if $\abs{\kappa(x)}\to 0$ as $\abs{x}\to\infty$, then $\delta_{\min} $ decreases as $\Delta_{\min}$ increases.

\subsection{Sparsistency result in the translation invariant setting}

We now present an abstract result that shows support recovery, with assumptions on the decay and curvature properties of the kernel. 
These properties are captured by the following functions
\begin{equation}\label{eq:kernel_fns}
\begin{split}
K_0(x)&\eqdef  \kappa(x)^2 +  \tau^2 \tilde \kappa_1(x)^2 ,  
\\
K_1(x) &\eqdef \tilde \kappa_1(x) \pa{ \kappa(x) +   \tau^2 \tilde \kappa_2(x) }, 
\\
K_2(x) &\eqdef \tau^{-2} \pa{\tilde  \kappa_1(x)^2 +\tau^2 \tilde \kappa_2(x)^2 }.
\end{split}
\end{equation}
The following Section~\ref{sec:Gaussian} shows that these assumptions hold in the case of Gaussian deconvolution.

\begin{thm}\label{thm_G}
Assume that  $\frac{\kappa''(0)^2}{\kappa''''(0)} <1$ and
\begin{equation}\label{eq:interval_c}
	\tau^2 \in \left[\frac{\kappa''(0)^2}{\kappa''''(0)},1\right).
\end{equation}
Assume that  there exists constants $\mu,r,\delta$  such that for all $x\in [-r,r]$
$$
	K_0''(x) - \frac{2 s_{b,i}}{\tau s_{a,i}} K_1''(x)\leq (1-\delta)K_0''(0)\qandq K_2''(x)\leq (1-\delta) K_2''(0)
$$
and for all $\abs{x}\geq r$,
$$
	\abs{K_0(x)- \frac{2 s_{b,i}}{\tau s_{a,i}} \gamma K_1(x)}\leq 1-\mu \qandq \abs{K_2(x)}\leq 1-\mu.
$$
Then, provided that
$$
	\delta_{\min} \lesssim 1-\tau^2 \qandq t\lesssim  h (1-\tau^2)
	\qandq 
	\lambda \sim (1-\tau^2) h \sqrt{\abs{\kappa''(0)}}, 
$$ 
we have stable support recovery. Note that $\delta, \mu$ are treated as constants and hidden inside  $\lesssim$.
\end{thm}

The main take-away from the above result is that, treating $\tau$ as a constant, super-resolved Lasso allows for the stable recovery up to $\Oo(h)$ inside the grid. Although our result requires that $\tau<1$ and the behaviour deteriorates as $\tau$ approaches 1, in our numerical simulations, we always take $\tau=1$ so this result likely is not optimal and does not fully explain the practical performance.

\begin{rem}[On the assumptions for $K_0,K_1,K_2$]
For translation invariant kernels, one can check that
\begin{itemize}
\item 
$K_0(0) = K_2(0) = 1$ and $K_1(0) = 0$,
\item $K_0'(0) = K_2'(0) = 0$,
\item $K_1''(0)<0$ and $K_2''(0)<0$ and $K_1''(0) = 0$.
\end{itemize}
We therefore expect that when $\gamma\eqdef  s_{b,i}/(s_{a,i}\tau) $ is sufficiently small, our assumptions  hold as $K_0 - 2\gamma K_1$ and $K_2$ `peak' around 0. 
In Section \ref{sec:Gaussian}, we explicitly check these conditions in the case of a Gaussian kernel.
Figure \ref{figs:intuition2} shows plots of $K_0$, $K_2$ and $K_1$ in the Gaussian case. Note that $K_0(0) = K_2(0)  =1$ and $K_0''(0), K_2''(0)<0$ and both $K_0(x),K_2(x)$ decay away from 1 as $x$ increases. Moreover, $K_1(0) = 0$, so if $\gamma$ is sufficiently small, $K_1$ does not impact the curvature of $K_0$. Note also that $s_{a,i}^2 + s_{b,i}^2=1$ and  since $$\gamma = s_{b,i}/(s_{a,i}\tau) =  \sqrt{\abs{\kappa''(0)}}  t_i/\tau^2, $$  a bound on $\gamma$ is simply a bound on  how far one can move inside the grid.
 
 \begin{figure}[!htp]
 \begin{center}
   \includegraphics[width=0.4\linewidth]{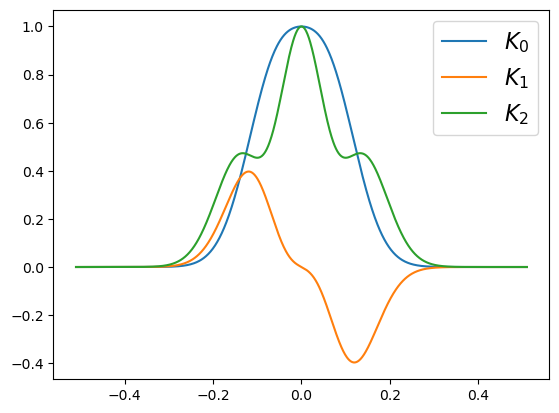}
 \end{center}
  \caption{Plot of $K_0,K_1,K_2$ defined in \eqref{eq:kernel_fns} for the Gaussian kernel. \label{figs:intuition2} }
 \end{figure}

\end{rem}

\subsection{Proof of Theorem \ref{thm_G}}

\paragraph{Intuition on proof}
To establish support stability, we need to check that $f_0$ satisfy the conditions \eqref{eq:cond_f0_near} and \eqref{eq:cond_f0_far} given in Theorem \ref{thm:nondegen}. In the translation invariant case, when $\Delta_{\min}$ is sufficiently large, we will show (Lemma \ref{lem:coeffs}) that the coefficients $u,v$ in \eqref{eq:eta} satisfy $u \approx s_a$ and $v\approx s_b/\tau^2$. So, $\eta$ defined in \eqref{eq:eta} satisfies 
$$\eta(x)\approx \sum_{j\in \Ii} g_i(x-x_j)$$ 
\begin{equation}
\qwhereq g_i(x)\eqdef  ( s_{a,i}\kappa(x) - s_{b,i} \tau^{-1} \tilde \kappa_1(x)) .
\end{equation}
Moreover, when the separation distance $\Delta_{\min}$ is sufficiently large with respect to the decay of $\kappa$ away from 0, we can essentially consider the $g_i$ functions as having disjoint support, and $f_0$ is closely approximated by the following function  $G_i$ when near $x_i$:
\begin{equation}\label{eq:G}
G_i(x) = g_i(x)^2 + \frac{\tau^2}{\abs{\kappa''(0)}} g_i'(x)^2.
\end{equation}
This is illustrated in Figure \ref{figs:intuition1} and  formalised in the appendix.
It is therefore  sufficient to check that $G_i$ satisfy the conditions \eqref{eq:cond_f0_far} and \eqref{eq:cond_f0_near} of Theorem \ref{thm:nondegen}.
To intuitively see that $G_i$ can be shown to satisfy these conditions, note that
\begin{equation}\label{eq:K}
G_i(x) = s_{a,i}^2  \underbrace{\pa{ \kappa(x)^2 +  \tau^2 \tilde \kappa_1(x)^2  }}_{K_0}+ s_{b,i}^2\underbrace{\tau^{-2} \pa{\tilde  \kappa_1(x)^2 +\tau^2 \tilde \kappa_2(x)^2 }}_{K_2}
- 2s_{a,i} s_{b,i}\tau^{-1} \underbrace{ \tilde \kappa_1(x) \pa{ \kappa(x) +   \tau^2 \tilde \kappa_2(x) }}_{K_1}
\end{equation}
So,
$$
G_i(x) = s_{a,i}^2 (K_0(x) -2\gamma  K_1(x) ) + s_{b,i}^2 K_2(x), \qwhereq \gamma = \frac{s_{b.i}}{s_{a,i} \tau}.
$$
The conditions on $K_0, K_1,K_2$ imply that
\begin{equation}
G_i''(x)\leq (1-\delta) G_i''(0),\qquad \forall x\in [-r,r]
\end{equation}
\begin{equation}
\qandq \abs{G_i(x)}\leq 1-\mu, \qquad \forall \abs{x}> r.
\end{equation}

\begin{figure}[htp]
 \begin{center}
  \begin{tabular}{ccc}
 \includegraphics[width=0.4\linewidth]{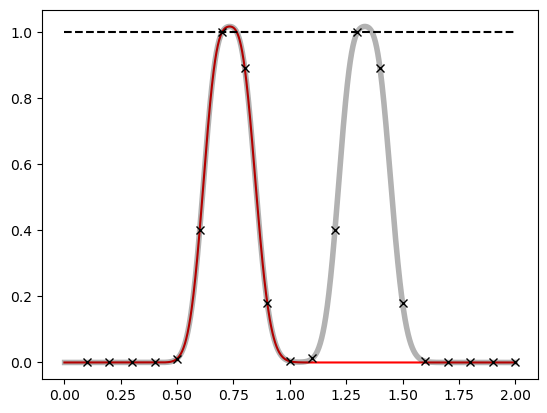}&
 \includegraphics[width=0.4\linewidth]{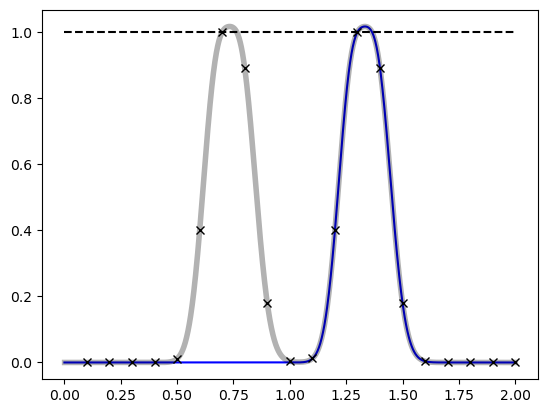}&
 \end{tabular}
  \end{center}
 \caption{Example of $f_0$ for a measure supported on $x_1=0.7,x_2=1.3$.  $f_0$ is shown in  gray, $G_1$ is shown in red and $G_2$ is shown in blue. Note that $f_0$ is well approximated by $G_1$ near $x_1$ and by $G_2$ near $x_2$. \label{figs:intuition1} }
\end{figure}

The assumptions essentially mean that in each neighbourhood of $x_i$, we can simply consider the quadratic approximation
$$
	G_i(0) + G_i'(0)x +\frac12 G_i''(0) x^2,
$$
in place of studying $f_0$.
We will make use of the following key properties of $G_i$, which can be readily verified by direct computation:
\begin{lem}\label{lem:property_G}
For all $\tau>0$,
\begin{itemize}
\item[i)]  $G_i(0) = 1$.

\item[ii)]   $G_i'(0) =\frac{1}{\tau} 2 s_{a,i} s_{b,i} \abs{\kappa''(0)}^{\frac12} \pa{1- \tau^2  }.$ This is 0 if $\tau =1$.

\item[iii)]  
$
G_i''(0) = 2 \kappa''(0) \pa{ s_{a,i}^2 (1- \tau^2)  + s_{b,i}^2  \pa{ 
  \frac{  \kappa^{(4)}(0) }{\kappa''(0)^2}-\tau^{-2}} }.
$
In order for $G''_i(0)$ to be negative,   it is sufficient that
$
\tau^2 \in \left[\frac{\kappa''(0)^2}{\kappa''''(0)},1\right]
$
\end{itemize}
\end{lem}

\begin{proof}[Proof of Theorem \ref{thm_G}]
First note that \eqref{eq:null} holds trivially because under the minimum separation assumption, $\Upsilon_\Ii$ is invertible as shown in Lemma \ref{lem:coeffs}. So, to  apply Theorem \ref{thm:nondegen}, we simply need to check \eqref{eq:cond_f0_near} and \eqref{eq:cond_f0_far}. We let $$t  \sqrt{\abs{\kappa''(0)}}= \hat T \qandq M/\abs{\kappa''(0)} = \hat M \qandq h \sqrt{\abs{\kappa''(0)}} = \hat H.$$
Note that for $w$ corresponding to the Taylor expansion residual,
$$
\norm{w} \sim \abs{\kappa''(0)} t^2 = \hat T^2
$$
and
$b = \frac{ta \abs{\kappa''(0)}^{\frac12}}{\tau}$, so
$$
  s_b =\frac{\frac{t \abs{\kappa''(0)}^{\frac12}}{\tau} }{ \sqrt{   t^2  \abs{\kappa''(0)} \tau^{-2}  + 1 } } = \Oo(\hat T)  \qandq s_a = \frac{1}{\sqrt{   t^2 \abs{\kappa''(0)} \tau^{-2}  +1 }} = \Oo(1).
$$

By Lemma \ref{lem:close},  first note that we want to choose $\epsilon_1$ so that
\begin{align*}
\abs{f_0'(x_i)} &\leq \abs{G_i'(0)} + \Oo(B_1 \delta_{\min}) \\
&=\frac{2}{\tau} \abs{s_a}\abs{s_b} \abs{\kappa''(0)}^{\frac12} (1-\tau^2 ) + \Oo(B_1 \delta_{\min}) \leq \epsilon_1.
\end{align*}
So, $$\epsilon_1/\sqrt{\abs{\kappa''(0)}} = \Oo( \hat T   +\delta_{\min} ).$$
Moreover, for $x\in  [x_i-r,x_i+r]$, again using Lemma \ref{lem:hessian}, we want to choose $M>0$ such that
\begin{align*}
f_0''(x)& \leq (1-\delta) G_i''(0) + \Oo(B_2 \delta_{\min})\\
&= 2 (1-\delta)\kappa''(0) \pa{ s_{a,i}^2 (1- \tau^2)  + s_{b,i}^2  \pa{ 
  \frac{  \kappa^{(4)}(0) }{\kappa''(0)^2}-\tau^{-2}} } + \Oo(B_2 \delta_{\min}) \leq  -M.
\end{align*}
Note that
$$
 s_{a,i}^2 (1- \tau^2)  + s_{b,i}^2  \pa{ 
  \frac{  \kappa^{(4)}(0) }{\kappa''(0)^2}-\tau^{-2}}  \geq \min \pa{(1- \tau^2),\frac{  \kappa^{(4)}(0) }{\kappa''(0)^2}-\tau^{-2}} ,
$$
So, we can let 
\begin{equation}\label{eq:n1}
\delta_{\min} = \Oo(1-\tau^2)\qandq \hat M = \Oo(1-\tau^2).
\end{equation}
By Lemma \ref{lem:far}, $f_0(x)<1-\mu/2$ for all $x\not\in \bigcup_i [x_i-r,x_i+r]$.

So, by the condition \eqref{eq:cond_w}, the grid size $h$ 
and $w$, $\lambda$ should satisfy
\begin{align*}
\frac{\norm{w}}{\lambda},\lambda\lesssim \min\pa{\mu, 1-\tau^2}
\end{align*}
\begin{align*}
\qandq \frac{\norm{w}}{\lambda} + \lambda + \hat T  \lesssim  h\hat M \sqrt{\abs{\kappa''(0)}} \sim h (1-\tau^2) \sqrt{\abs{\kappa''(0)}}.
\end{align*}
Plugging in $\norm{w} \sim \hat T^2$, we have
\begin{align*}
\frac{\hat T^2}{\lambda},\lambda\lesssim \min\pa{\mu, 1-\tau^2}
\qandq
\frac{\hat T^2}{\lambda} + \lambda + \hat T  \lesssim 
h (1-\tau^2) \sqrt{\abs{\kappa''(0)}}.
\end{align*}
From this, we see that it is sufficient to choose
$$
\lambda \sim (1-\tau^2) h \sqrt{\abs{\kappa''(0)}} \qandq t\lesssim  h (1-\tau^2).
$$
\end{proof}

\subsection{Gaussian deconvolution}\label{sec:Gaussian}
 
To gain even deeper insights into the sparsistency capability of SR-Lasso, we now apply Theorem \ref{thm_G} specifically to the case of the Gaussian operator
$$
	\phi(x) = \pa{\frac{2}{\pi\sigma^2}}^{1/4} \exp(-(x-\cdot)^2/\sigma^2).
$$
This corresponds to the kernel \(\kappa(x) = \exp(-x^2/(2\sigma^2))\). By applying Theorem \ref{thm_G} to this setting, we obtain the following recovery result.
 
 \begin{thm}\label{thm:Gaussian}
 Let $\tau\in (0.8,1)$.
 Assume that $\Delta_{\min}\gtrsim  \sigma \sqrt{\log(1/(1-\tau^2))} $, $h = \Oo(\sigma)$ and
 $$
 t\lesssim h (1-\tau^2) \qandq \lambda \sim h(1-\tau^2)/\sigma,
 $$
then,  we have stable support recovery.
 \end{thm}

\paragraph{Numerical illustration.}

We consider reconstruction from Gaussian measurements with $\sigma = 0.07$. The reconstruction grid consists of uniformly spaced points on the interval $[0,2]$, with a grid spacing $h=\sigma$. We consider the reconstruction of two spikes spaced $4\sigma$ apart with their positions at a distance $0.25h$ inside the grid, that is, the positions are $x_i + 0.25 h$ where $x_i$ are grid points. The certificates and reconstructions (with $\lambda = 3\sigma \lambda_{\max}$) are shown in Figure \ref{fig:gaussian_tau}. The certificate is nondegenerate for a range of $\tau$ values around 1 and for these values ($\tau =0.8, 1$ are shown), there is support stability.

\begin{figure}
\begin{tabular}{cccc}
$\tau=0.6$&$\tau=0.8$&$\tau=1$&$\tau=1.2$\\
\includegraphics[width=0.22\linewidth]{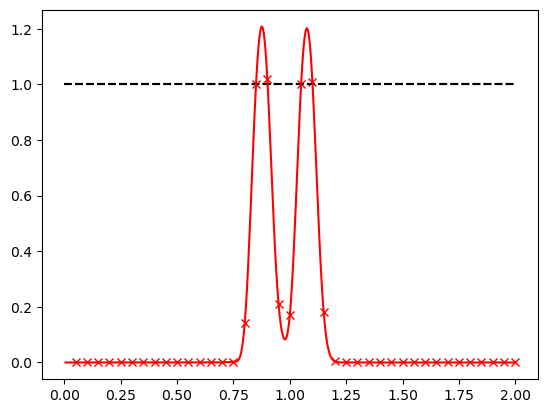}
&
\includegraphics[width=0.22\linewidth]{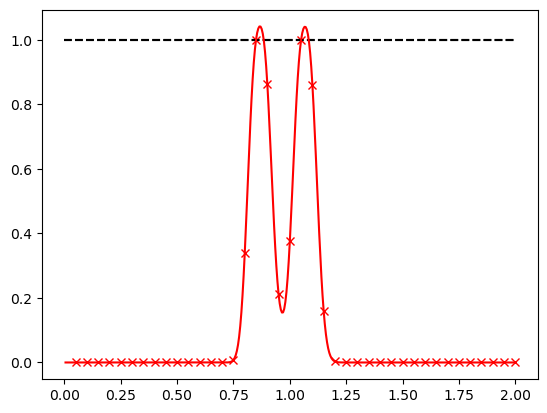}
&
\includegraphics[width=0.22\linewidth]{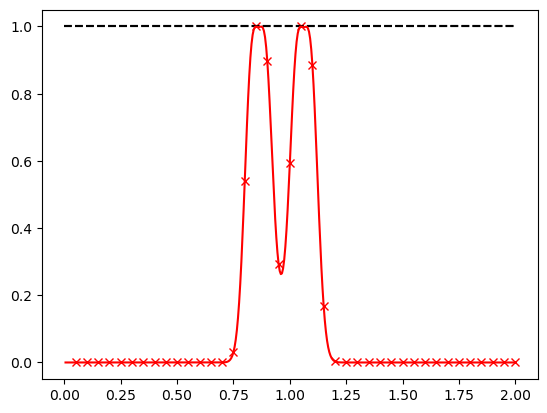}
&
\includegraphics[width=0.22\linewidth]{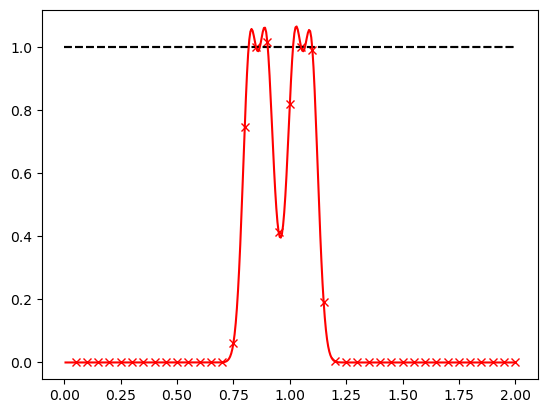}
\\
\includegraphics[width=0.22\linewidth]{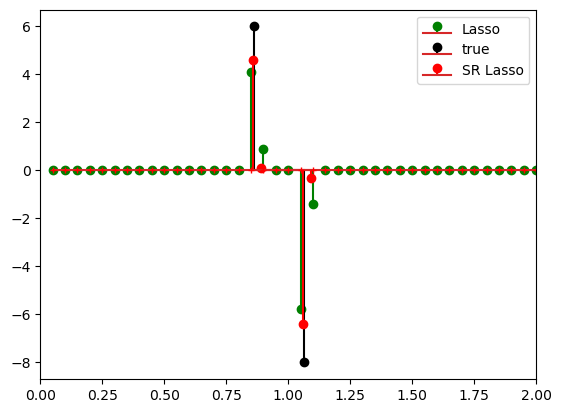}\
&
\includegraphics[width=0.22\linewidth]{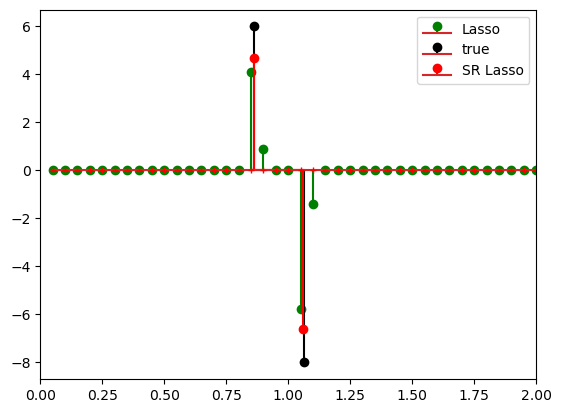}
&
\includegraphics[width=0.22\linewidth]{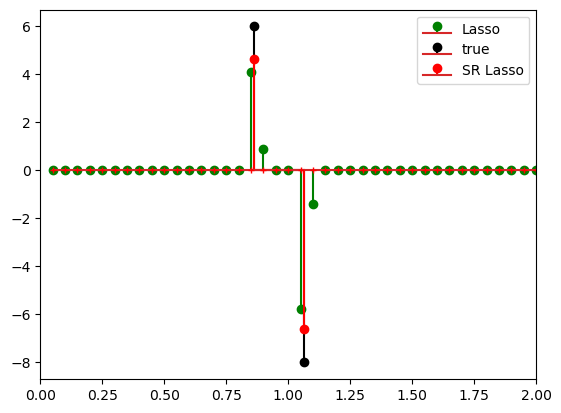}
&
\includegraphics[width=0.22\linewidth]{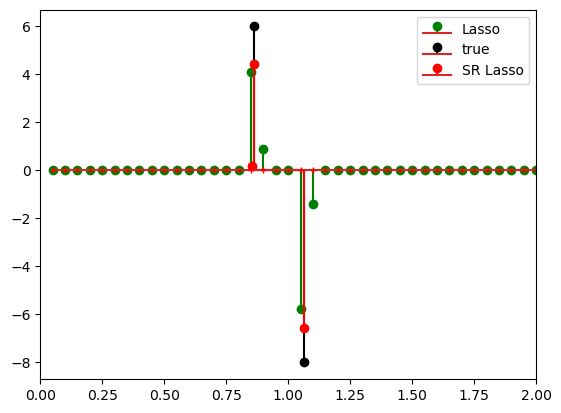}\\
\end{tabular}
\caption{ This figure displays the certificate $f_0$ and super-resolved Lasso for a range of $\tau$ values. In the above, $\tau=0.8,1$ have nondegenerate certificates and hence, there is support stability. \label{fig:gaussian_tau}}
\end{figure}

\subsubsection{Proof of Theorem \ref{thm:Gaussian}}

To prove this theorem, we simply need to check that the assumptions on the kernel functions in Theorem \ref{thm_G} hold.  This is done in Lemma \ref{lem:gauss1} and \ref{lem:gauss2} below. Note that $\sqrt{\abs{\kappa''(0)}} \sim 1/\sigma$ and  $\sqrt{\frac{\kappa''(0)^2}{\kappa''''(0)} }= 1/\sqrt{3}<0.8$ holds. To ensure that $\delta_{\min} \lesssim (1-\tau^2)$, we need $\Delta_{\min}\gtrsim\sqrt{\log(1/(1-\tau^2))} \cdot \sigma
$ by Lemma \ref{lem:gauss1}. By Lemma \ref{lem:gauss2} below, the conditions on $G_i$ are satisfied if $r =  \Oo(\sigma)$ and $t= \Oo(\sigma(1-\tau^2))$.

\begin{lem}\label{lem:gauss1}
For $c>0$,
$\delta_{\min}\leq c$ provided that $\Delta_{\min}\gtrsim\sqrt{\log(1/c)} \cdot \sigma.
$
\end{lem} 
\begin{proof}
Since $\kappa = \exp(-x^2/(2\sigma^2))$, if $\abs{x_i-x_j}\geq \Delta$ for all $i\neq j$, we have
$$
\sum_{j\neq i} \kappa(x_i - x_j) \leq \sum_{j=1}^\infty \exp(-j^2\Delta^2/(2\sigma^2)) \lesssim \exp(-\Delta^2/(2\sigma^2)) \leq c
$$
provided that 
$$
\sqrt{\log(1/c)} \cdot \sigma \lesssim \Delta.
$$
\end{proof}
 
\begin{lem}\label{lem:gauss2}
 Assume that $\tau\in [0.8,1]$ and let $r = 0.3 \sigma$. Assume also that
$$
t\sqrt{\abs{\kappa''(0)}} \leq \tau^2  \min\pa{\frac{1}{10}, \frac12(1-\tau^2)}.
$$ 
Then,  $ K_2''(x)\leq \frac{1}{10} K_2''(0)$  and $ K_0''(x) + \frac{2s_{b,i}}{s_{a,i}\tau}  K_1''(x)\leq \frac{1}{10} K_0''(0)$ for all $x\in [-r,r]$ and   for all $\abs{x}>r$,
 $$
 \max\pa{\abs{ K_0(x) + \frac{2s_{b,i}}{s_{a,i}\tau}  K_1(x)},\abs{ K_2(x)}}\leq 0.9962.
 $$
 Consequently,
 $$
 G_i''(x)<  \frac{1}{10} G_i''(0)<0 \qquad \forall x\in [-r,r]
 $$
 and  for all $\abs{x}>r$,
 $$
 \abs{G_i(x)}<  0.9962.
 $$
 \end{lem}
 
\begin{proof}
 Note that
\begin{align*}
G_i(x) &= s_{a,i}^2  \underbrace{\pa{ \kappa(x)^2 +  \tau^2 \tilde \kappa_1(x)^2  }}_{K_0}+ s_{b,i}^2\underbrace{\tau^{-2} \pa{\tilde  \kappa_1(x)^2 +\tau^2 \tilde \kappa_2(x)^2 }}_{K_2}
- 2s_{a,i} s_{b,i}\tau^{-1} \underbrace{ \tilde \kappa_1(x) \pa{ \kappa(x) +   \tau^2 \tilde \kappa_2(x) }}_{K_1}.
\end{align*}
So,
$$
G_i(x) = s_{a,i}^2  \pa{ K_0(x) -2\gamma K_1(x)  } + s_{b,i}^2 K_2(x), \qwhereq \gamma = \frac{s_{b.i}}{s_{a,i} \tau}
$$
For the Gaussian $\kappa(x) = \exp(-x^2/(2\sigma^2))$, $K_j(x) = \hat K_j(x/\sigma)$ and $K_j''(x) = \sigma^{-2} \hat K_j''(x/\sigma)$ where
\begin{align*}
&\hat K_0(x) = \pa{\tau^2 x^2 + 1} \exp(-x^2)\qandq
\hat K_1(x) =  -x\pa{\tau^2 x^2 + 1 - \tau^2 } \exp(-x^2)\\
&\qandq
\hat K_2(x)= \frac{1}{\tau^2}\pa{\tau^2(x^2-1)^2 + x^2} \exp(-x^2).
\end{align*}
We will therefore  show that for $i=0,2$, $\hat K''_i(x)$ is negative in a neighbourhood around 0 and upper bound $\abs{\hat K''(x)}$ in this neighbourhood.
Note that 
\begin{align*}
\hat K_0''(x) = \pa{4 \tau^2 x^4  + (4 - 10 \tau^2 )x^2 - 2 + 2 \tau^2}\exp(-x^2)\\
\hat K_1''(x) =  {-2 x (2 \tau^2 x^4+( 2 - 9\tau^2  ) x^2+ 6 \tau^2 - 3)}\exp(-x^2)\\
\hat K_2''(x) =\frac{1}{\tau^2} (4 \tau^2 x^6  + (4 - 26 \tau^2 )x^4 + \pa{36 \tau^2  - 10 }x^2 - 6 \tau^2 + 2)\exp(-x^2).
\end{align*}
We will first find a region $r$ so that for $i=0,2$ and all $\abs{x}\leq r$,
$$
\hat K_i''(x) \leq c \hat K_i''(0)
$$
for some appropriate constant $c$ and bound $\abs{\hat K_1''(x)}$ in this region so that this term can be `absorbed' into $\hat K_0''(0)$.

\begin{enumerate}
\item 
Note that $\hat K_2''(x)<(1-\delta) \hat K_2''(0) = (1-\delta)(-6\tau^2 + 2)/\tau^2$ provided that $x$ satisfies
$$
(4 \tau^2 x^6  + (4 - 26 \tau^2 )x^4 + \pa{36 \tau^2  - 10 }x^2 - 6 \tau^2 + 2)\exp(-x^2) < (1-\delta)(-6\tau^2 + 2).
$$
We can ignore that $\exp(-x^2)\leq 1$ term, and it is sufficient to require $x$ to satisfy
$$
\abs{x} <\sqrt{ \min\pa{\frac{\delta(6\tau^2 - 2)}{36\tau^2 - 10}, \frac{26\tau^2 - 4}{4\tau^2}}}.
$$
Choosing $\delta=0.9$, the RHS is bounded from below by $0.356$ if $\tau>0.8$.

\item 
We have
$$
\hat K_0''(x) \leq  \hat K''_0(0) = (-2+2\tau^2)
$$
provided that
$
\abs{x} \leq \sqrt{\frac{10\tau^2-4}{4\tau^2}},
$
the RHS is lower bounded by $0.968$ if $\tau>0.8$.

\item 
We have
$$
\abs{\hat  K_1''(x)} \leq 2\abs{x}(6\tau^2-3)
$$
provided that $\abs{x} \leq \sqrt{\frac{9\tau^2-2}{2\tau^2}}$, the RHS is lower bounded by $1.713$ for $\tau>0.8$.
We therefore have
 $$\hat K_0''(x) - 2\gamma\hat K_1''(x)  <(1-\delta)(-2+2\tau^2) = (1-\delta)\hat K_0''(0)$$ provided that
$$
\abs{x}\leq \frac{\delta(2-2\tau^2)}{4\abs{\gamma} (6\tau^2-3)}.
$$
Choosing $\delta = 0.9$ and $\abs{\gamma} \leq \min(0.1,\frac12(1-\tau^2))$, the RHS in the bound for $\abs{x}$ is lower bounded by $0.3$ if $\tau >0.8$.

This condition on $\gamma$ can be written as
$$
   \left|\frac{s_{b,i}}{\tau s_{a,i}}  \right|\leq \min\pa{\frac{1}{10}, \frac12(1-\tau^2)},
$$ 
which holds if
$$
\frac{t\sqrt{\abs{\kappa''(0)}}}{\tau^2} \leq  \min\pa{\frac{1}{10}, \frac12(1-\tau^2)}.
$$
\end{enumerate}
By combining the three observations above, we have
$$
G_i''(x) \leq (1-\delta) G_i''(0)
$$
provided that $\tau>0.8$ and $\abs{x} \leq 0.3 \sigma$.

Finally, one can numerically verify that (see Figure \ref{figs:proof} for  $\max_\tau \hat K_0(x) + 2 \min(0.1,\frac12(1-\tau^2)) \hat K_1(x)$ and $\hat K_2$)  that
for $\abs{x}>0.3$,
$$
\hat K_0(x) - 2\abs{\gamma \hat K_1(x)} \leq 0.9962
$$
$$
\qandq \hat K_2(x) \leq 0.8854.
$$ It follows that
 $$\abs{G_i(x)}< 0.9962 s_a^2 + 0.8854 s_b^2\leq 0.9962.
$$
 \end{proof}
 
\begin{figure}
 \begin{center}
  \begin{tabular}{ccc}
 \includegraphics[width=0.3\linewidth]{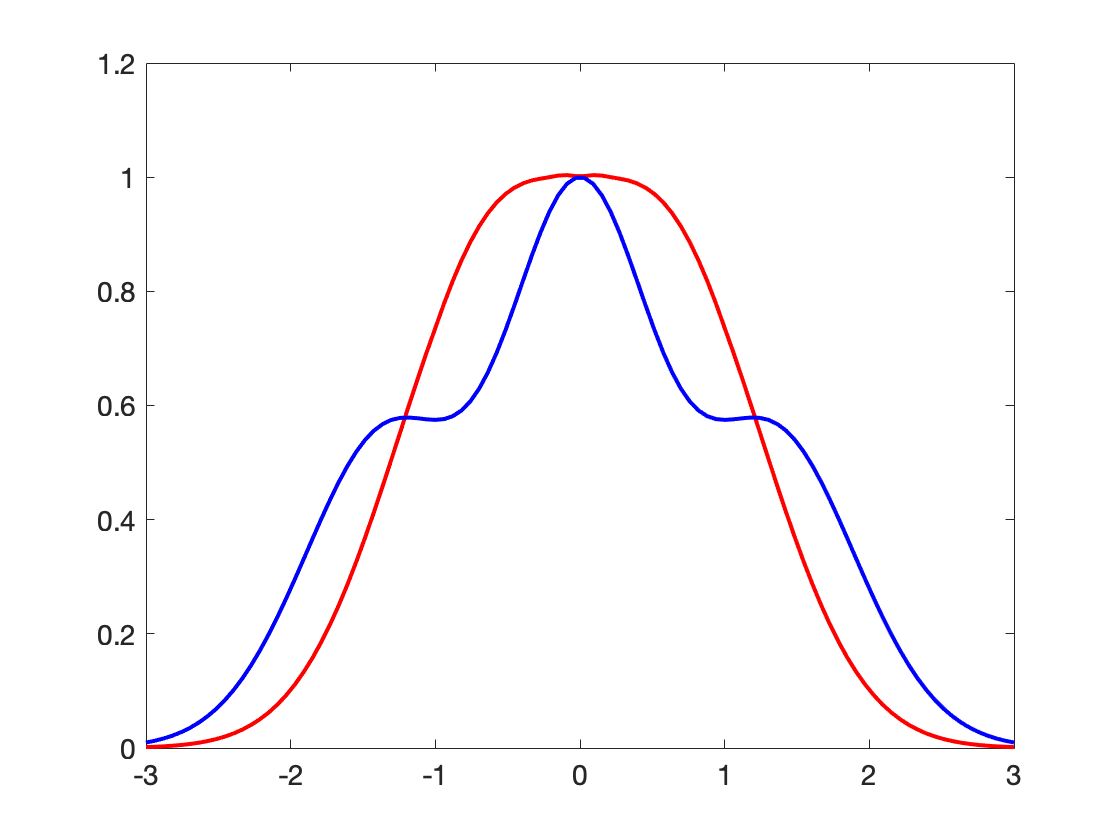}&
 \includegraphics[width=0.3\linewidth]{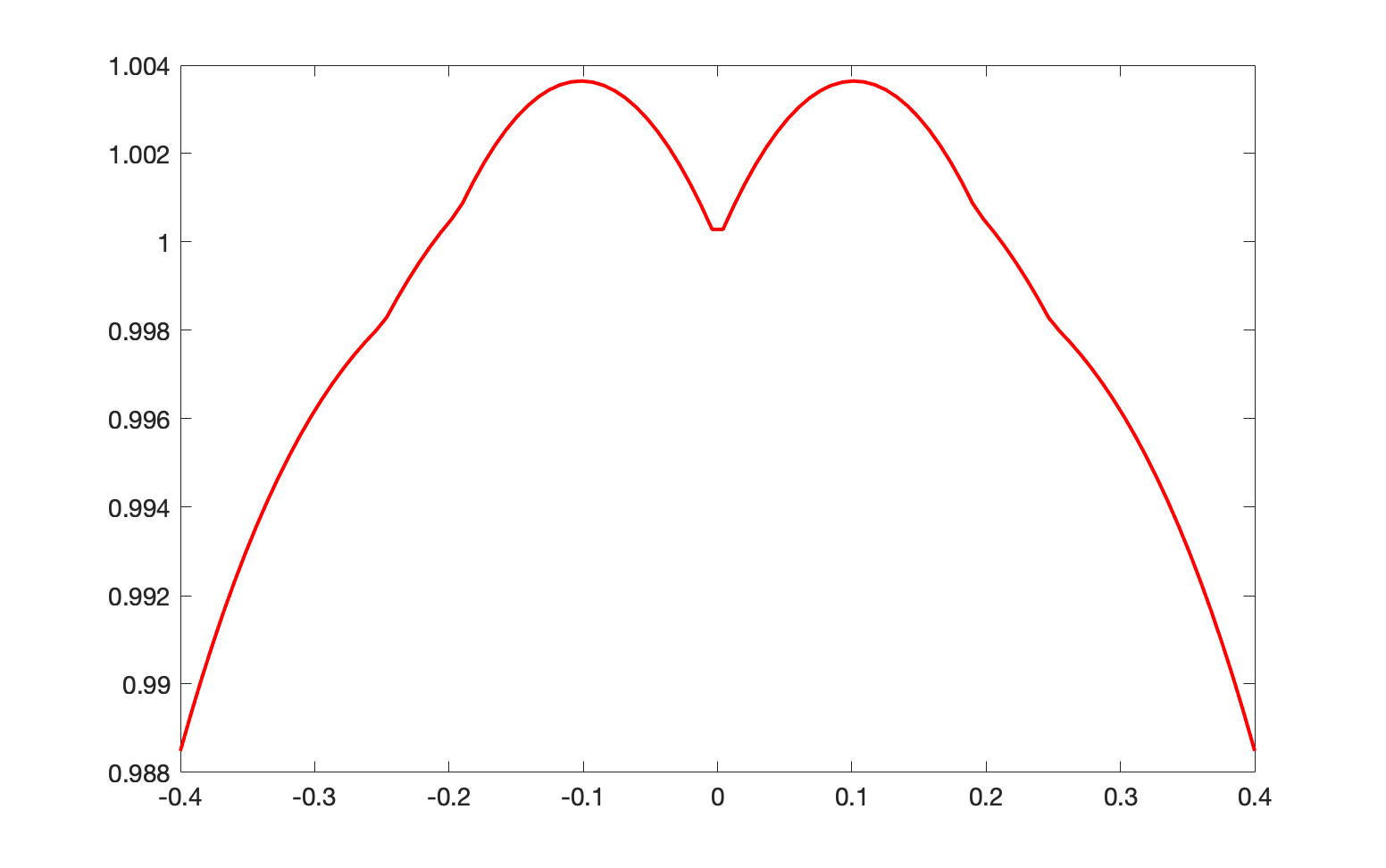}&
 \includegraphics[width=0.3\linewidth]{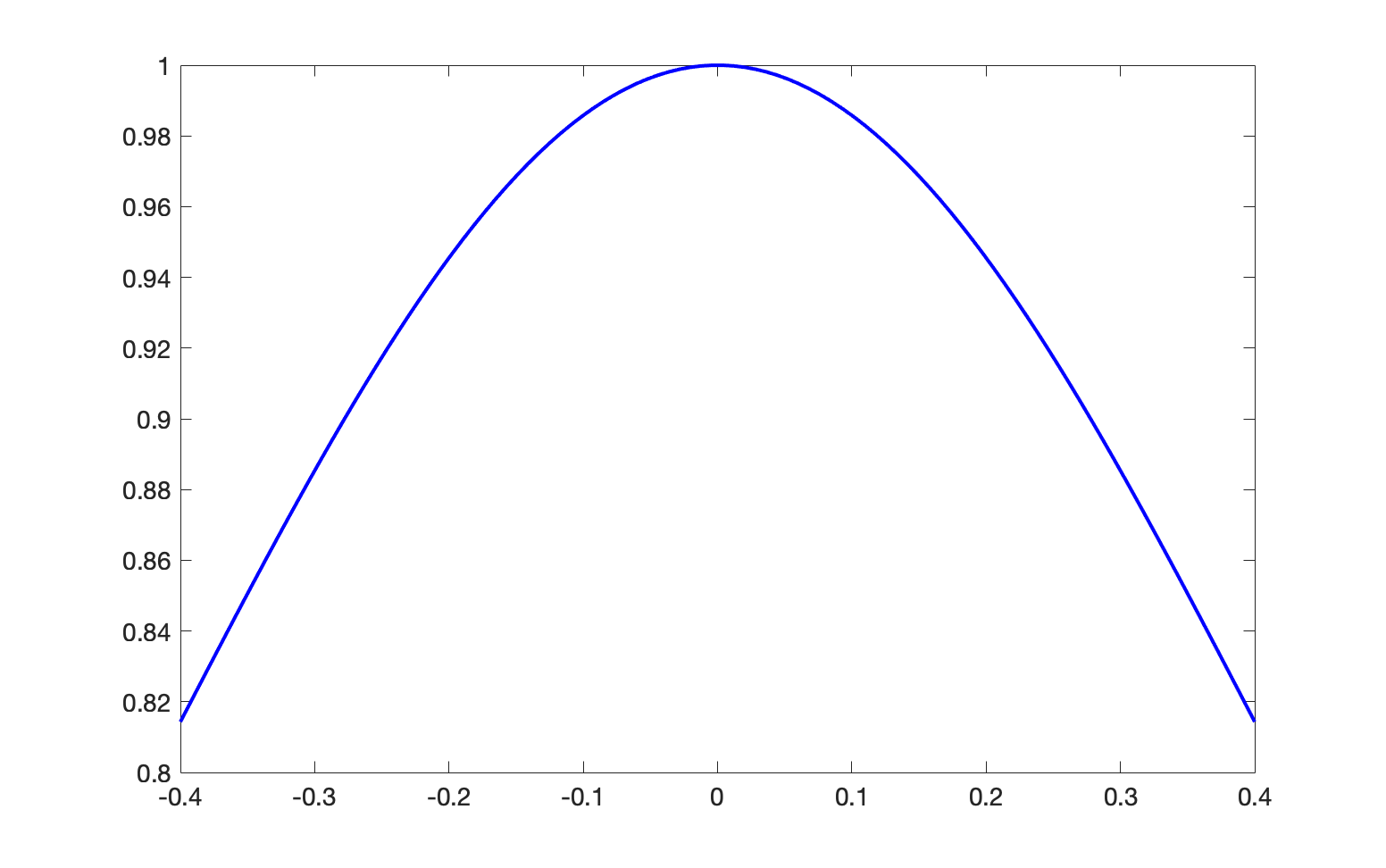}
 \end{tabular}
 \caption{Left: Plots of $q_1(x)\eqdef \max_\tau \hat K_0(x) + \min(0.2,(1-\tau^2))\abs{ \hat K_1(x)}$ (in red) and $q_2(x)\eqdef \max_\tau \hat K_2(x)$ (in blue), Middle: Zoom in  of $q_1(x)$, Right: Zoom in of  $q_2(x)$. \label{figs:proof} }
  \end{center}
\end{figure}

%
%
%
%

\section{Comparison with Lasso and C-BP}

\label{sec:applis}
\setlength{\tabcolsep}{5pt}


\subsection{Relationship to continuous basis pursuit}

Continuous basis pursuit (C-BP)~\cite{ekanadham2011recovery,ekanadham2014unified} is an off-the-grid approach to recovering \textit{positive} measures:
$$
\min_{(a,b)\in \RR_{\geq 0}^n\times \RR^n} \frac12 \norm{y-\Phi_X a - \Phi_X' b}^2 + \lambda \norm{a}_1 \quad \text{s.t.} \quad \abs{b}\leq \frac{h}{2} a,
$$
where $h$ is the grid size.
While our SR-Lasso problem is directly formulated as a group-Lasso problem, 
this C-BP can be re-formulated as a Lasso problem with positivity constraint
\begin{equation}\label{eq:cbp}
\min_{(r,l)\in \RR_{\geq0}^n\times \RR_{\geq0}^n} \lambda \norm{r}_1+\lambda \norm{l}_1 + \frac12 \norm{y- \Aa_h\begin{pmatrix}
r\\ l
\end{pmatrix}}^2
\end{equation}
after the change of variable $a_i = r_i+l_i$ and $b_i = \frac{h}{2} r_i-\frac{h}{2} l_i$, and 
$$
\Aa_h \eqdef \begin{pmatrix}
A+\frac{h}{2} B & A-\frac{h}{2} B
\end{pmatrix}.
$$

Following~\cite{duval2017sparse}, it is possible to study the support stability of this problem by computing the standard dual certificate of the Lasso problem, which is defined as $\eta_V \eqdef \Phi^* \Gamma_x^{\dagger,*} \binom{1_n}{0}$.
In \cite{duval2017sparse}, it was shown that a sufficient condition (but this is in fact necessary by standard Lasso arguments) for support stable recovery of the measure $\sum_{i\in J} \alpha_i \delta_{x_i+ s_i}$ where $s_i$ is some sufficiently small shift and $\abs{J}=n$, is that
\begin{equation}\label{cond:ICh}
\max_{i\in J^c} \eta_V(x_i) \pm \frac{h}{2} \eta_V'(x_i) <1.\tag{IC$_h$}
\end{equation}
By construction, $\eta_V(x_i) = 1$ and $\eta_V'(x_i) = 0$ for $i\in J$, so that condition~\eqref{cond:ICh} is in general false when $h$ is sufficiently small, since by Taylor expansion,  for $i\in J$,
\begin{align*}
&\eta_V(x_i+h) - \frac{h}{2} \eta_V'(x_i+h)
= 1+\frac{h^3}{12} \eta_V'''(x_i) + \Oo(h^4),  \\
\qandq&
\eta_V(x_i-h) +\frac{h}{2} \eta_V'(x_i+h)
= 1-\frac{h^3}{12} \eta_V'''(x_i) + \Oo(h^4).
\end{align*}
If $\eta_V'''(x_i)\neq 0$, then one of these expressions is greater than 1 for $h$ sufficiently small. Indeed, for deconvolution,  $\eta_V'''(x_i)\neq 0$ whenever $x_i$ are sufficiently separated.
Hence, there are two main advantages to SR-Lasso, compared to C-BP
\begin{enumerate}
\item \emph{Generality:} SR-Lasso can handle \emph{arbitrary signed signals}, whereas C-BP can only handle positive signals. Moreover, it is straightforward to extend SR-Lasso to the higher dimensional setting, whereas the extension of the convex formulation \eqref{eq:cbp} to higher dimensions is unclear.

\item \emph{Sparsistency:} One of the highlights of  Theorem \ref{thm_G} is that if the underlying signal is supported on the grid, it can be stably recovered. Moreover, for any grid size $h$, one can stably resolve signals off-the-grid to some extent, how far one can move inside the grid depends on the curvature of the kernel. On the other hand, in the case of continuous basis pursuit, even if the true signal lies on the grid, the problem is unstable when the grid becomes too fine since the condition \eqref{cond:ICh} is false in general.

\end{enumerate}
\subsection{Numerical Comparison}

The previous section has demonstrated that SR-Lasso possesses better theoretical sparsity properties compared to C-BP (and of course, to the basic Lasso as well). We now illustrate, using synthetic numerical simulations, that this also translates into practical gains in terms of support estimation.

\paragraph{Evaluation metric.}
The problem of source position estimation is infinite-dimensional (inverting the operator $\Phi$ in~\eqref{eq:fwd-model}), rendering usual finite-dimensional performance criteria (such as the $\ell^2$ or $\ell^\infty$ norms on discretized recovered coefficients) less meaningful.
Each field of application typically develops its own measures of performance, often intertwining errors in source amplitude and position while penalizing both false positives and negatives in a somewhat ad-hoc manner. For an in-depth discussion in the context of single molecule localization for fluorescence imaging, see~\cite{sage2015quantitative}.
A more mathematically rigorous approach to assess recovery quality is to view the problem as one of measure estimation, recovering $\mu_0$ from a noisy observation $y \approx \Phi \mu_0$. In our case, to concentrate on the intrinsic capabilities of the methods for estimating the source support despite discretization errors, we do not add measurement noise to the observations, so that $y = \Phi \mu_0$. In light of this measure estimation problem, it makes sense to compute the estimation error using some norm or distance between measures.
For simplicity, we employ here the Maximum Mean Discrepancies (MMD) norms \cite{gretton2012kernel}, which are Hilbert kernelized norms over the space of measures. An alternative -- and arguably more complex -- approach is to use the Optimal Transport distance, as advocated in~\cite{denoyelle2018theoretical,denoyelle2021optimal}.

For some positive kernel $k(x,y)$, its associated MMD norm of a complex-valued measure $\xi$ is defined as 
$$
    \| \xi \|_k^2 := \int k(x,y) \text{d} \xi(x) \bar \xi(y).
$$
Provided that $k$ is universal, this defines a norm, which metrizes the weak$^*$ convergence on the space of measure, thus offering a way to asses the estimation error on both amplitude and position of the sources. 
We consider here the Laplace kernel $k(x,y) = \exp(-\norm{x-y})$, which in practice has better behavior that the more usual Gaussian kernel (because it has a slower spacial decay), and define our error metric as
$$
	D(\mu,\nu) \triangleq \| \mu -  \nu\|_k^2.
$$
For all measures $\mu,\nu$, we have $D(\mu,\nu)\geq 0$ and $D(\mu,\nu) = 0$ if and onlly if $\mu=\nu$.
An interesting connexion with the initial estimation problem is that, assuming $y=\Phi(\mu_0)$, then the standard Euclidean loss minimized in Lasso-type problems such as~\eqref{eq:blasso} can be re-written as a MMD norm since
$$
	\norm{\Phi \mu -y}_\Hh^2 = \| \mu-\mu_0 \|_{k_0}^2
	\qwhereq
	k_0(x,x') \triangleq \phi(x)^\top \phi(y).
$$
Since the Laplace kernel \( k \) that we use for error assessment typically has a much slower Fourier decay than the measurement kernel \( k_0 \) (which, in our experiments, is either an ideal low-pass filter or a Gaussian kernel), the discrepancy \( D \) evaluates the super-resolution capability of the method, i.e., its ability to recover ``higher order moments''.

\paragraph{Influence of $\tau$.}

In Figure \ref{fig:tau}, we show the MMD error for different values of $\tau$. We use here finite dimensional Fourier frequency measurements $\phi(x)=(e^{2\imath\pi k x})_{k=-f_c}^{f_c}$. In this experiment, we observe samples of the Fourier transform up to $f_c=4$ of a sparse measure made up of 2 spikes, whose positions  are spaced $0.3h$ off-the-grid, where $h = 1/N$ is the grid size and $N=20$.  At $0h$ (true spikes are on the grid), the error is smallest for $\tau=0$, and unsurprisingly, as $h$ increases, the optimal $\tau$ value increases. Note however that the error for choosing $\tau=1$ in the on-the-grid case is only slightly higher than choosing $\tau=0$, and in practice, we find that choosing $\tau=1$ is effective.

\begin{figure}
\begin{center}
\begin{tabular}{ccccccc}
&$0h$&$0.1 h$&$0.2h$&$0.3h$\\
\rotatebox{90}{\hspace{1cm}Error}&
\includegraphics[width=0.2\linewidth]{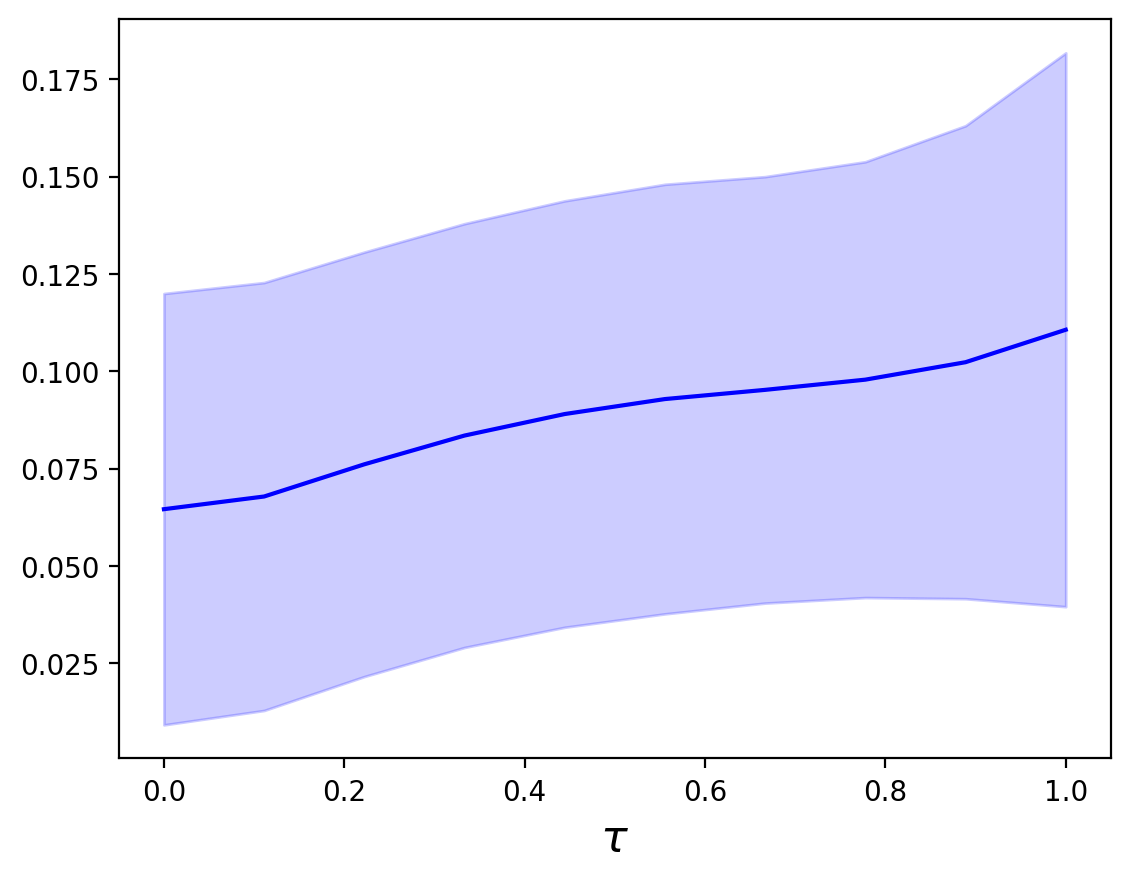}
&
\includegraphics[width=0.2\linewidth]{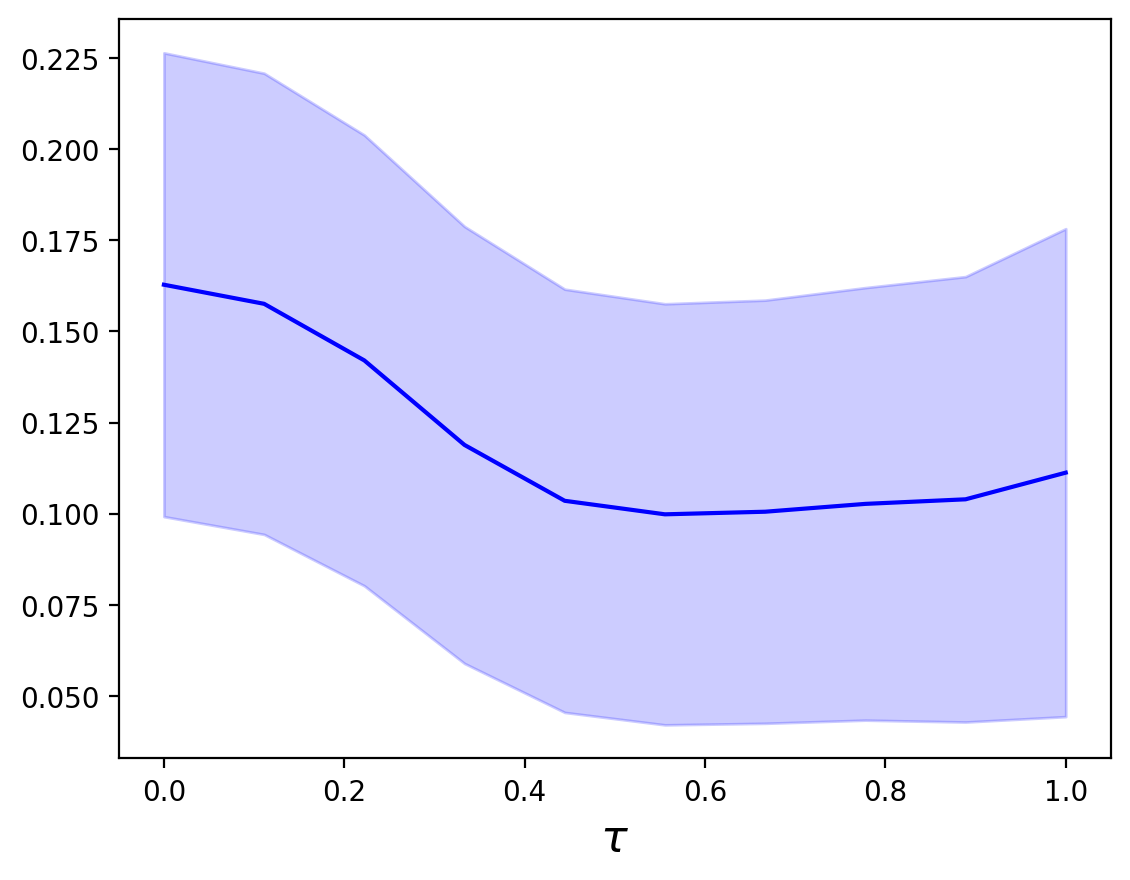}&
\includegraphics[width=0.2\linewidth]{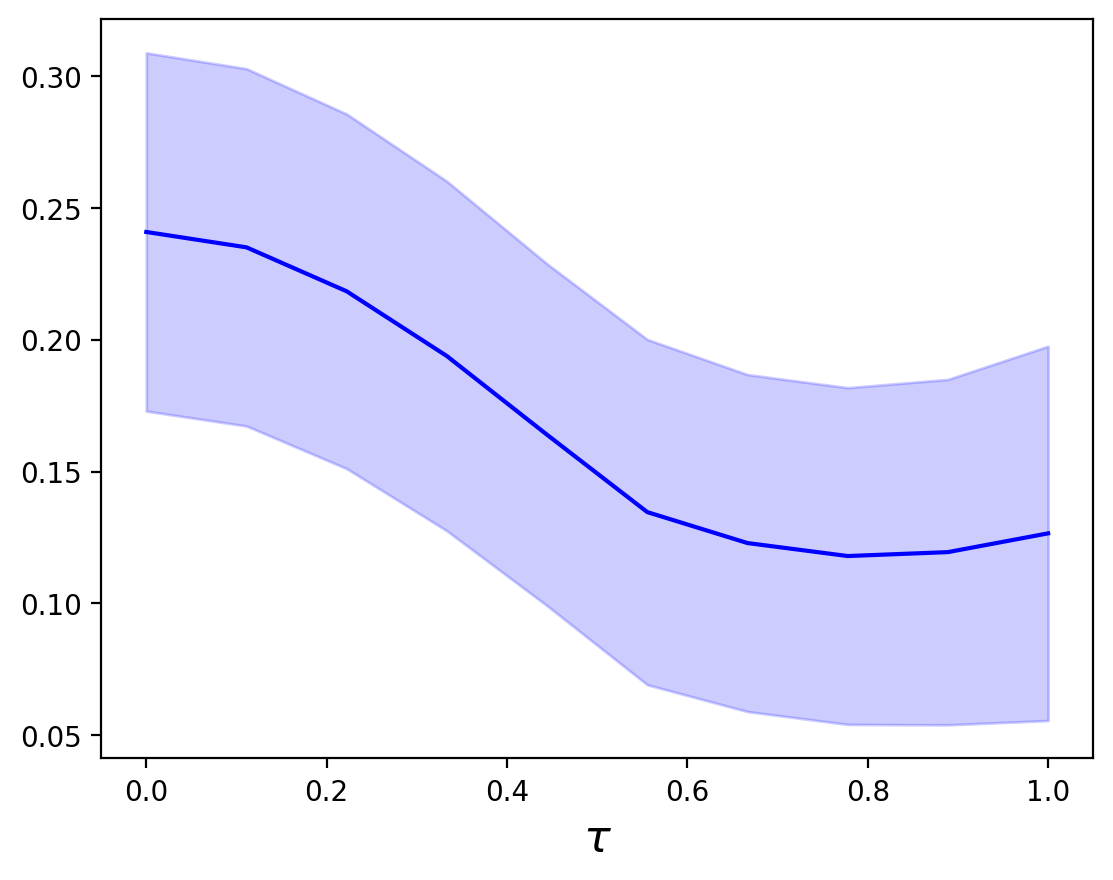}&
\includegraphics[width=0.2\linewidth]{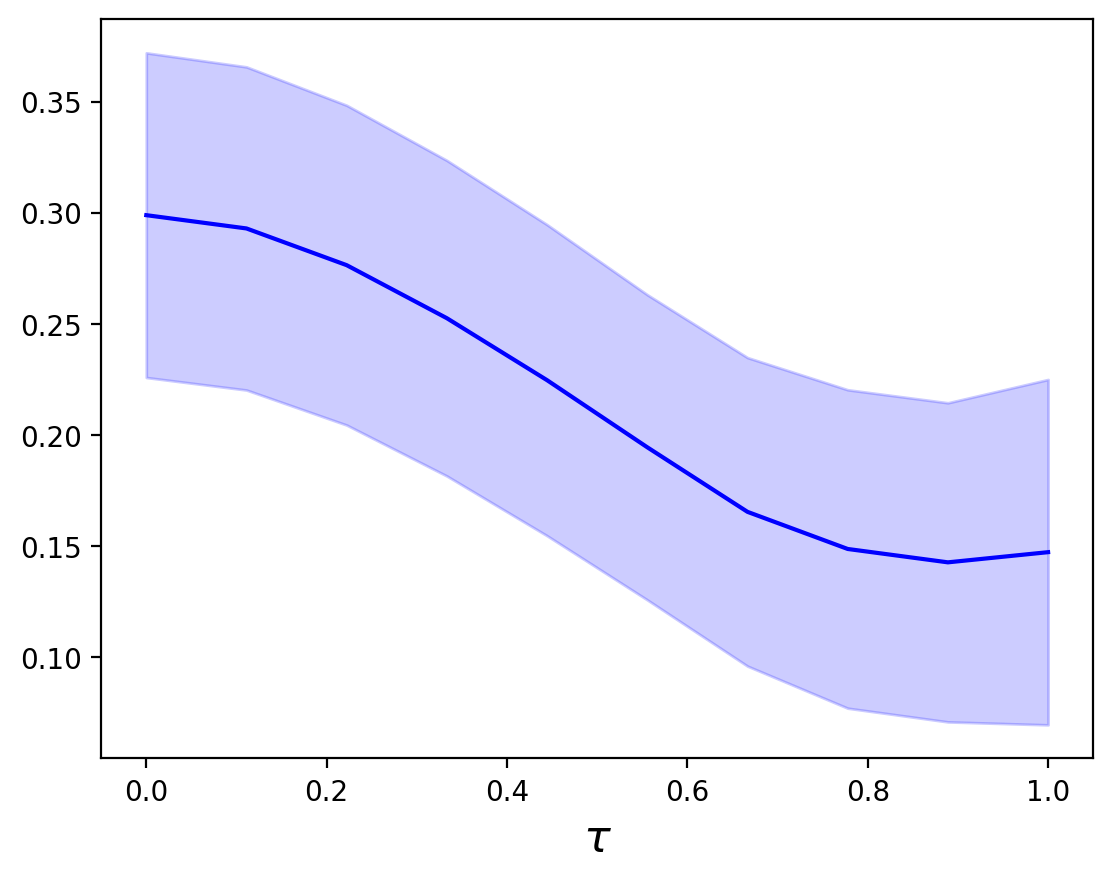}
\end{tabular}
\end{center}
\caption{1D Fourier. Showing effects of $\tau$. Grid size $N=20$ and fixed $\lambda=0.1$. \label{fig:tau}}
\end{figure}

\paragraph{Comparison against continuous basis pursuit (c-BP) and Lasso}

In Figure \ref{fig:cBP_cert}, we consider the reconstruction of 2 spikes from samples of the Gaussian operator, that is, $\phi(x) =\pa{ \exp(-(x-t_j)^2/\sigma^2)}_{j=1}^m$. We let $\ens{t_j}$ be 50 uniformly spaced points in $[0,1]$. For this experiment, we observe that the certificates for both c-BP and SR-Lasso are nondegenerate when the grid size is not too large (around $N=10$), but the c-BP certificate becomes degenerate as $N$ increases. Accordingly, one can observe that when the c-BP certificate become degenerate, one no longer has sparsistent reconstructions. In contrast, SR-Lasso remains non-degenerate as $N$ increases and its reconstructions are sparsistent.

More generally, we show in Figure~\ref{fig:cBP_comp} the reconstruction error as a function of $\lambda$, where the data is again samples of the Gaussian operator, the the ground truth are randomly generated positive sparse measures, made up of 4 spikes spaced at least 0.2 apart, and 0.2h off-the-grid. The results are averaged over 10 randomly generated signals.
In Figure \ref{fig:2dFourier}, we compare against Lasso for recovering complex signed measures from samples of the Fourier transform of a 2D sparse measure. In Figure \ref{fig:3d}, we compare against Lasso in the case where 
$$
	\phi(x_1,x_2,z) = \pa{\exp(-\norm{(x_1,x_2)-(t_{1,i}-t_{2,i})}^2/(2\sigma^2)) \exp(-t_{3,i} z)}_{i=1}^m.
$$
In all cases, one can observe the SR-Lasso significantly improves on the reconstruction error and has sharper support recovery.

%

\begin{figure}
\begin{center}
\begin{tabular}{cccc}
\includegraphics[width=0.24\linewidth]{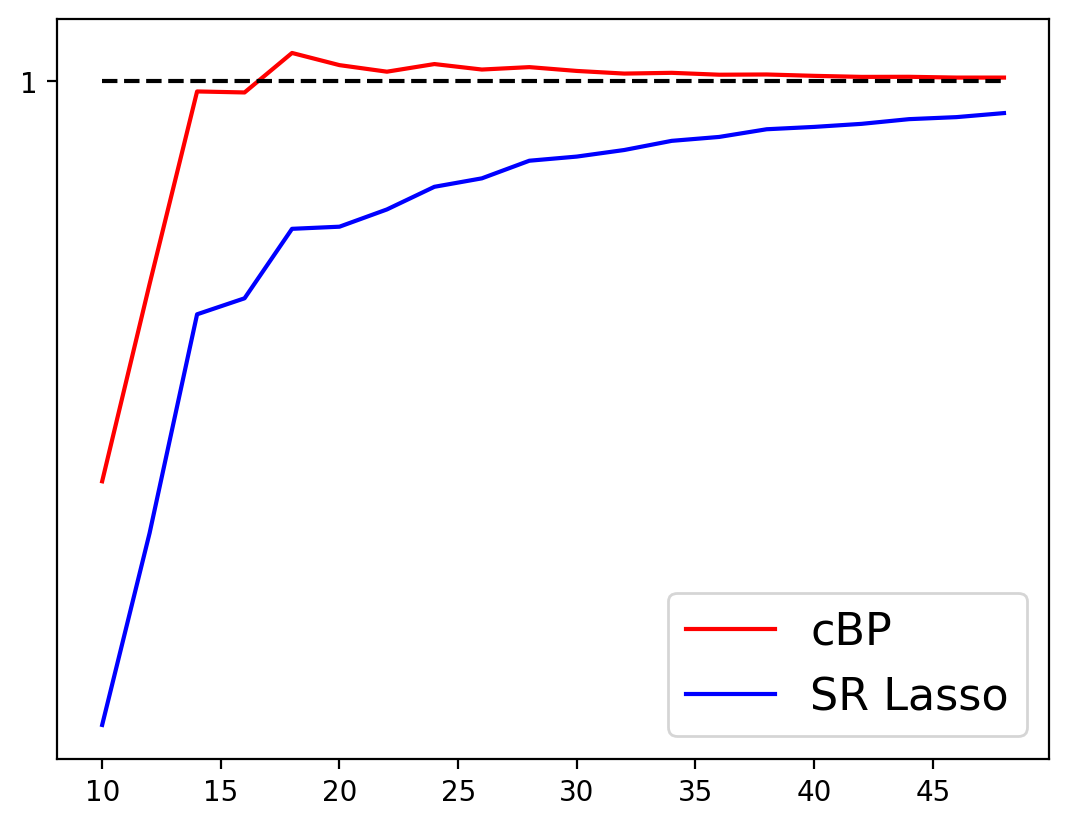}&
\includegraphics[width=0.24\linewidth]{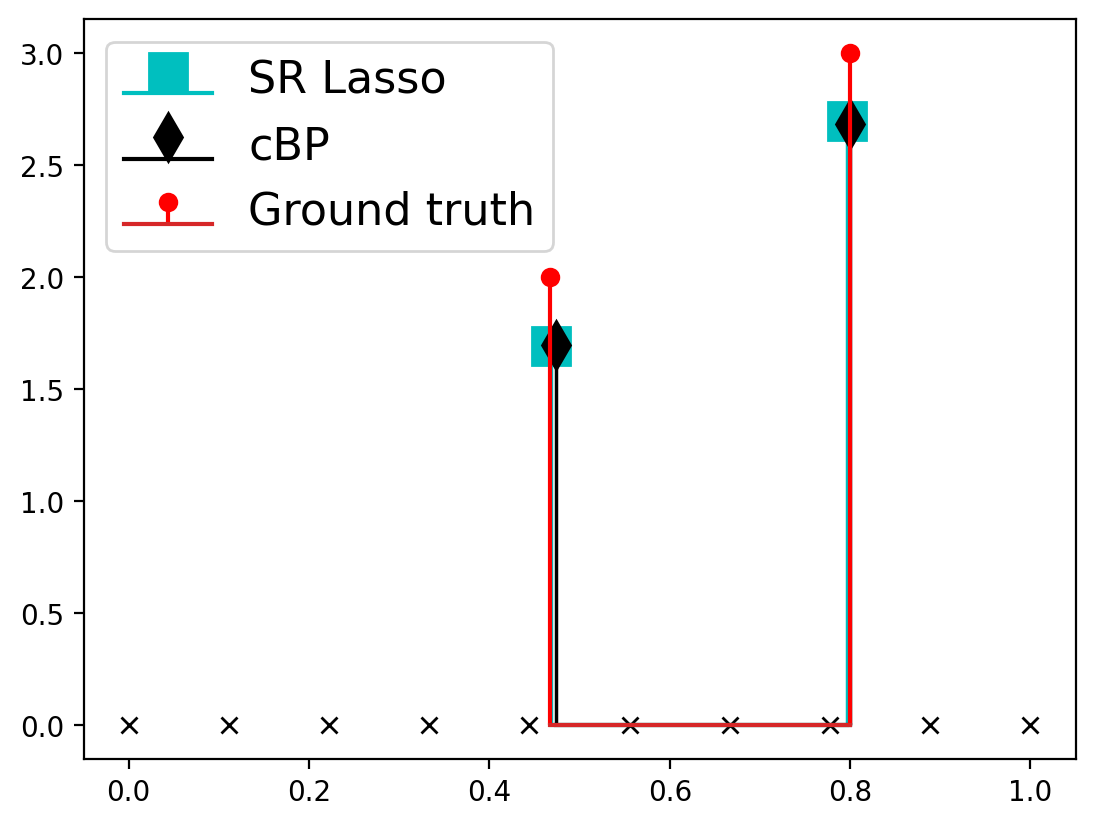}&
\includegraphics[width=0.24\linewidth]{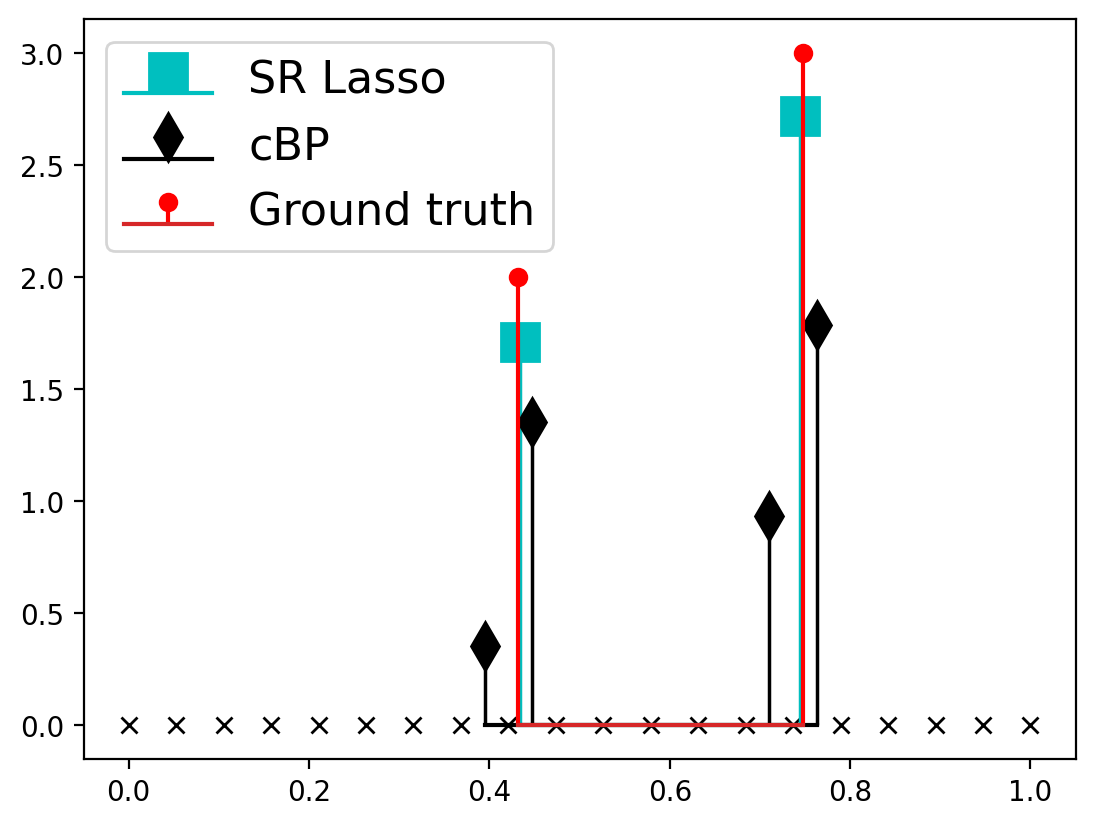}&
\includegraphics[width=0.24\linewidth]{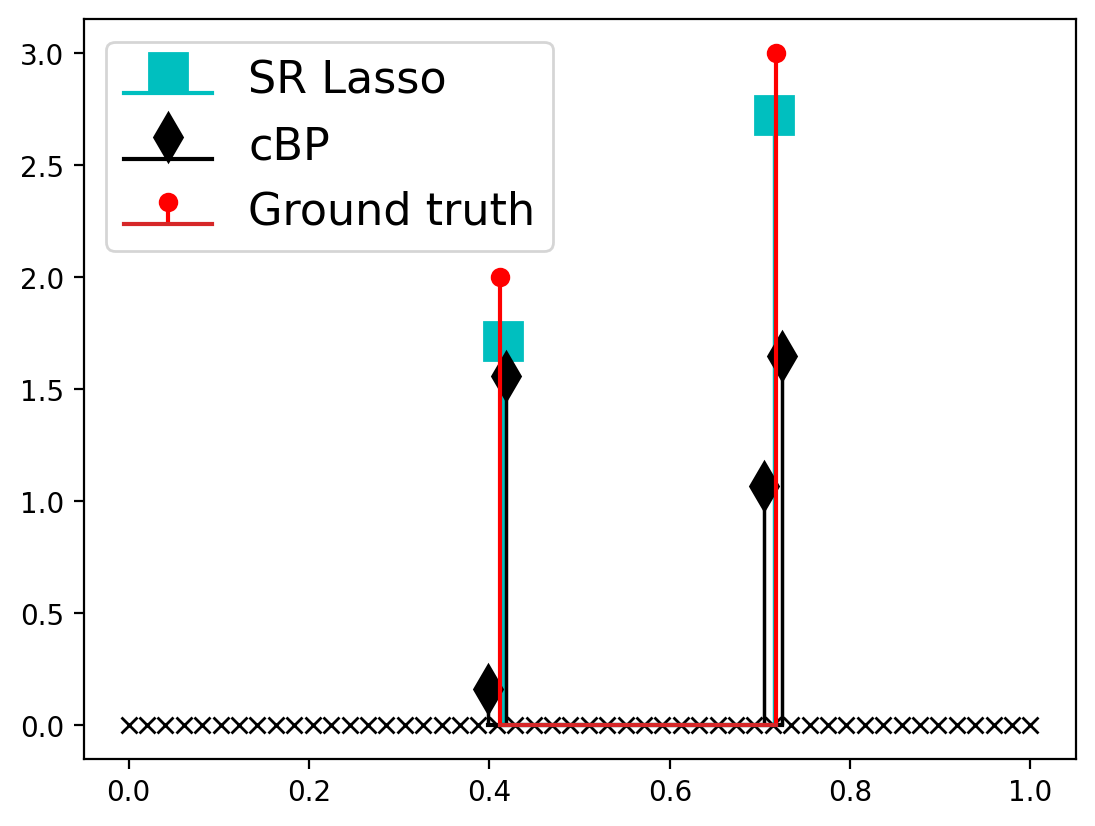}
\end{tabular}
\caption{
The left graph shows the maximum absolute value of the C-BP certificates and the SR-Lasso certificate outside the true support. For this example, one can observe that SR-Lasso remains nondegenerate as $N$ increases, and the reconstructions are sparsistent. For C-BP, the reconstruction is sparsistent and certificate is nondegenerate when the grid is sufficiently coarse ($N=10$), but for larger values of $N$, the certificates are degenerate and hence, the solution is not sparsistent for large $N$. \label{fig:cBP_cert}
}
\end{center}
\end{figure}

\begin{figure}
\begin{center}
\begin{tabular}{cccc}
&$N=20$&$N=30$&$N=40$\\
\rotatebox{90}{\hspace{2cm}Error}&
\includegraphics[width=0.3\linewidth]{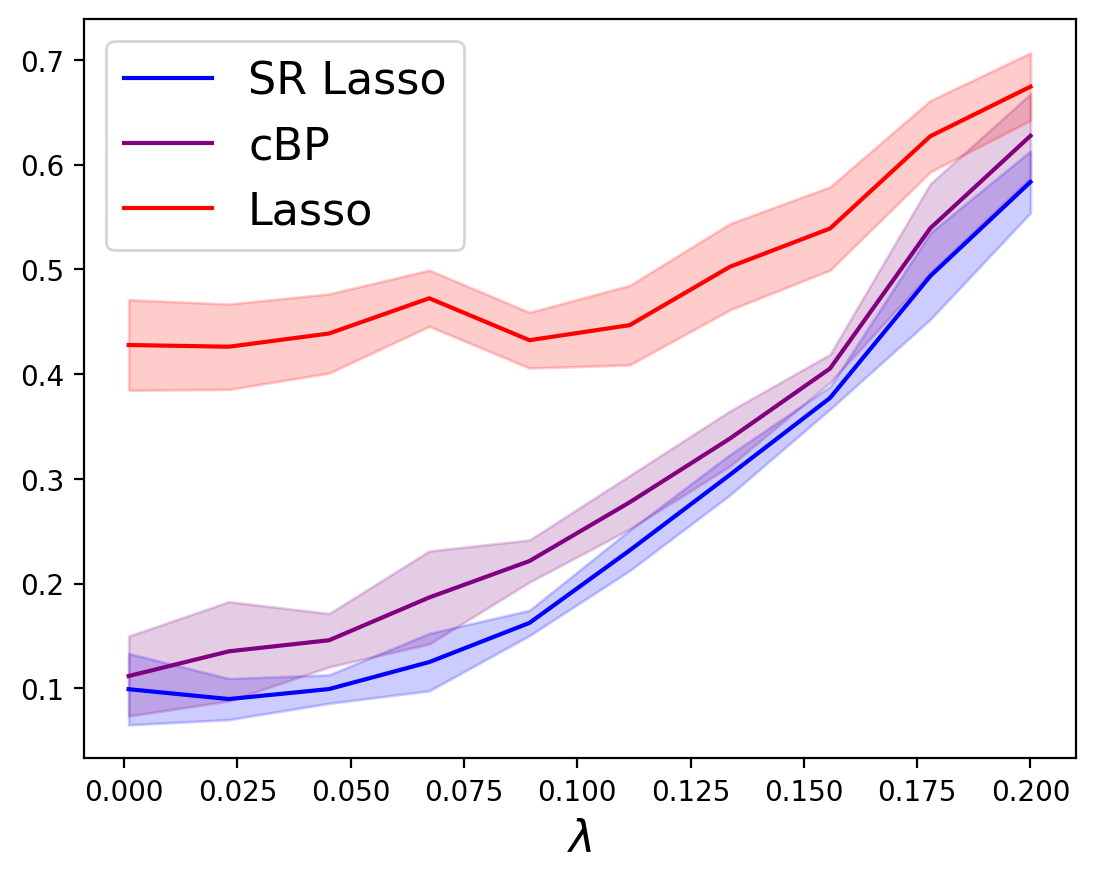}
&
\includegraphics[width=0.3\linewidth]{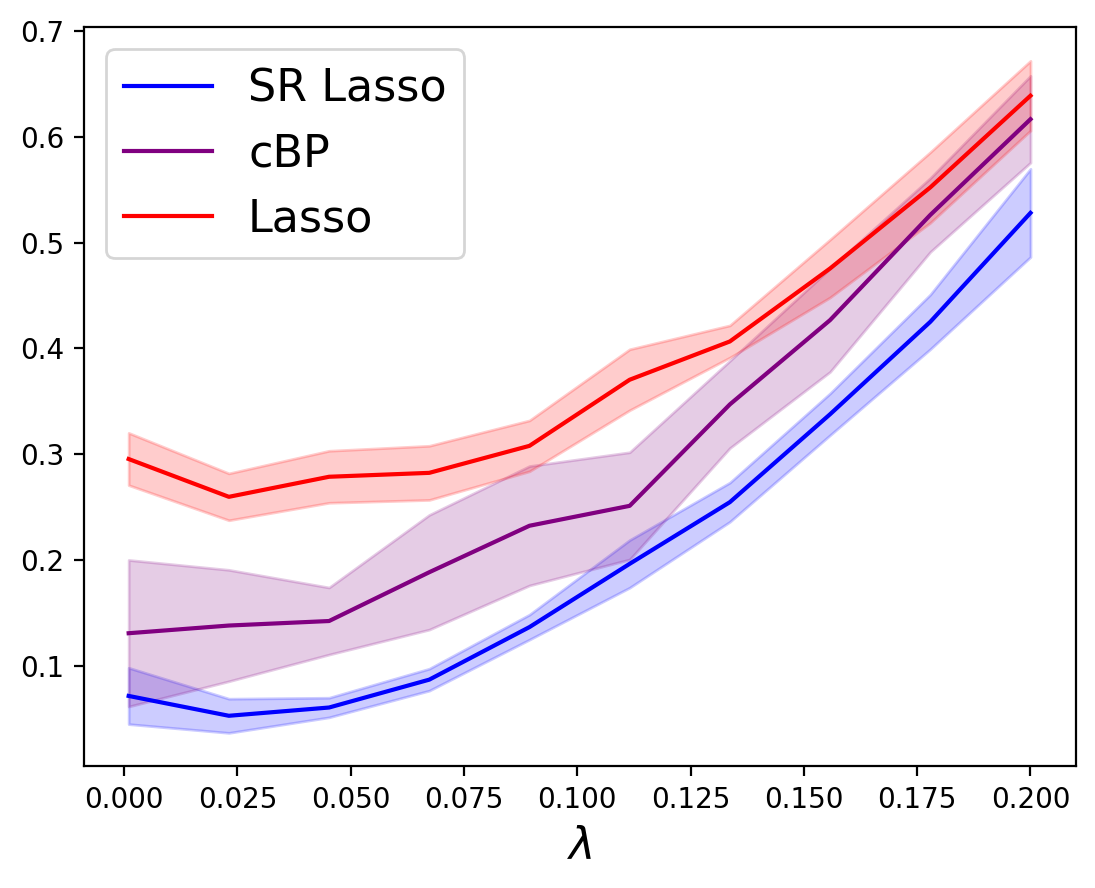}
&\includegraphics[width=0.3\linewidth]{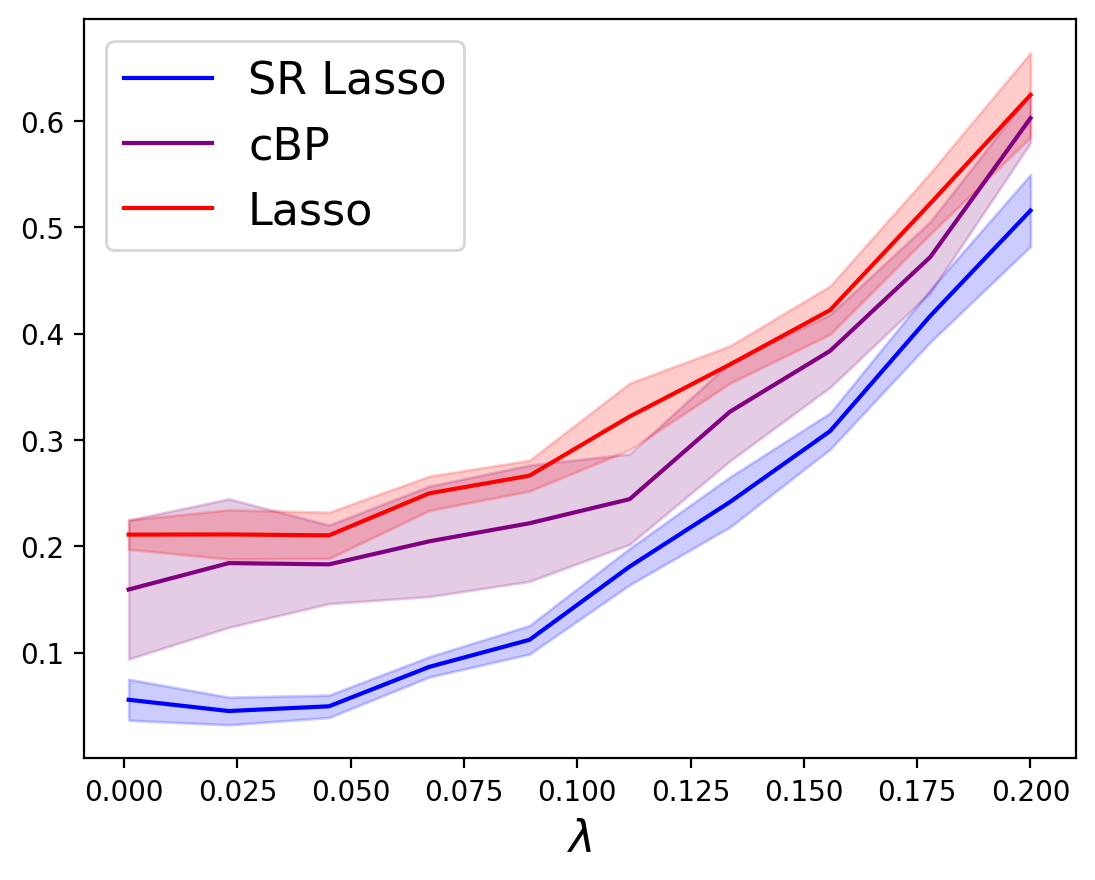}
\\
\rotatebox{90}{\hspace{2cm}Support}&
\includegraphics[width=0.3\linewidth]{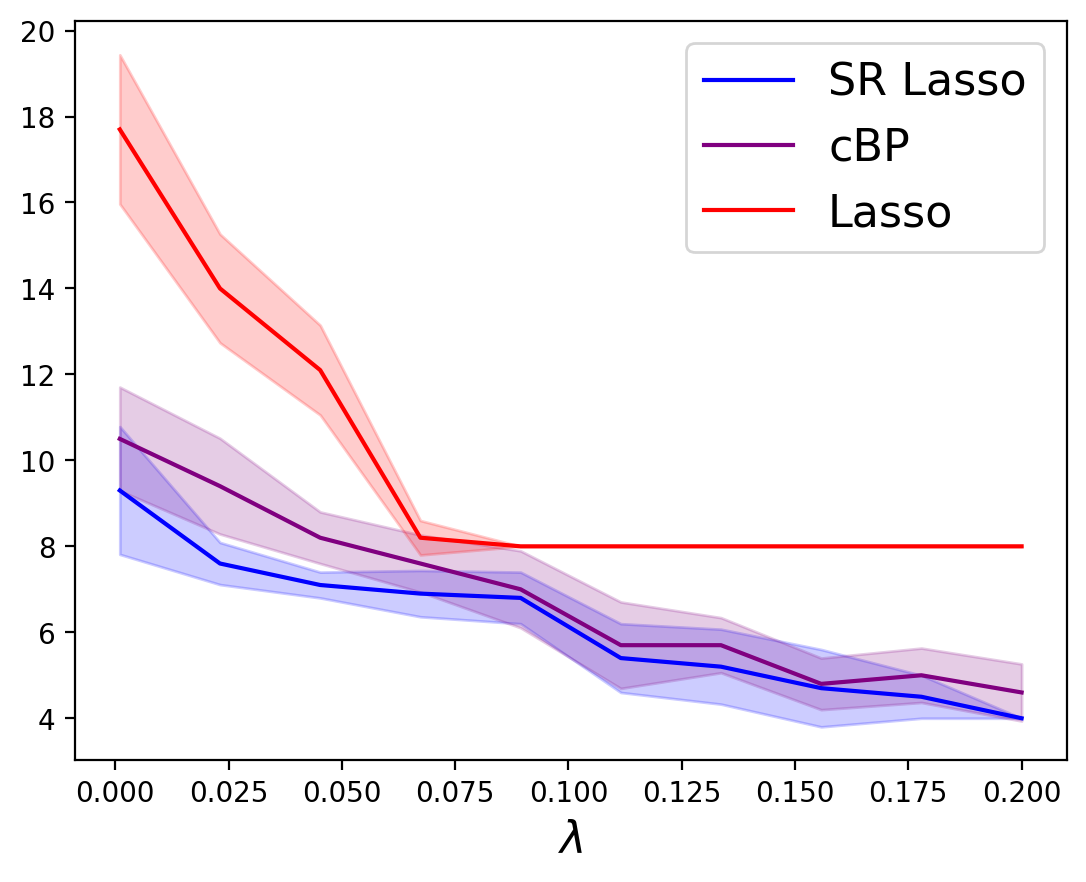}
&
\includegraphics[width=0.3\linewidth]{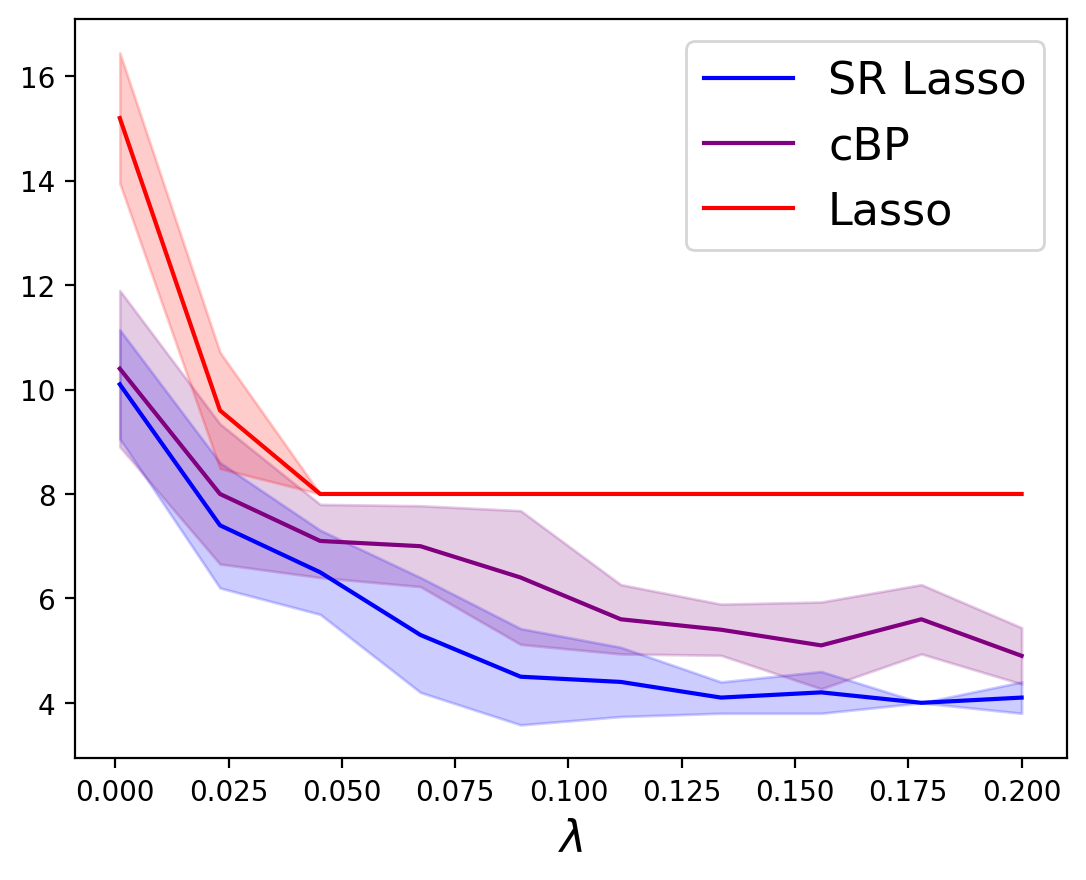}
&\includegraphics[width=0.3\linewidth]{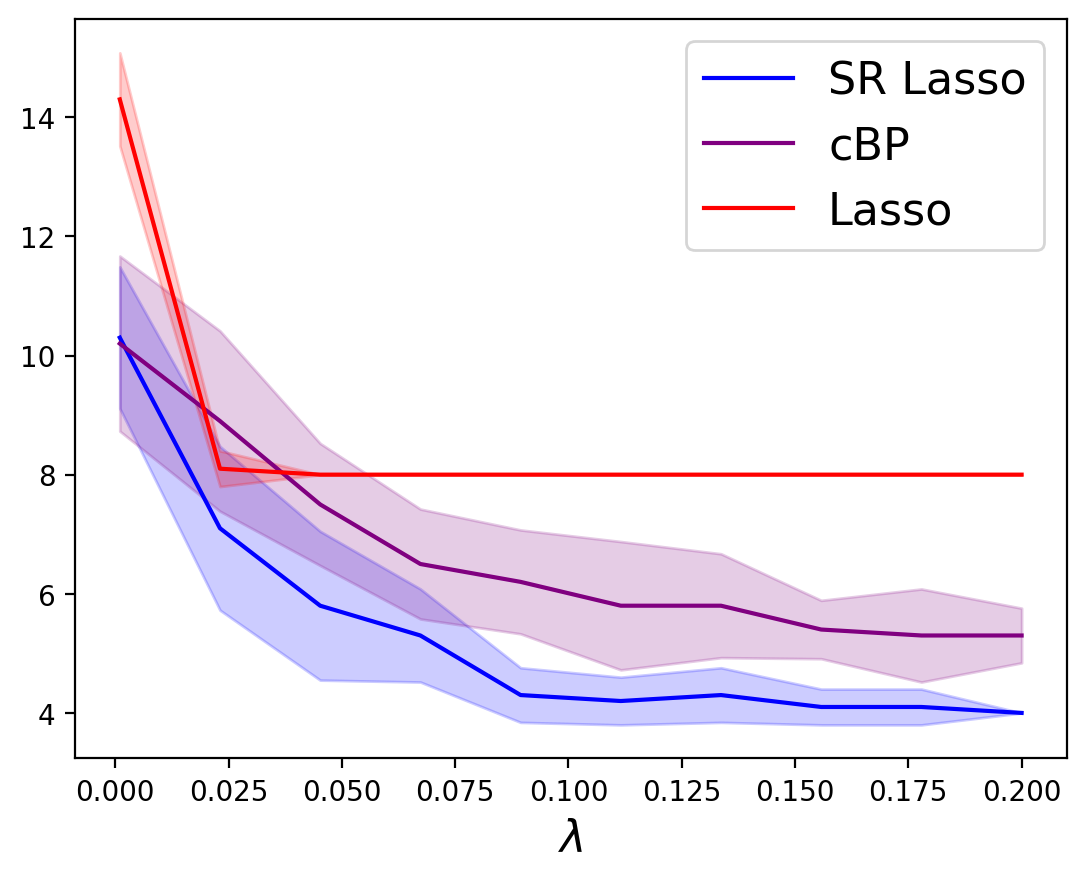}
\end{tabular}
\end{center}
\caption{1D comparison for the recovery of 4 positive spikes in the case of Gaussian sampling. \label{fig:cBP_comp}}
\end{figure}

\begin{figure}
\begin{center}
\begin{tabular}{cccc}
&$N=20$&$N=30$&$N=40$\\
\rotatebox{90}{\hspace{2cm}Error}&
\includegraphics[width=0.3\linewidth]{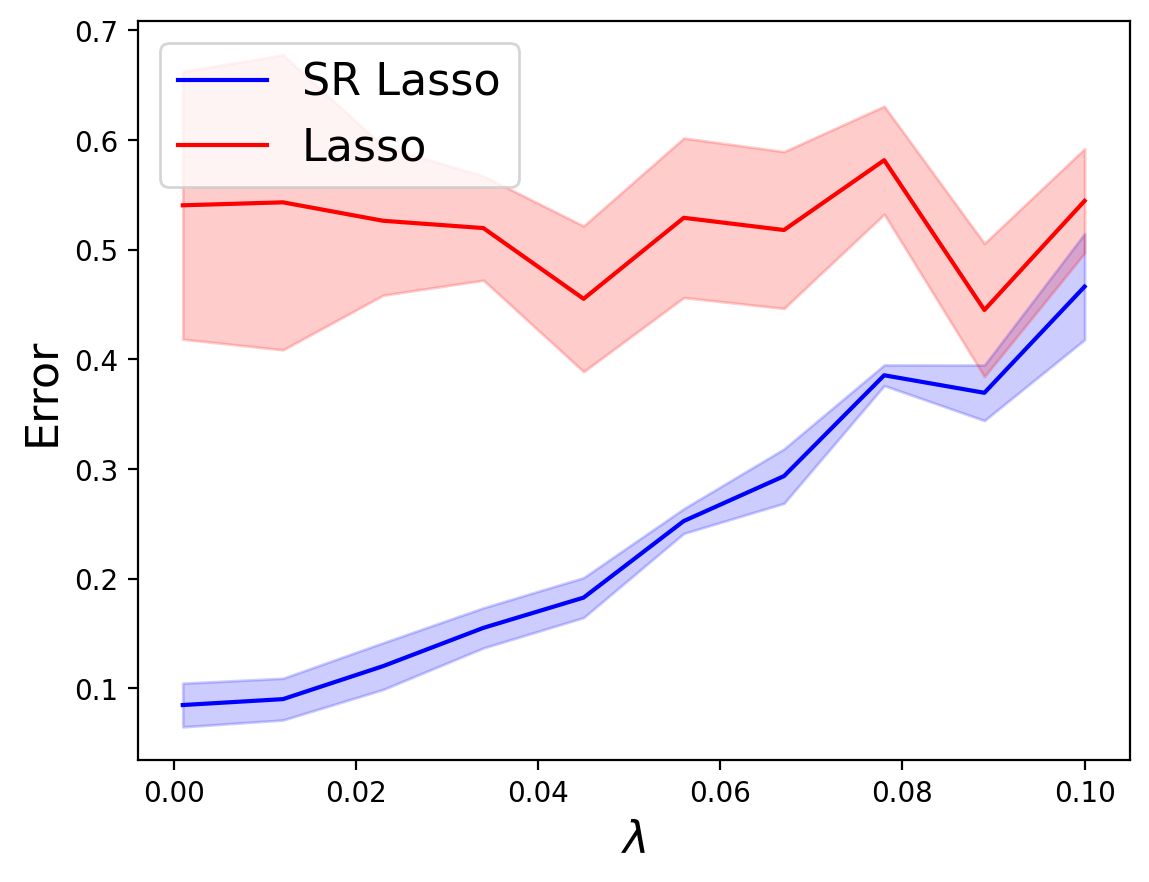}
&
\includegraphics[width=0.3\linewidth]{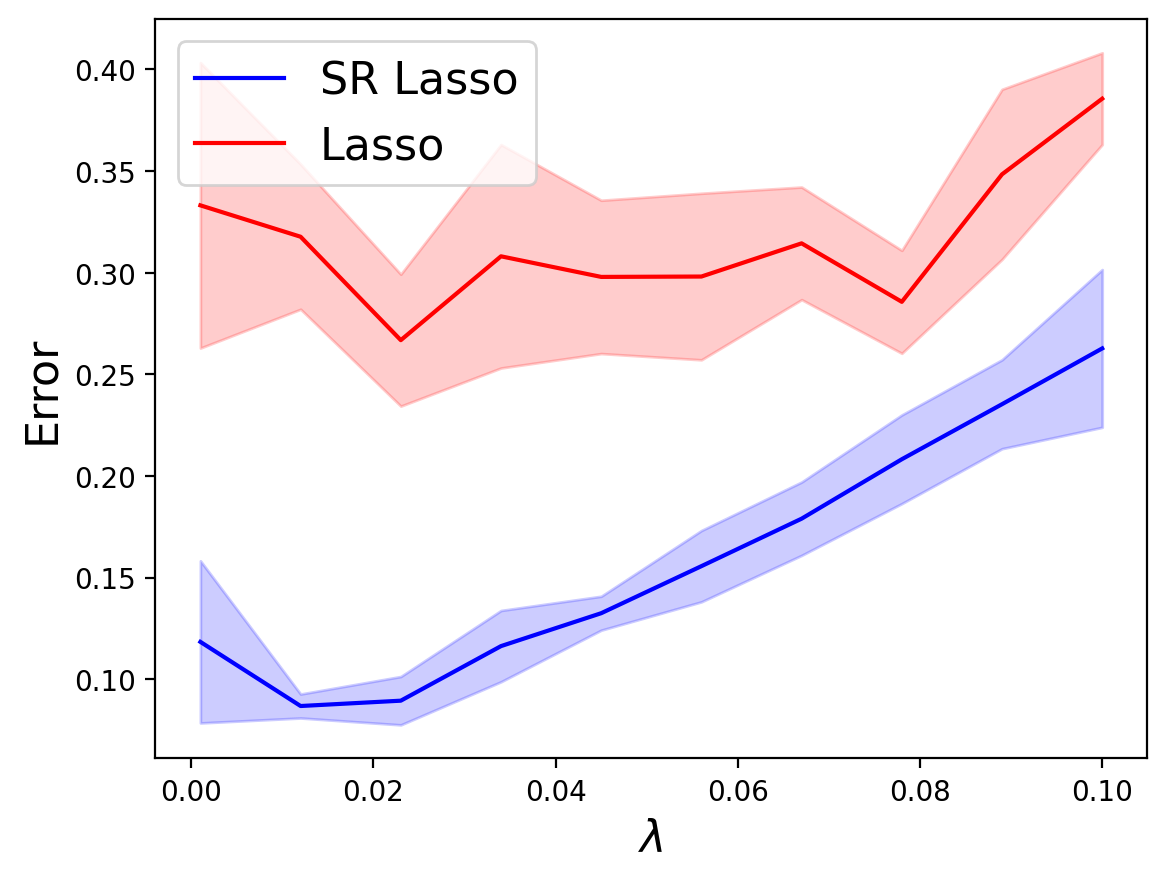}
&\includegraphics[width=0.3\linewidth]{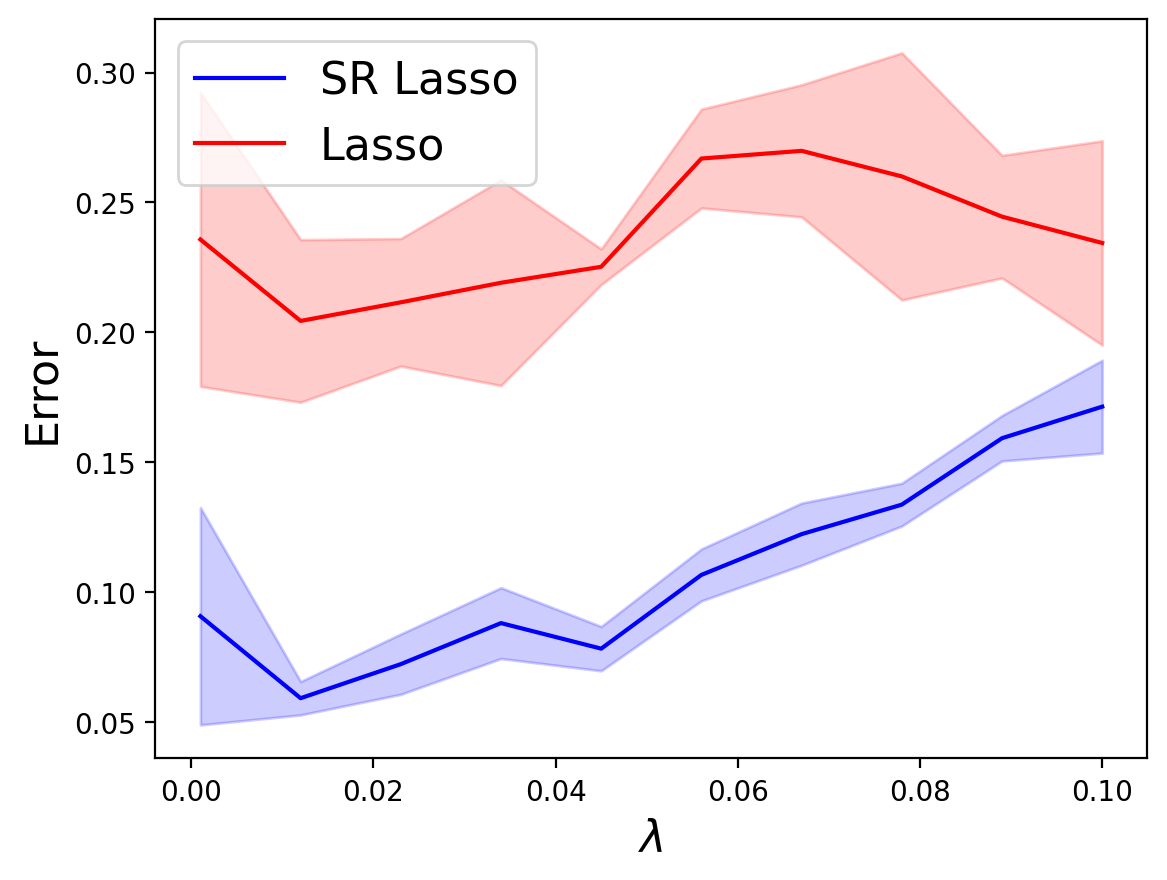}
\\
\rotatebox{90}{\hspace{2cm}Support}&
\includegraphics[width=0.3\linewidth]{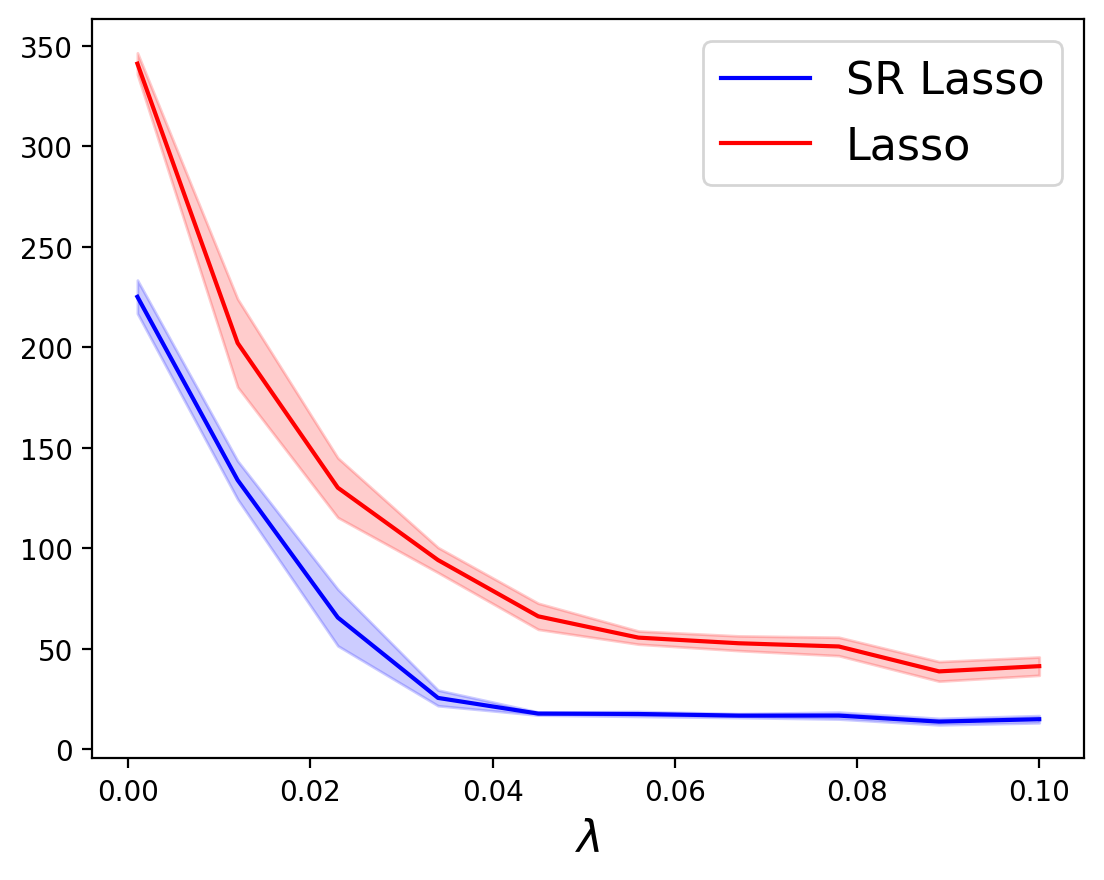}
&
\includegraphics[width=0.3\linewidth]{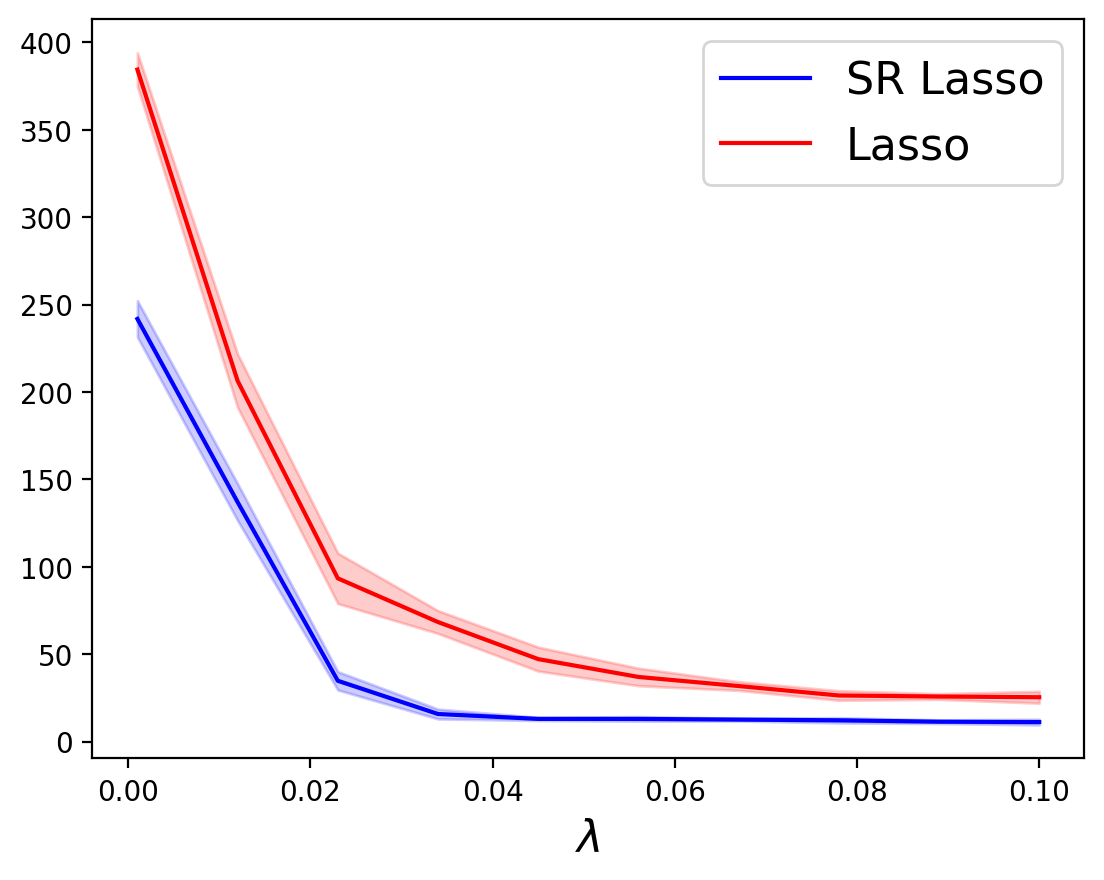}
&\includegraphics[width=0.3\linewidth]{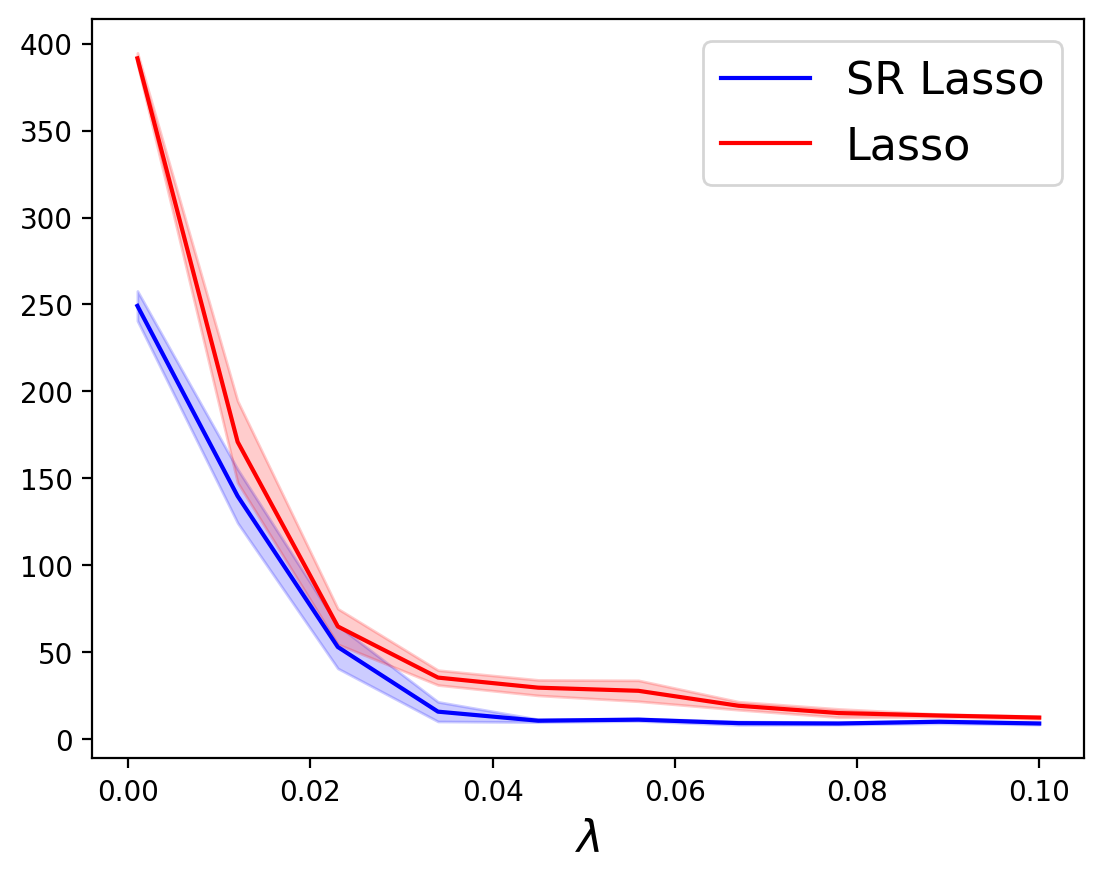}
\end{tabular}
\end{center}
\caption{2D comparison for the recovery of 3 signed spikes in the case of Fourier sampling.\label{fig:2dFourier}}
\end{figure}

\begin{figure}
\begin{center}
\begin{tabular}{cccc}
&$N=10$&$N=20$&$N=50$\\
\rotatebox{90}{\hspace{2cm}Error}&
\includegraphics[width=0.3\linewidth]{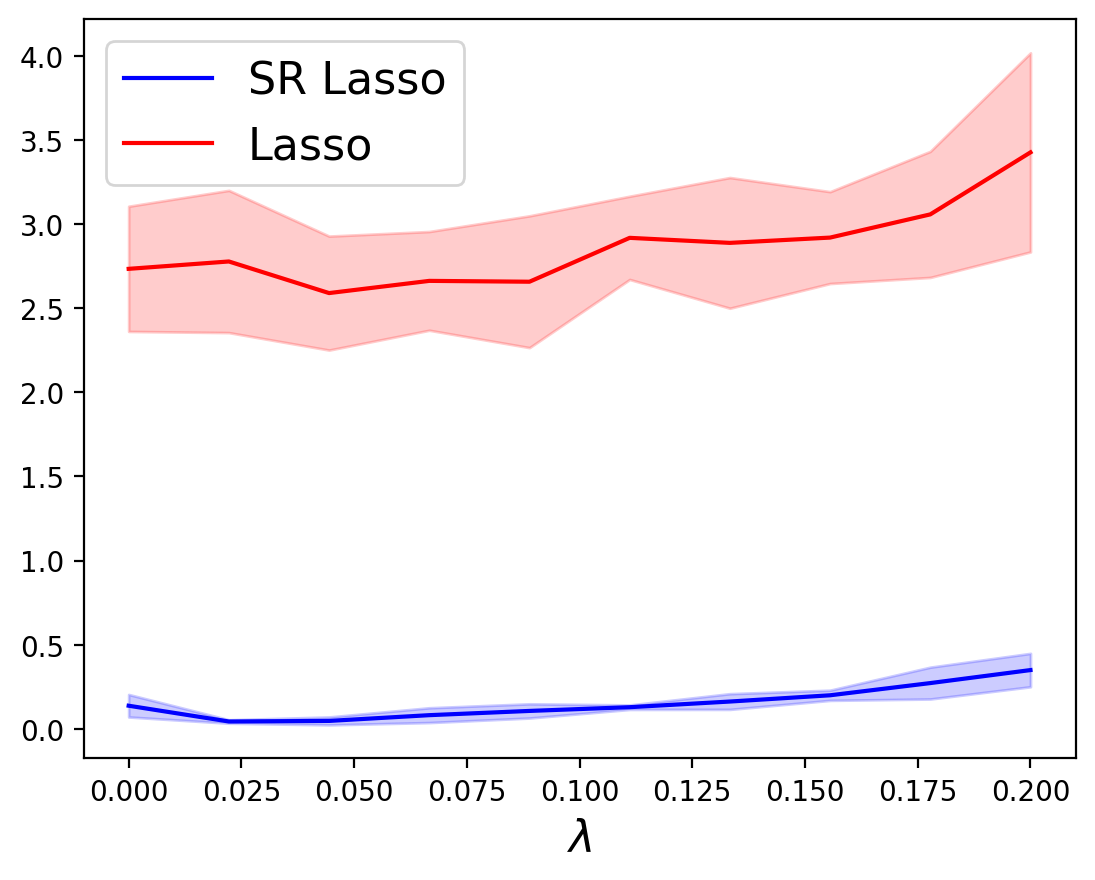}
&
\includegraphics[width=0.3\linewidth]{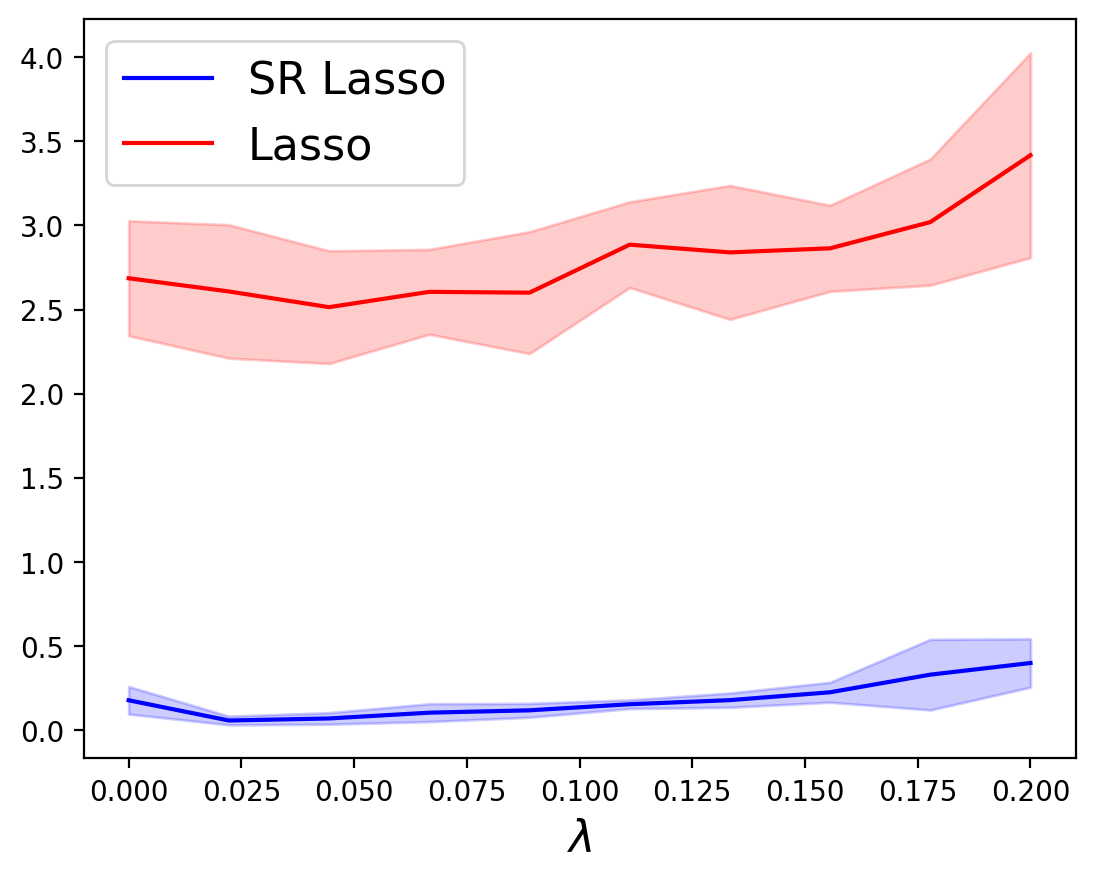}
&\includegraphics[width=0.3\linewidth]{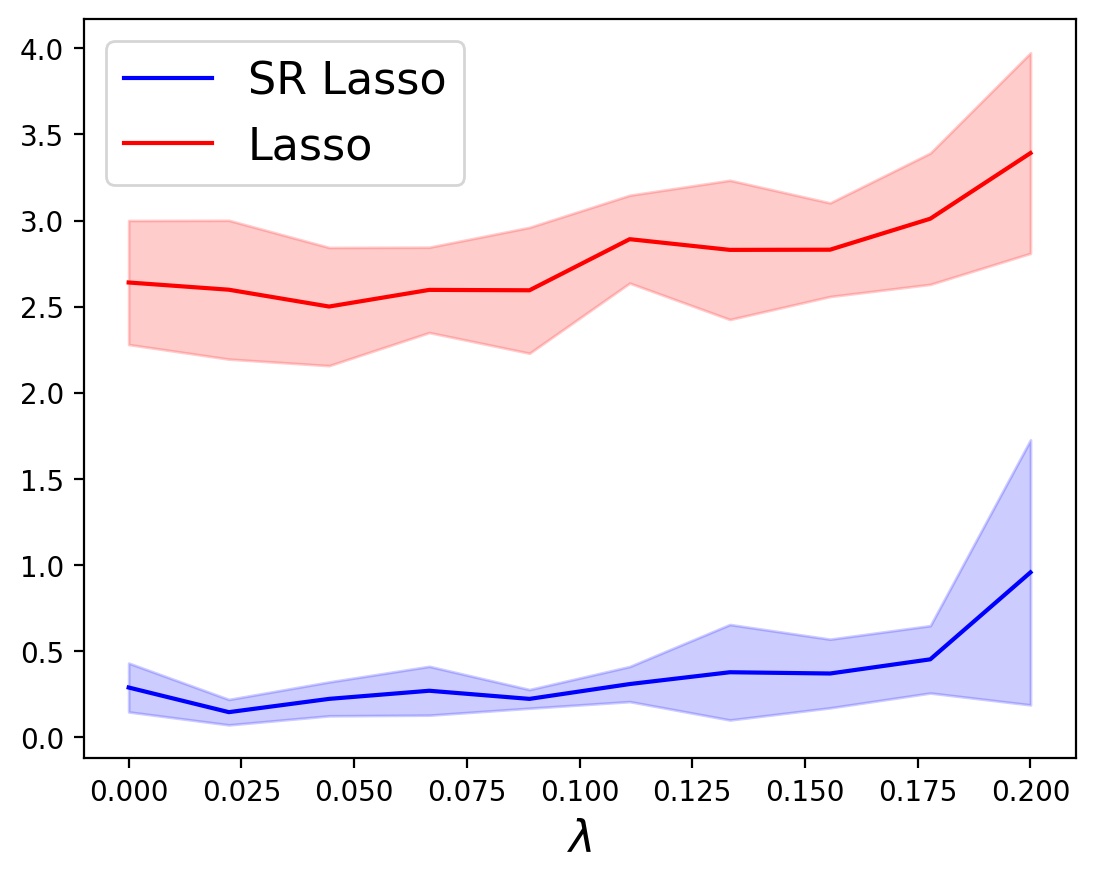}
\\
\rotatebox{90}{\hspace{2cm}Support}&
\includegraphics[width=0.3\linewidth]{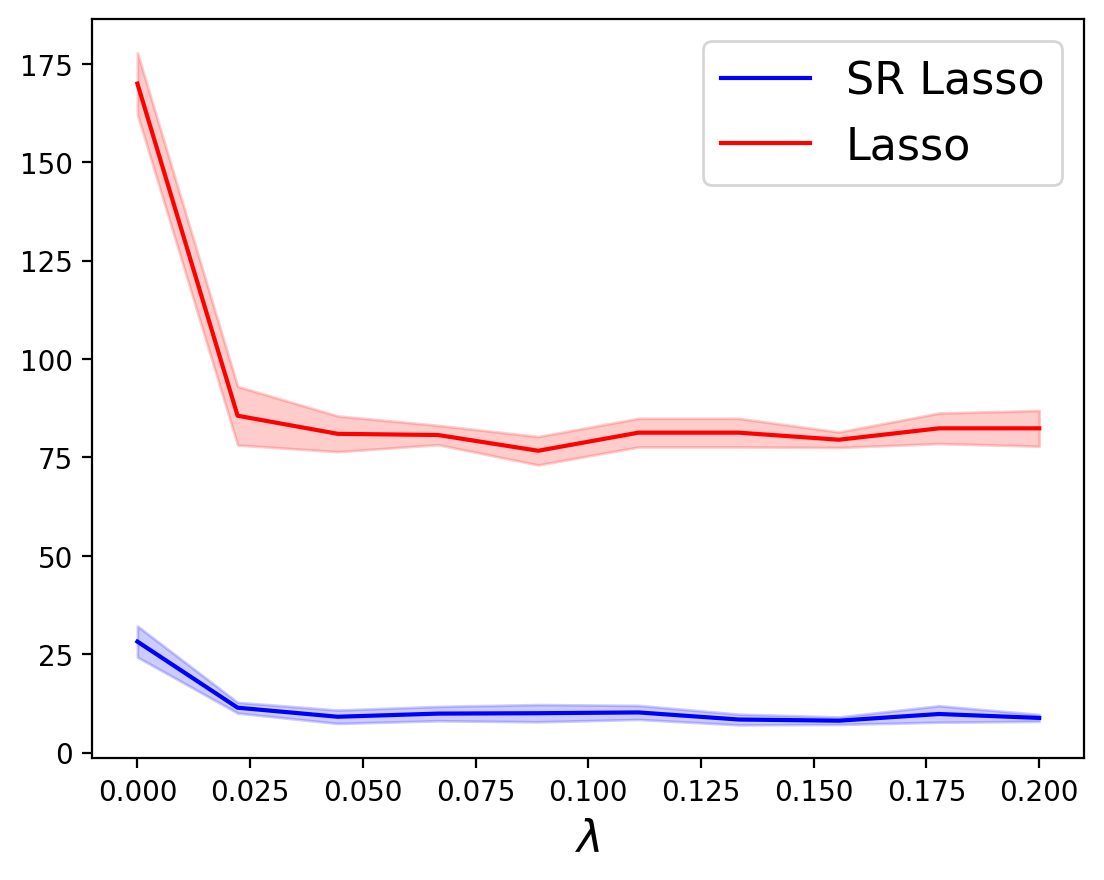}
&
\includegraphics[width=0.3\linewidth]{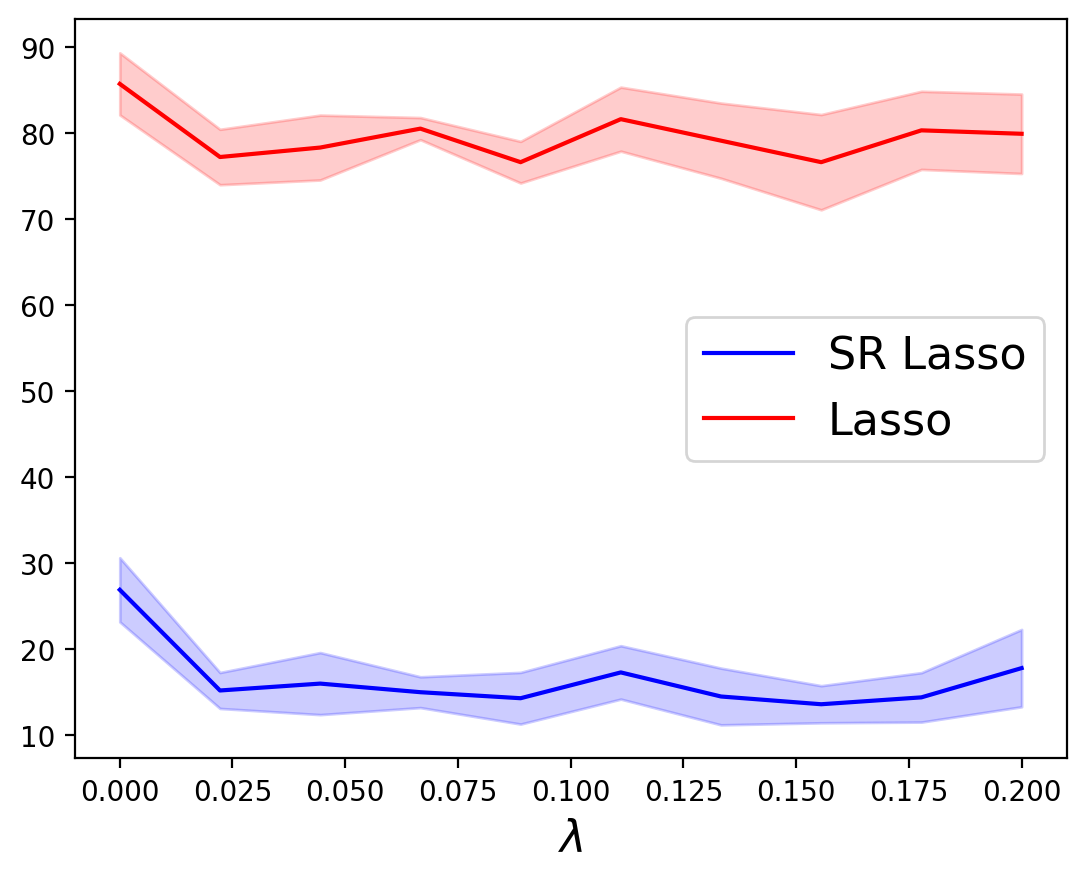}
&\includegraphics[width=0.3\linewidth]{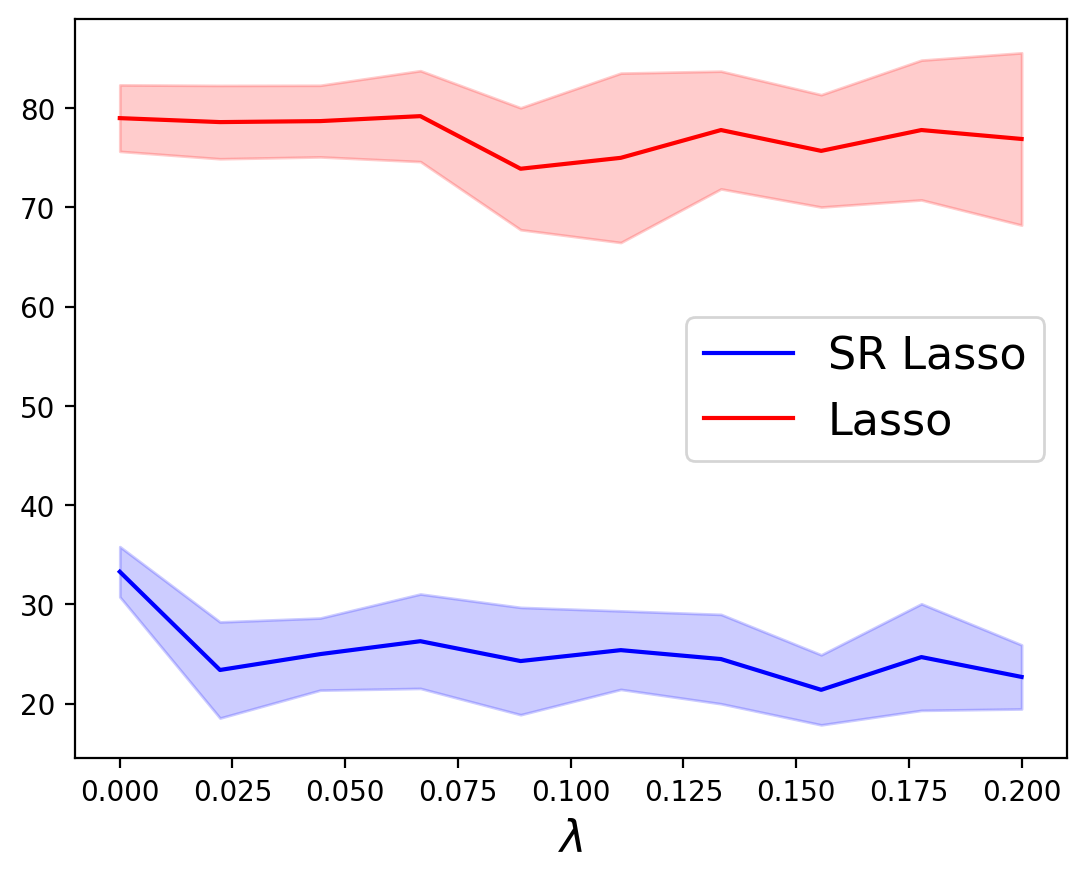}
\end{tabular}
\end{center}
\caption{3D comparison for the recovery of 3 signed spikes in the case of Gaussian-Laplace sampling. \label{fig:3d}}
\end{figure}


\section*{Conclusion}

In this paper, we have proposed a new source estimation method that relies on solving a finite-dimensional convex optimization problem. This approach allows bypassing the use of infinite-dimensional solvers by depending on a standard group lasso problem. In contrast to existing finite-dimensional alternatives with similar computational costs (such as Lasso or C-BP), it demonstrates superior sparsistency properties, which, in practice, translates into sharper source localization.

\section*{Acknowledgement} 

The work of G. Peyr\'e was supported by the European Research Council (ERC project NORIA) and the French government under management of Agence Nationale de la Recherche as part of the ``Investissements d'avenir'' program, reference ANR-19-P3IA-0001 (PRAIRIE 3IA Institute).

\appendix

\appendix

\section{Supplementary bounds for the Proof of Theorem \ref{thm_G}}
\paragraph{Bound on the coefficients}
We first provide the following bound on the coefficients $u$ and $v$, the proof is in the appendix. 
\begin{lem}\label{lem:coeffs}
We have $\Upsilon_\Ii$ is invertible and the following bounds on the coefficients $u,v$ defined in \eqref{eq:min_sep}:
$$
\norm{u-s_a}_\infty \leq  \frac{\delta_{\min}}{(1-2\delta_{\min})} \pa{\norm{s_a}_\infty  + \frac{1}{\tau}\norm{s_b}_\infty}
$$
and
$$
\norm{v - \tau^{-2} s_b}_\infty \leq  \frac{\delta_{\min}}{(1-2\delta_{\min})} \pa{\frac{1}{\tau^2} \norm{s_b}_\infty + \frac{1}{\tau}\norm{s_a}_\infty}.
$$

\end{lem}
\begin{proof}[Proof of Lemma \ref{lem:coeffs}]

Denote $\Upsilon_\Ii = \begin{pmatrix}
A & B\\
B^\top & D
\end{pmatrix}$,
where
$$
A = \pa{\Phi_X^* \Phi_X}_{\Ii,\Ii}, \quad B\eqdef  \pa{\tau \Phi_X^* \Psi_X}_{\Ii,\Ii}, \quad D\eqdef\tau^2 \pa{\Psi_X^* \Psi_X}_{\Ii,\Ii}.
$$
then
$$
v = (D^{-1}+D^{-1} B^\top S^{-1} B D^{-1} ) s_b - D^{-1} B^\top S^{-1} s_a \qandq u = S^{-1} s_a - S^{-1} B D^{-1} s_b.
$$
where $S\eqdef \pa{A - B D^{-1} B^\top}$.

We will make use of the following operator bounds:
\begin{align*}
&\norm{\Id - A}_\infty \leq \max_{i\in\Ii} \sum_{j\in \Ii\setminus\ens{i}} \abs{\kappa(x_i-x_j)} \leq \delta_{\min},\\
&\norm{B}_\infty \leq  \tau \max_{i\in\Ii} \sum_{j\in \Ii\setminus\ens{i}} \abs{\tilde \kappa'(x_i-x_j)} \leq  \tau  \delta_{\min},\\
&\norm{\tau^2 \Id - D}_\infty \leq \tau^2 \max_{i\in\Ii} \sum_{j\in \Ii\setminus\ens{i}} \abs{\tilde \kappa''(x_i-x_j)} \leq \tau^2 \delta_{\min}.
\end{align*}
So,
$$
\norm{\Id - \tau^{-2} D}_\infty \leq  \delta_{\min} \implies  \norm{D^{-1}} \leq \frac{1}{\tau^2(1-\delta_{\min})}.
$$
Moreover,
$$
\norm{\Id - S}_\infty \leq \norm{\Id-A}_\infty + \norm{B}_\infty^2 \norm{D^{-1}}_\infty \leq \delta_{\min} +\delta_{\min}^2 \frac{1}{1-\delta_{\min}} = \frac{\delta_{\min}}{1-\delta_{\min}}
$$
and hence, $\norm{S^{-1}}\leq \frac{1-\delta_{\min}}{1-2\delta_{\min}}$.
It follows that
\begin{align*}
\norm{s_a - S^{-1} s_a}_\infty \leq  \norm{ S^{-1} }_\infty\norm{s_a - S s_a}_\infty \leq \frac{\delta_{\min}}{(1-2\delta_{\min})} \norm{s_a}_\infty 
\end{align*}
and hence,
\begin{align*}
\norm{s_a - u}_\infty &\leq \frac{\delta_{\min}}{(1-2\delta_{\min})} \norm{s_a}_\infty  +\norm{S^{-1}}_\infty \norm{B}_\infty \norm{D^{-1}}_\infty \norm{s_b}_\infty\\
& \leq\frac{\delta_{\min}}{(1-2\delta_{\min})} \norm{s_a}_\infty  + \frac{\delta_{\min}}{\tau(1-2\delta_{\min})}\norm{s_b}_\infty
\end{align*}

To bound $v$,
\begin{align*}
&\norm{(D^{-1}+D^{-1} B^\top S^{-1} B D^{-1} ) s_b - \tau^{-2} s_b}_\infty \leq \tau^{-2} \norm{ D^{-1}(c \Id - D) s_b}_\infty + \norm{D^{-1} B^\top S^{-1} B D^{-1}  s_b}_\infty\\
&\leq   \frac{\delta_{\min}}{\tau^2(1-\delta_{\min})} \norm{s_b}_\infty +  \frac{\delta_{\min}^2}{\tau^2(1-2\delta_{\min})(1-\delta_{\min})} \norm{s_b}_\infty = \frac{\delta_{\min}  }{\tau^2(1-2\delta_{\min})} \norm{s_b}_\infty
\end{align*}
and 
$$
\norm{D^{-1}B^\top S^{-1} s_a }_\infty\leq \frac{\delta_{\min}}{\tau(1-2\delta_{\min})}\norm{s_a}_\infty.
$$
Therefore,
$$
\norm{v - \tau^{-2} s_b}_\infty \leq  \frac{\delta_{\min}}{\tau^2(1-2\delta_{\min})} \norm{s_b}_\infty + \frac{\delta_{\min}}{\tau(1-2\delta_{\min})}\norm{s_a}_\infty.
$$

\end{proof}

Define
$$
C_s \eqdef \norm{s_a}_\infty +\tau^{-1} \norm{s_b}_\infty.
$$
Note that as a immediate consequence of the above Lemma,
$$
\norm{u}_\infty + \tau \norm{v}_\infty \leq  \frac{1}{1-2\delta_{\min}} C_s 
$$

Since $\kappa$ is translation invariant, it suffices to consider the case of $x_1$. Throughout,  define
$$
g(x) = (s_a)_1 \kappa(x) -  \tau^{-1} s_{b,1} \tilde \kappa_1(x)$$
and
$$
G(x)\eqdef g(x)^2 + \tau^2 g'(x)^2.
$$

\begin{lem}\label{lem:approx_kernel}
Provided that $\abs{x-x_j}>\Delta_{\min}/2$ for all $j\neq 1$,
$$
\abs{\eta^{(i)}(x) - g^{(i)}(x-x_1)}\leq (2B+1)   \frac{\delta_{\min}}{1-2\Delta}\abs{\kappa''(0)}^{i/2} C_s  .
$$
Moreover,
$$
\abs{\eta(x_1) -(s_a)_1 } \leq  \frac{2\delta_{\min}}{1-2\delta_{\min}} C_s.
$$
and
\begin{align*}
\abs{\abs{\kappa''(0)}^{-1} \eta''(x_1) + (s_a)_1 } \leq \frac{2\delta_{\min}}{1-2\delta_{\min}} C_s
\end{align*}
where $B = \max_{i=0,1,2}  \norm{\tilde \kappa^{(i+1)}}_\infty$.
\end{lem}
\begin{proof}

Recall by definition of $\eta$, 
\begin{equation}\label{eq:eta_rec}
\begin{split}
\eta^{(i)}(x) &=  u_1 \kappa^{(i)}(x-x_1) -  \tau v_1 \tilde \kappa^{(i+1)}(x-x_1) +S_i(x)
\end{split}
\end{equation}
where $$
S_i(x) \eqdef \sum_{\substack{j\in \Ii\\ j\neq 1}}  u_j{\kappa^{(i)}(x-x_j)} - \tau  \sum_{\substack{j\in \Ii\\ j\neq 1}}  v_j {\tilde \kappa^{(i+1)}(x-x_j)}.
$$

We can focus on the first two terms in each of the expressions above, since  $S_i(x)$ can be made arbitrarily small by sufficient separation. We will use the bounds: if $\abs{x-x_j}\geq \Delta_{\min}/2$ for $j\neq 1$, then
\begin{equation}\label{eq:S}
\begin{split}
\abs{S_i(x)} &\leq \abs{\kappa''(0)}^{i/2} (\norm{u}_\infty + \tau \norm{v}_\infty )\delta_{\min} \leq \frac{\delta_{\min}}{1-2\delta_{\min}}\abs{\kappa''(0)}^{i/2} C_s  
\end{split}.
\end{equation}

By combining  \eqref{eq:S} with Lemma \ref{lem:coeffs},
\begin{align*}
&\abs{g^{(i)}(x-x_1) - u_1 \kappa^{(i)}(x-x_1) + \tau v_1 \tilde \kappa^{(i+1)}(x-x_1)}
\leq \abs{s_{a,1}-u_1} \norm{\kappa^{(i)}}_\infty + \tau \abs{\tau^{-2} s_{b,1}-v_1} \norm{\tilde \kappa^{(i+1)}}_\infty\\
&\leq  \frac{2\delta_{\min}}{(1-2\delta_{\min} ) } C_s B \abs{\kappa''(0)}^{i/2} 
\end{align*}

In the case of $x=x_1$ and $i=1$, since $\kappa'(0) = 0$ and $\kappa(0) = 1$, 
$$
\abs{\eta(x_1) -(s_a)_1 } \leq \abs{(s_a)_1- u_1} + \frac{\delta_{\min}}{1-2\delta_{\min}} C_s\leq  \frac{2\Delta}{1-2\Delta} C_s.
$$

\begin{align*}
\abs{\abs{\kappa''(0)}^{-1} \eta''(x_1) + (s_a)_1 } &\leq \abs{(s_a)_1 - u_1} + \abs{\kappa''(0)^{-1} \eta''(0) + u_1 }
\\
&= \abs{(s_a)_1 - u_1} + \tau^2 \abs{ \kappa''(0)^{-1} S_2(0)}\\
&\leq \frac{2\delta_{\min}}{1-2\delta_{\min}} C_s
\end{align*}

\end{proof}

\begin{lem}\label{lem:close}[Bound on the gradient]

$$
f_0'(x_1) =2 \frac{(s_b)_1}{\tau} \pa{  (s_a)_1 + \tau^2 \eta''(0) },
$$
Hence,
$$
\abs{f_0'(x_1)}  \leq \frac{2\abs{(s_b)_1}}{\tau} \pa{
 (1-\tau^2) \abs{(s_a)_1}  +  \frac{2\tau^2 \delta_{\min} }{1-2\delta_{\min}} C_s  } = \abs{G'(0)} + \Oo(\abs{\kappa''(0)}^{\frac12}\delta_{\min})
$$

\end{lem}
\begin{proof}
We have 
\begin{align*}
\frac12 f_0'(x_1) &=  \eta'(x_1)  \eta(x_1) +\frac{\tau^2}{\abs{\kappa''(0)}} \eta''(x_1) \eta'(x_1)  = \frac{\abs{\kappa''(0)}^{\frac12}}{\tau} (s_b)_1 (s_a)_1 +\frac{ \tau}{\abs{\kappa''(0)}^{\frac12}} \eta''(x_1) (s_b)_1\\
&= \frac{(s_b)_1 \abs{\kappa''(0)}^{\frac12}}{\tau} \pa{  (s_a)_1 + \frac{\tau^2}{\abs{\kappa''(0)}} \eta''(x_1) },
\end{align*}
By Lemma \ref{lem:approx_kernel},
$$
\abs{(s_a)_1 + \frac{\tau^2}{\abs{\kappa''(0)}} \eta''(x_1)} \leq  (1-\tau^2) \abs{(s_a)_1}  +  \frac{2\tau^2\delta_{\min}}{1-2\delta_{\min}} C_s.
$$


\end{proof}

\begin{lem}[Bound on the second derivative]\label{lem:hessian}
$$
 f_0''(x) \leq G''(x) +  2\abs{\kappa''(0)} W \delta_{\min}
$$
where
$$
W\eqdef \frac{2}{1-2\delta_{\min}} C_s (1+\tau^2) (1+2B)\pa{(1+2B)  C_s \frac{\delta_{\min}}{1-2\delta_{\min}} + 2BC }.
$$

\end{lem}
\begin{proof}
By direct computation,
\begin{align*}
f_0''(x) &= 2 \eta'(x)^2 + 2\eta''(x) \eta(x) + 2\tau^2  \eta''(x)^2 + 2\tau^2 \eta'(x) \eta'''(x)\\
&= 2 \eta'(x) (\eta'(x) + \tau^2 \eta'''(x))+ 2\eta''(x) (\eta(x) + \tau^2  \eta''(x)) 
\end{align*}
Let $p_\delta\eqdef  \frac{\delta_{\min}}{1-2\delta_{\min}} C_s$.
From Lemma \ref{lem:approx_kernel},
$$
\abs{\eta^{(i)}(x) -g^{(i)}(x) } \leq (1+2 B ) \abs{\kappa''(0)}^{i/2}p_\delta  
$$
and
$$
\abs{g^{(i)}(x)}\leq  \abs{\kappa''(0)}^{i/2}B  ( \norm{s_{a}}_\infty  + \norm{s_b}_\infty  \tau^{-1} )\eqdef C  \abs{\kappa''(0)}^{i/2}B 
$$
So,
\begin{align*}
\frac12 f_0''(x) &\leq g'(x)^2 + g(x) g''(x) + \tau^2 g''(x)^2 +\tau^2 g'(x) g'''(x)\\
&+ (2+2\tau^2)(1+2B)^2\abs{\kappa''(0)} p_\delta^2 + 2(2+2\tau^2)(1+2B)BC\abs{\kappa''(0)} p_\delta
\end{align*}

\end{proof}

\paragraph{``Far" behaviour}
We now use the fact that the kernels $K^{(i)}$ have sufficient decay to conclude that $\abs{f_0(x)} <1$ for all $x$ such that $\abs{x}>r$ and $\abs{x-x_i}>\Delta_{\min}/2$ for $i\neq 1$. 
\begin{lem}\label{lem:far}

Let $\mu>0$ be such that for all $\abs{x}>r$,
$$
\abs{G(x)} \leq 1-\mu \qandq \delta_{\min}\lesssim \mu.
$$
Then,
$$
\abs{f_0(x)}\leq 1-\mu/2
$$
for all $\abs{x}>r$ and $\abs{x-x_j}\geq \frac{\Delta_{\min}}{2}$.
\end{lem}
\begin{proof}Let  $p_\delta\eqdef  C_s \frac{\delta}{1-2\delta}$.
By Lemma \ref{lem:approx_kernel},
$$
\abs{\eta^{(i)}(x) - g^{(i)}(x)} \leq   (2B+1) \frac{\delta}{1-2\delta} \abs{\kappa''(0)}^{i/2} C_s
$$
and
$$
\abs{g^{(i)}(x)} \leq C_s \abs{\kappa''(0)}^{i/2}.
$$
So,
\begin{align*}
\abs{\kappa''(0)}^{-i} \eta^{(i)}(x) ^2 &\leq \abs{\kappa''(0)}^{-i} g^{(i)}(x)^2 +  \pa{ (2B+1) p_\delta }^2 + 2\abs{\kappa''(0)}^{-i/2}  \abs{g^{(i)}(x)} \pa{ (2B+1) p_\delta  }\\
&\leq\abs{\kappa''(0)}^{-i}  g^{(i)}(x)^2 +  \pa{ (2B+1) p_\delta }^2 + 2 C_s (2B+1) p_\delta 
\end{align*}
So,
$$
\abs{f_0(x)} \leq G(x) + \Oo(\delta) \leq 1-\mu/2.
$$

\end{proof}

\bibliographystyle{plain}
\bibliography{biblio}

\end{document}